\documentclass[11pt,reqno]{amsart}

\usepackage[T1]{fontenc}
\usepackage{lmodern}
\usepackage{amsmath,amssymb,amsthm,mathtools}
\usepackage{microtype}
\usepackage{xcolor}
\usepackage[margin=1.05in]{geometry}
\usepackage[colorlinks=true,hypertexnames=false,linkcolor=blue!50!black,citecolor=blue!50!black,urlcolor=blue!50!black]{hyperref}
\usepackage[numbers, sort&compress]{natbib}

\newcommand{\EE}{\mathbb E}
\newcommand{\PP}{\mathbb P}
\newcommand{\R}{\mathbb R}

\newcommand{\N}{\mathbb N}
\newcommand{\cA}{\mathcal{A}}
\newcommand{\cB}{\mathcal{B}}

\newcommand{\cH}{\mathcal{H}}
\newcommand{\cL}{\mathcal{L}}
\newcommand{\cP}{\mathcal{P}}
\newcommand{\cM}{\mathcal{M}}
\newcommand{\fG}{\mathfrak{G}}
\newcommand{\1}{\mathbf 1}
\newcommand{\pa}[1]{\left(#1\right)}
\newcommand{\bra}[1]{\left[#1\right]}
\newcommand{\abs}[1]{\left|#1\right|}

\newcommand{\norm}[1]{\left\lVert#1\right\rVert}
\newcommand{\supp}{\operatorname{supp}}
\newcommand{\ip}[2]{\left\langle #1,#2\right\rangle}

\newtheorem{theorem}{Theorem}[section]
\newtheorem{proposition}[theorem]{Proposition}
\newtheorem{definition}[theorem]{Definition}
\newtheorem{lemma}[theorem]{Lemma}
\newtheorem{example}[theorem]{Example}
\newtheorem{corollary}[theorem]{Corollary}
\newtheorem{remark}[theorem]{Remark}

\title[On rigidity of infinite exchangeable sequences]{On rigidity of infinite exchangeable sequences}
\author{Maximillian Newman}
\address{Section of Genetic Medicine, University of Chicago, Chicago, Illinois, USA}
\email{mnewman98@uchicago.edu}
\thanks{Correspondence: \texttt{mnewman98@uchicago.edu}.}
\date{\today}


\begin{document}

\begin{abstract}
    We ask when the law of an infinite
    exchangeable sequence is uniquely determined by a finite-dimensional marginal. For mixtures of a
    probability kernel, a weak localization condition makes
    three coordinates determine the directing law. Completeness alone
    makes four coordinates sufficient. Neither order can be lowered in
    general. For exponential kernels with a nonnegative sufficient statistic, a
    M\"untz--Sz\'asz support condition implies three coordinates suffice. In the
    scalar case, this result is sharp. We also
    characterize when completeness holds simultaneously for all scalar mixing laws by a weighted
    M\"untz density condition and show that statistic support alone need
    not decide completeness. These methods yield sharp finite-dimensional refinements of Freedman's theorem on rotatable sequences, namely that spherical symmetry of three coordinates makes an infinite exchangeable sequence
    rotatable while symmetry of the first two coordinates is in general insufficient. Finally,
    for energy-radial densities we give a complete-monotonicity criterion
    for infinite extendibility and determine the exact uniqueness
    thresholds for $\ell_q$-norm exponential laws.
\end{abstract}

\maketitle

\section{Introduction}

    In this manuscript we consider an infinite sequence $X = \pa{X_i}_{i=1}^\infty$ of random variables taking values in some Polish space $E$. Such a sequence is called \textit{exchangeable} when the distribution of any finite block $(X_1, \ldots, X_k)$ is invariant under relabeling of its indices. While perhaps the simplest relaxation of the classical iid assumption for sequences of random variables, exchangeable sequences exhibit remarkable rigidity properties, most famously the de Finetti-Hewitt-Savage Theorem \cite[Theorem 1.1]{kallenberg05}. This theorem states that an infinite sequence of random variables is exchangeable if and only if there exists a random measure $\eta$, known as the \textit{directing measure}, taking values in the space of probability measures $\cP(E)$ on $E$ under weak convergence so that $X$ is given by taking iid samples from $\eta$. That is, the law of $X$ conditional on $\eta$ is iid, i.e.
    \begin{equation*}
        \cL(X \mid \eta) = \eta^\infty.
    \end{equation*}
    Equivalently, $X$ is uniquely characterized by the law $\Lambda = \cL(\eta) \in \cP(\cP(E))$ of its directing measure $\eta$. $\Lambda$ is itself uniquely characterized by the infinite collection of moment measures
    \begin{equation*}
        \cM_k(\Lambda) := \int_{\cP(E)} \mu^{\otimes k} d\Lambda(\mu) = \EE[\eta^{\otimes k}].
    \end{equation*}
    We ask when $X$ is uniquely determined by $\cM_k(\Lambda)$ for some $k < \infty$. We call this property $k$-rigidity when $k$ coordinates suffice to uniquely determine $\Lambda$.

    $k$-rigidity is a natural extension of classical notions of uniqueness in the literature on mixture identifiability. Classical mixture identifiability asks when a mixture of probability kernels, such as $\Lambda = \frac{1}{2}\delta_{\gamma_1} + \frac{1}{2}\delta_{\gamma_2}$ for $\gamma_v$ the law of a centered Gaussian with variance $v$, can be uniquely identified from $\cM_1(\Lambda)$ among finite mixtures of a prescribed class \cite{TeicherFinite, YakowitzSpragins}. In the
    grouped-sample literature
    \cite{VandermeulenScott,VandermeulenSaitenmacher}, an $m$-component target $\Lambda$ is \emph{$k$-identifiable} if its $k$-observation law identifies it among mixtures with at most $m$ components, and \emph{$k$-determined} if uniqueness holds among all finite mixtures. Competing components are arbitrary in both definitions. $k$-rigidity asks for uniqueness among all directing laws, whether finite mixtures or not. For a finite target and $k\geq2$, it is equivalent to $k$-determinedness by Lemma~\ref{L: finite_competitor_reduction}. In this sense the results below extend finite-mixture determinedness to general mixing laws. 

    3-identifiability has been found in finite-mixture and nonparametric latent-variable models. Kruskal's tensor uniqueness theorem \cite{kruskal77} underlies the finite-mixture identification results of Allman, Matias, and Rhodes \cite{allman09}. Hall and Zhou \cite{hallzhou03} study nonparametric component identification in multivariate mixtures, while Hu and Schennach \cite{huschennach09} identify continuous latent-variable models through an operator factorization under assumptions including injectivity, separation, and normalization. These results restrict the competing number of components or the admissible model class. Here competing directing laws are unrestricted, and their components need not belong to the target kernel family. For an $m$-component target, completeness gives identification from three observations among competitors with at most $m$ atoms (Proposition~\ref{P: finite_independent_bounds}(a)), but Remark~\ref{R: completeness_sharpness} shows that this conclusion can fail against arbitrary directing laws. Weak localization supplies the additional hypothesis in Theorem~\ref{T: general_rigidity} to obtain the stronger 3-rigidity.
    
    Recent work \cite{newman26} demonstrated that joint Gaussianity of four coordinates forces an infinite exchangeable sequence to be a Gaussian process. It was conjectured in \cite{newman26} and proved in \cite{zhengzhou26} that joint Gaussianity of two coordinates suffices. In this manuscript, we establish $3$-rigidity for a broad class of mixture kernels under a suitable localization hypothesis and $4$-rigidity under a weaker completeness hypothesis. Applying these results to Gaussian scale families, consisting of a conditionally iid sequence of centered Gaussian random variables with a random variance, we obtain a sharp finite-dimensional refinement of Freedman's Theorem \cite{freedman62}; see also \cite[Theorem~1.31]{kallenberg05}. These scale mixtures need not have jointly Gaussian finite-dimensional laws. Freedman's Theorem shows that if an infinite exchangeable sequence is rotatable, i.e. spherically symmetric in any finite dimension, then it must be a Gaussian scale family. We prove this conclusion under the weaker hypothesis that just $(X_1, X_2, X_3)$ be spherically symmetric. The three-coordinate hypothesis cannot be reduced to two in general, as we construct counterexamples when only $(X_1, X_2)$ is spherically symmetric.

    The principal applications of our localization methods are to exponential mixtures. So long as the essential support of the energy functional of the exponential family is sufficiently rich, e.g. containing a copy of $\N$, our localization hypothesis holds and we obtain $3$-rigidity. This richness condition is described in terms of a M\"untz--Sz\'asz type condition which covers Gaussian scale, gamma, Weibull, Pareto, beta, negative binomial, and bounded-mean Poisson families, as well as finite products with arbitrarily dependent parameters of these same families. Moreover, we demonstrate the sharpness of $3$-rigidity for any non-point mixing law in one parameter under this richness condition. We also provide a weighted M\"untz--Sz\'asz criterion characterizing completeness simultaneously for every scalar mixing law.

    \subsection{Moment measures and unrestricted rigidity}\label{SS: moment_measures_intro}

        Fix a probability space $\pa{\Omega,\Sigma,\PP}$ and a Polish space $E$, and equip $\cP(E)$, the space of Borel probability measures on $E$, with the topology of weak convergence. The directing measure $\eta$ of an infinite exchangeable sequence $X$ taking values in $E$ is a $\cP(E)$-valued random variable such that the conditional law of $X$ given $\eta$ is $\eta^\infty$. We study $X$ through its directing law $\Lambda:=\cL(\eta)\in\cP(\cP(E))$.

        For later use with signed perturbations, we also define moment measures of signed directing measures. For any finite signed Borel measure $\Lambda$ on $\cP(E)$ and any non-negative integer $k$, we define the $k$th moment measure $\cM_k(\Lambda)$ by
        \begin{equation*}
            \cM_k(\Lambda) := \int_{\cP(E)} \mu^{\otimes k} d\Lambda(\mu),
        \end{equation*}
        where $\mu^{\otimes 0} = 1$.
        For a directing law, $\cM_k(\Lambda)$ is the law of the first
        $k$ coordinates. We seek directing laws determined by one such
        moment measure.

        \begin{definition}[$k$-rigidity]\label{D: k_rigidity}
             We call a probability measure $\Lambda$ on $\cP(E)$ $k$-rigid when it is the unique such probability measure that has its $k$th moment measure equal to $\cM_k(\Lambda)$. Equivalently, $\Lambda$ is $k$-rigid if $\cM_k(\Lambda) = \cM_k(\Gamma)$ and $\Gamma \in \cP(\cP(E))$ implies $\Lambda = \Gamma$.
        \end{definition}
        \begin{remark}
            By the de Finetti-Hewitt-Savage Theorem \cite[Theorem 1.1]{kallenberg05}, $k$-rigidity is equivalent to uniqueness of the infinite exchangeable extension of the corresponding $k$-coordinate law.
            Note also that $k$-rigidity of $\Lambda$ implies $k'$-rigidity of $\Lambda$ for all $k' \geq k$. An important consequence of $k$-rigidity is that it is easy to check convergence of measures to a $k$-rigid law. To be precise, suppose $\Lambda$ is $k$-rigid. Then a sequence $\Lambda_n$ of elements of $\cP(\cP(E))$ converges weakly to $\Lambda$ if and only if $\cM_k(\Lambda_n)$ converges weakly to $\cM_k(\Lambda)$.

            To justify this assertion, continuity of $\mu\mapsto\mu^{\otimes k}$
            gives the forward implication. For the converse, assume $k\geq1$.
            Marginalization gives tightness of $\cM_1(\Lambda_n)$. For
            $\varepsilon>0$, choose compact sets $K_j\subset E$ such that
            $\sup_n\cM_1(\Lambda_n)(K_j^c)\leq\varepsilon2^{-j}/j$.
            Markov's inequality and a union bound show that
            $\Lambda_n(\mathcal K_\varepsilon)\geq1-\varepsilon$ for every $n$,
            where
            \begin{equation*}
                \mathcal K_\varepsilon
                :=\{\mu\in\cP(E):\mu(K_j)\geq1-1/j\text{ for every }j\geq1\}.
            \end{equation*}
            This is a closed tight subset of $\cP(E)$, hence compact by
            Prokhorov's theorem. Thus $(\Lambda_n)$ is tight. Every
            subsequential limit $\Gamma$ satisfies
            $\cM_k(\Gamma)=\cM_k(\Lambda)$ by continuity, and $k$-rigidity
            gives $\Gamma=\Lambda$. Therefore $\Lambda_n$ converges weakly
            to $\Lambda$. The case $k=0$ is immediate, since $0$-rigidity
            forces $\cP(E)$ to be a singleton.

        \end{remark}

        Let $Y$ be another Polish space. By $P: Y \to \cP(E)$ we denote a measurable probability kernel from $Y$ to $E$ \cite[p. 55]{KallenbergFoundations}. We fix a \textit{mixing measure} $\nu\in\cP(Y)$ and define the \textit{mixture} $\fG_\nu = \fG_\nu(P)$, an element of $\cP(\cP(E))$,  by
        \begin{equation*}
            \fG_\nu := \int_Y \delta_{P_y} d\nu(y).
        \end{equation*}
        To each $f \in L^2(\nu)$ we may associate a finite signed measure $Q_\nu[f]$ on $E$ given by mixing the kernel against $f$, i.e. for every Borel set $C\in\cB(E)$ it assigns the mass
        \begin{equation*}
            Q_\nu[f](C) := \int_Y f(y) P_y(C) d\nu(y).
        \end{equation*}

        \begin{definition}[Completeness property]
            We say that $(\nu,P)$ satisfies the completeness property if
            \begin{equation}\label{E: completeness}\tag{C}
                Q_\nu[f] = 0
                \quad
                \text{ implies }
                \quad
                f = 0 \quad \nu\text{-almost everywhere}
            \end{equation}
            for every real $f\in L^2(\nu)$.
            Equivalent forms are given in Lemma~\ref{L: Kh_dense}.
        \end{definition}
        \begin{remark}
            The completeness condition~\eqref{E: completeness} is a form of conditional
            completeness familiar from nonparametric identification.
            Newey and Powell~\cite{NeweyPowell2003} use conditional
            completeness to characterize identification in instrumental
            variable models; the precise $L^2$ formulation relevant here
            is discussed by Andrews~\cite[Section~2]{Andrews2017}.
            To make the connection explicit, let $(U,V)$ have joint law
            $\nu(du)P_u(dv)$ on an auxiliary probability space.
            Lemma~\ref{L: Kh_dense} identifies \eqref{E: completeness}
            with conditional $L^2$ completeness of $U$ given $V$.
            Thus completeness concerns the conditional law of the parameter
            $U$ given the observation $V$, and depends on the mixing law
            $\nu$. 
            
            This differs from completeness of the statistical family $(P_y)$
            in the usual sense; see
            \cite[Section~1]{KaganMalinovskyMattner2015}.
            That condition tests functions on the observation space $E$
            through their expectations under every $P_y$, whereas
            \eqref{E: completeness} tests functions on the parameter space
            $Y$ relative to the fixed mixing law $\nu$.
        \end{remark}

        For finite signed Borel measures $\mu$ and $\mu'$ on $E$, we
        write $\mu\leq\mu'$ if $\mu'-\mu$ is a positive measure, or
        equivalently if $\mu(C)\leq\mu'(C)$ for every $C\in\cB(E)$.
        Thus $0\leq Q_\nu[f]\leq Q_\nu[\1_B]$ means
        \begin{equation*}
            0\leq\int_Y f(y)P_y(C)\,d\nu(y)
                \leq\int_B P_y(C)\,d\nu(y)
            \quad\text{for every }C\in\cB(E).
        \end{equation*}
        Our main hypothesis extracts information about $f$ from this
        domination of its signed mixture.
        \begin{definition}[Strong and weak localization]
            We say that $(\nu, P)$ is weakly localized if there is a countable $\pi$-system $\cA \subset \cB(Y)$ generating $\cB(Y)$ such that, for every $B \in \cA$ and every real $f \in L^2(\nu)$ we have 
            \begin{equation}\label{E: weak_localization}\tag{WL}
                0 \leq Q_{\nu}[f] \leq Q_\nu[\mathbf{1}_B]
                \quad
                \text{ implies }
                \quad
                \int_{Y \setminus B} f(y) d\nu(y) = 0.
            \end{equation}
            We say $(\nu, P)$ is strongly localized if there is such an $\cA$ so that
            \begin{equation}\label{E: strong_localization}\tag{SL}
                0 \leq Q_\nu[f] \leq Q_\nu[\mathbf{1}_B] 
                \quad 
                \text{ implies } 
                \quad 
                f \mathbf{1}_{Y \setminus B} = 0 \quad \nu\text{-a.e.}
            \end{equation}
            for all $B \in \cA$ and $f$ in $L^2(\nu)$.
        \end{definition}

        \begin{remark}\label{R: hierarchy}
            We have the following hierarchy of regularity assumptions on $(\nu, P)$:
            \begin{equation*}
                \eqref{E: strong_localization} \implies \eqref{E: weak_localization} \implies\eqref{E: completeness}.
            \end{equation*}
            This is proved in Lemma~\ref{L: localization_completeness}.
        \end{remark}

        \begin{theorem}\label{T: general_rigidity}
            Let $E$ and $Y$ be Polish spaces, $P$ be a measurable probability kernel from $Y$ to $E$, and $\nu$ be an element of $\cP(Y)$. If \eqref{E: weak_localization} holds, then $\fG_\nu$ is $3$-rigid.
        \end{theorem}
        \begin{remark}
            More explicitly, Theorem~\ref{T: general_rigidity} says that if a random probability measure $\eta$ satisfies
            \begin{equation}\label{E: three_rigidity}
                \EE[\eta^{\otimes 3}] = \int_Y P_y^{\otimes 3} d\nu(y),
            \end{equation}
            then, on the probability space carrying $\eta$, there is a measurable $Y$-valued random variable $A$ with law $\nu$ such that $\eta = P_A$ almost surely.

            In general, $3$ cannot be improved to $2$ here. See Proposition~\ref{P: exponential_two_nonrigidity} for an example.
        \end{remark}
        Exponential mixtures give the principal applications in
        Subsection~\ref{SS: exponential_mixtures_intro}. Uniform laws
        with a varying upper endpoint give a non-exponential example
        (Example~\ref{Ex: uniform_localization}).

        The completeness property \eqref{E: completeness} alone ensures rigidity at order four.

        \begin{theorem}\label{T: completeness_rigidity}
             Let $E$ and $Y$ be Polish spaces, $P$ be a measurable probability kernel from $Y$ to $E$, and $\nu$ be an element of $\cP(Y)$. If $(\nu, P)$ satisfy the completeness property \eqref{E: completeness}, then $\fG_\nu$ is $4$-rigid.
        \end{theorem}

        \begin{remark}
            In particular, Theorem~\ref{T: completeness_rigidity}
            recovers the four-coordinate Gaussian rigidity theorem
            \cite[Theorem~1.4]{newman26}; see
            Corollary~\ref{C: translation_gaussian_rigidity} for the proof. More generally, Corollary~\ref{C: translation_gaussian_rigidity} provides a sufficient condition for completeness of translation kernels.
        \end{remark}
        
        \begin{remark}\label{R: completeness_sharpness}
            The number four in
            Theorem~\ref{T: completeness_rigidity} cannot in general
            be replaced by three. Let $E=\{0,1\}$, $Y=[0,1]$, and
            \begin{equation*}
                P_y=(1-y)\delta_0+y\delta_1.
            \end{equation*}
            Define
            \begin{equation*}
                p_\pm=\frac{1}{2}\pa{1\pm\frac{1}{\sqrt{3}}},
                \quad
                \nu=\frac{1}{2}\delta_{p_-}
                     +\frac{1}{2}\delta_{p_+},
            \end{equation*}
            and let $\widetilde\nu$ be the uniform probability measure
            on $[0,1]$. The component measures $P_{p_-}$ and $P_{p_+}$
            are linearly independent, so $(\nu,P)$ satisfies
            \eqref{E: completeness}. However,
            \begin{equation*}
                \int_0^1 y^j\,d\nu(y)
                =\int_0^1 y^j\,d\widetilde\nu(y)
                =\frac{1}{j+1},
                \quad j=0,1,2,3.
            \end{equation*}
            Each mass assigned by $P_y^{\otimes3}$ to a point of
            $\{0,1\}^3$ is a polynomial in $y$ of degree at most three.
            Consequently,
            \begin{equation*}
                \cM_3\pa{\fG_\nu}
                =\cM_3\pa{\fG_{\widetilde\nu}},
                \quad
                \text{ yet }
                \quad
                \fG_\nu\ne\fG_{\widetilde\nu}.
            \end{equation*}
            Thus $\fG_\nu$ is $4$-rigid but not $3$-rigid.
        \end{remark}

    \subsection{Rigidity of exponential mixtures}\label{SS: exponential_mixtures_intro}

        A rich class of hierarchical models that exhibit the rigidity properties described in Theorem~\ref{T: general_rigidity} are exponential mixtures. Let $m$ be a $\sigma$-finite measure on $E$, $T: E \to [0,\infty)^d$ a Borel measurable map, and $Y$ a Polish subset under the subspace topology of $(0,\infty)^d$. Let $\rho:=T_*m$ be the pushforward of $m$ by $T$.
        Throughout the exponential-mixture arguments,
        $\ip{y}{t}:=\sum_{i=1}^d y_it_i$ denotes the Euclidean inner
        product on $\R^d$.

        Assume that
        \begin{equation}\label{E: finite_partition_function}
            0 < Z(y) := \int_E \exp\pa{-\ip{y}{T(x)}}dm(x) = \int_{[0,\infty)^d} \exp\pa{-\ip{y}{t}}d\rho(t) < \infty\
            \quad
            \text{for all $y \in Y$.}
        \end{equation}
        Then we may define the exponential kernel $P_y$ as the measure absolutely continuous with respect to $m$ with density
        \begin{equation}\label{E: exponential_kernel}
            \frac{dP_y}{dm}(x) := \frac{\exp\pa{-\ip{y}{T(x)}}}{Z(y)}.
        \end{equation}
        This is a measurable probability kernel. For any $\nu \in \cP(Y)$ its directing law is the exponential mixture
        \begin{equation*}
            \fG_\nu = \int_Y \delta_{P_y} d\nu(y)
        \end{equation*}

        For exponential kernels, the density of a signed mixture is given by
        \begin{equation*}
            \frac{dQ_\nu[f]}{dm}(x)
            =(\mathcal L_\nu f)(T(x)),
        \end{equation*}
        where
        \begin{equation}\label{E: exponential_transform}
            (\mathcal L_\nu f)(t)
            :=\int_Y\frac{f(y)}{Z(y)}
                        e^{-\langle y,t\rangle}\,d\nu(y),
        \end{equation}
        the integral is absolutely convergent for $\rho$-almost every
        $t$, and the density identity holds $m$-almost everywhere.
        Thus $0\leq Q_\nu[f]\leq Q_\nu[\mathbf1_B]$ gives
        $0\leq\mathcal L_\nu f\leq\mathcal L_\nu\mathbf1_B$
        $\rho$-almost everywhere. Strong localization requires these
        inequalities to force $f$ to vanish $\nu$-almost everywhere outside $B$.
        
        For one parameter and $B=Y\cap[b,\infty)$, the desired conclusion
        excludes all contributions from parameters below $b$.
        A finite signed measure $\sigma$ concentrated on $[b,\infty)$
        satisfies
        \begin{equation*}
            \left|\int_{[0,\infty)}e^{-\lambda y}\,\sigma(dy)\right|
            \leq\|\sigma\|_{\mathrm{TV}}e^{-b\lambda}.
        \end{equation*}
        The relevant converse asks whether such bounds on the Laplace
        transform of $\sigma$ force $\sigma|_{[0,b)}=0$.
        The following proposition characterizes the uniformly separated exponent
        sequences for which this converse holds.

        \begin{proposition}\label{L: laplace_support}
            Let $0=\lambda_0<\lambda_1<\cdots$ satisfy
            \begin{equation*}
                \inf_{n\geq0}(\lambda_{n+1}-\lambda_n)>0.
            \end{equation*}
            For a finite signed Borel measure $\sigma$ on $[0,\infty)$,
            we define the Laplace transform
            \begin{equation*}
                \widehat\sigma(\lambda)
                :=\int_{[0,\infty)}e^{-\lambda y}\,d\sigma(y).
            \end{equation*}
            The following are equivalent:
            \begin{enumerate}
                \item $\displaystyle\sum_{n=1}^{\infty}
                            \frac{1}{\lambda_n}=\infty$.
                \item For every finite signed Borel measure $\sigma$
                    on $[0,\infty)$ and every $b\geq0$,
                    \begin{equation}\label{E: laplace_support}
                        \sigma|_{[0,b)}=0
                        \quad\Longleftrightarrow\quad
                        \sup_{n\geq1}e^{b\lambda_n}
                            |\widehat\sigma(\lambda_n)|<\infty.
                    \end{equation}
                \item For every nonzero finite signed Borel measure
                    $\sigma$ on $[0,\infty)$,
                    \begin{equation*}
                        \inf\operatorname{supp}|\sigma|
                        =-\limsup_{n\to\infty}
                            \frac{1}{\lambda_n}
                            \log|\widehat\sigma(\lambda_n)|,
                    \end{equation*}
                    where $\log0:=-\infty$.
            \end{enumerate}
        \end{proposition}
        \begin{remark}
            Under the substitution $u=e^{-y}$, the quantities $\widehat\sigma(\lambda_n)$ become generalized moments. The divergence condition in Proposition~\ref{L: laplace_support} is the classical M\"untz--Sz\'asz condition for density of $\operatorname{span}\{1,u^{\lambda_n}:n\geq1\}$ in $C([0,1])$ \cite{BorweinErdelyi1996}. The proposition additionally recovers a support restriction from exponential bounds on these moments.
        \end{remark}
        
        For the exponential kernel, we impose the corresponding condition
        in each coordinate.
        
        \begin{definition}[M\"untz--Sz\'asz condition]
            \label{D: exponential_MS}
            We say that $\rho$ satisfies the
            M\"untz--Sz\'asz condition if there exist
            $t_0\in[0,\infty)^d$ and a sequence
            $0=\lambda_{i,0}<\lambda_{i,1}<\cdots$ for every $1\leq i\leq d$
            such that
            \begin{equation}\label{E: exponential_MS}\tag{MS}
                \begin{aligned}
                    &\inf_{n\geq0}
                        (\lambda_{i,n+1}-\lambda_{i,n})>0,
                    \quad
                    \sum_{n=1}^{\infty}\frac{1}{\lambda_{i,n}}=\infty,
                    \quad 1\leq i\leq d,
                    \quad \text{and}\\
                    &t_0+
                    \bigl\{
                        (\lambda_{1,n_1},\ldots,\lambda_{d,n_d}):
                        n_1,\ldots,n_d\geq0
                    \bigr\}
                    \subseteq\operatorname{supp}\rho.
                \end{aligned}
            \end{equation}
        \end{definition}

        The support proposition converts \eqref{E: exponential_MS} into
        strong localization and hence rigidity.
        
        \begin{corollary}\label{C: exponential_rigidity}
            Suppose that \eqref{E: finite_partition_function} and
            \eqref{E: exponential_MS} hold. Then, for every
            $\nu\in\cP(Y)$, the exponential kernel
            \eqref{E: exponential_kernel} satisfies strong localization
            \eqref{E: strong_localization}.
            In particular, $\fG_\nu$ is $3$-rigid.
        \end{corollary}

        Moreover, Proposition~\ref{P: exponential_two_nonrigidity} demonstrates Corollary~\ref{C: exponential_rigidity} is sharp in dimension $d = 1$.

        \begin{proposition}\label{P: exponential_two_nonrigidity}
            Suppose that $d=1$, \eqref{E: finite_partition_function}
            holds, and $T$ is not $m$-almost everywhere constant.
            If $\nu\in\cP(Y)$ is not a point mass, then
            $\mathfrak G_\nu$ is not $2$-rigid.
        \end{proposition}

        The corollary applies to arbitrary mixing laws over the scalar
        families in Table~\ref{Tab: exponential_families} and their
        finite products. For example, a triple sampled conditionally
        iid from a gamma law of fixed shape and random rate determines
        the entire exchangeable sequence, even among unrestricted
        competitors. For every non-point scalar mixing law, order three
        is exact by Proposition~\ref{P: exponential_two_nonrigidity}.

        \begin{table}[!htbp]
            \centering
            \small
            \setlength{\tabcolsep}{3pt}
            \renewcommand{\arraystretch}{1.8}
            \begin{tabular}{@{}p{0.13\textwidth}p{0.17\textwidth}
                p{0.12\textwidth}p{0.14\textwidth}
                p{0.19\textwidth}p{0.16\textwidth}@{}}
                \hline
                Family & Reference measure & $T$ & $Z(y)$
                    & $\rho=T_*m$ & Usual parameter\\
                \hline
                Gaussian
                    & $dx$
                    & $(x-a)^2$
                    & $\sqrt{\frac{\pi}{y}}$
                    & $t^{-\frac{1}{2}}\,dt$
                    & variance $v=\frac{1}{2y}$\\
                Gamma
                    & $x^{a-1}\,dx$
                    & $x$
                    & $\Gamma(a)y^{-a}$
                    & $t^{a-1}\,dt$
                    & rate $y$\\
                Weibull
                    & $ax^{a-1}\,dx$
                    & $x^a$
                    & $y^{-1}$
                    & $dt$
                    & scale $s=y^{-\frac{1}{a}}$\\
                Pareto
                    & $x^{-1}\,dx$
                    & $\log\pa{\frac{x}{\ell}}$
                    & $y^{-1}$
                    & $dt$
                    & tail index $y$\\
                Beta
                    & $x^{-1}(1-x)^{a-1}\,dx$
                    & $-\log x$
                    & $\mathrm B(y,a)$
                    & $(1-e^{-t})^{a-1}\,dt$
                    & shapes $(y,a)$\\
                Negative binomial
                    & $\frac{\Gamma(n+r)}{\Gamma(r)n!}$
                    & $n$
                    & $(1-e^{-y})^{-r}$
                    & $\sum_{n\geq0}
                        \frac{\Gamma(n+r)}{\Gamma(r)n!}\delta_n$
                    & $p=1-e^{-y}$\\
                Poisson
                    & $\frac{M^n}{n!}$
                    & $n$
                    & $e^{Me^{-y}}$
                    & $\sum_{n\geq0}\frac{M^n}{n!}\delta_n$
                    & mean $\theta=Me^{-y}$\\
                \hline
            \end{tabular}
            \caption{Exponential representations of standard scalar
                families. In the last two rows, the reference-measure
                entry is $m(\{n\})$. All parameters other than $y$ are
                fixed.}
            \label{Tab: exponential_families}
        \end{table}

        In the table, $a$ denotes the fixed mean in the Gaussian row
        and a fixed positive shape in the gamma, Weibull, and beta rows.
        The size $r$, lower endpoint $\ell$, and mean bound $M$ are
        fixed and positive. The observation spaces are $\R$ for Gaussian laws,
        $(0,\infty)$ for gamma and Weibull laws, $[\ell,\infty)$ for
        Pareto laws, $(0,1)$ for beta laws, and $\mathbb N_0$ for the
        last two families. The statistic measure $\rho=T_*m$ has support
        $[0,\infty)$ in the first five rows and $\mathbb N_0$ in the last two.
        The displayed statistic densities in the
        first five rows are on $(0,\infty)$, and every scalar parameter
        $y$ ranges over $(0,\infty)$.

        Completeness admits a separate characterization in the scalar
        case. We first consider the full parameter space $Y=(0,\infty)$
        and characterize completeness simultaneously for every mixing
        law. By Theorem~\ref{T: completeness_rigidity}, this also gives
        a sufficient condition for every mixture over the family to
        be $4$-rigid.

        \begin{proposition}\label{P: exponential_completeness}
            Suppose that $d=1$, $Y=(0,\infty)$, and
            \eqref{E: finite_partition_function} holds. Define
            \begin{equation*}
                a:=\min\supp\rho,
                \quad
                \text{ and }
                \quad
                \mathcal S_\rho
                :=\{t-a:t\in\supp\rho,\ t>a\}.
            \end{equation*}
            For $0<u\leq1$, define the weight
            \begin{equation*}
                w_\rho(u)
                :=\int_{[a,\infty)}
                    e^{-(t-a)}u^{t-a}\,d\rho(t)
                =e^a u^{-a}Z(1-\log u),
            \end{equation*}
            and set $w_\rho(0):=\rho(\{a\})$.
            Then $w_\rho$ is continuous on $[0,1]$ and positive
            on $(0,1]$. For every $\lambda\in\mathcal S_\rho$,
            the function $\frac{u^\lambda}{w_\rho(u)}$ extends
            continuously to zero with value zero.

            The following are equivalent:
            \begin{enumerate}
                \item For every $\nu\in\cP((0,\infty))$, the pair
                    $(\nu,P)$ satisfies completeness
                    \eqref{E: completeness}.
                \item
                    \begin{equation}\label{E: weighted_muntz_completeness}
                        \overline{\operatorname{span}
                        \left\{
                            \frac{u^\lambda}{w_\rho(u)}:
                            \lambda\in\mathcal S_\rho
                        \right\}}^{\,\norm{\cdot}_\infty}
                        =
                        \{g\in C[0,1]:g(0)=0\}.
                    \end{equation}
            \end{enumerate}
            Under these equivalent conditions,
            $\fG_\nu$ is $4$-rigid for every
            $\nu\in\cP((0,\infty))$.
        \end{proposition}

    \subsection{Applications to spherical symmetry}\label{SS: spherical_intro}

        Recall that an infinite exchangeable sequence $X$ is rotatable if for every $d \geq 1$ the tuple $\pa{X_1,\ldots,X_d}$ is invariant under the orthogonal group $O(d)$. By Freedman's theorem \cite{freedman62} (see also \cite[Theorem~1.31]{kallenberg05}), this holds if and only if
        \begin{equation*}
            X\stackrel d=\pa{\sqrt V G_i}_{i=1}^\infty,
        \end{equation*}
        where $V\geq0$ is independent of an iid standard Gaussian sequence $G=\pa{G_i}_{i=1}^\infty$; see also \cite{kelker70,kingman72,steerneman05,kallenberg05}.
        The following theorem shows that spherical symmetry of three coordinates suffices.

        \begin{theorem}\label{T: three-point-rigidity}
            Let \(X=\pa{X_i}_{i=1}^\infty\) be an infinite exchangeable
            sequence. If \(\pa{X_1,X_2,X_3}\) is spherically symmetric, then
            \(X\) is rotatable.
        \end{theorem}
        \begin{remark}
            For earlier work on rotatable and elliptically symmetric sequences, see \cite{kingman72,ars90}.
        \end{remark}

        By \(\gamma_v=\mathcal N\pa{0,v}\) we denote the law of a centered Gaussian random variable on $\R$ with variance $v \geq 0$, with \(\gamma_0=\delta_0\). The following Proposition shows that spherical symmetry of the first two dimensions does not suffice.

        \begin{proposition}\label{P: sharpness}
            Let \(\Pi\) be a probability measure on \([0,\infty)\). The
            rotatable extension $\int_0^\infty \gamma_v^{\infty} d\Pi(v)$ of
            \begin{equation*}
                \cL\pa{X_1,X_2}
                =\int_0^\infty\gamma_v^{\otimes2}\,d\Pi\pa{v}
            \end{equation*}
            is its unique infinite exchangeable extension if and only if the
            restriction of \(\Pi\) to \(\pa{0,\infty}\) is zero or is
            supported on one point. Otherwise the same spherical pair has a
            nonrotatable infinite exchangeable extension.
        \end{proposition}

        For related characterizations involving sphericity and
        exchangeability, see \cite{smith81,letac81,ars90,diaconisfreedman87}.

    \subsection{\texorpdfstring{Energy-radial laws and $\ell_q$-norm exponential laws}{Energy-radial and norm-exponential laws}}
    \label{SS: radial_intro}

        Let $E$ be Polish, let $m$ be a nonzero $\sigma$-finite Borel
        measure on $E$, and let $T:E\to[0,\infty)$ be Borel measurable.
        Put $\rho=T_*m$, and assume that
        $\operatorname{supp}\rho=[0,\infty)$ and
        \begin{equation}\label{E: critical_parameter_intro}
            Z(\beta):=\int_E e^{-\beta T(x)}\,dm(x)<\infty
            \quad\Longleftrightarrow\quad \beta>0,
            \quad \beta\in\R.
        \end{equation}
        For $\beta>0$, let $P_\beta$ have density
        \begin{equation*}
            \frac{dP_\beta}{dm}(x)=\frac{e^{-\beta T(x)}}{Z(\beta)}.
        \end{equation*}
        Let $d\geq2$, and suppose that a probability measure $Q_d$ on
        $E^d$ has $m^{\otimes d}$-density
        \begin{equation*}
            f_d(x_1,\ldots,x_d)=F\pa{\sum_{j=1}^dT(x_j)},
        \end{equation*}
        where $F:[0,\infty)\to[0,\infty)$ is continuous. Only this
        marginal is assumed to have the stated form; its exchangeable
        extensions are otherwise unrestricted.
 
        Recall that $F$ is completely monotone on $(0,\infty)$ if it is
        smooth there and $(-1)^nF^{(n)}\geq0$ for every integer $n\geq0$.
 
        \begin{theorem}\label{T: energy_radial_extendibility}
            The law $Q_d$ has an infinite exchangeable extension if and
            only if $F$ is completely monotone on $(0,\infty)$. In this
            case there is a unique finite positive measure $\mu$ on
            $(0,\infty)$ such that
            \begin{equation*}
                F(t)=\int_{(0,\infty)}e^{-\beta t}\,d\mu(\beta),
                \quad t\geq0.
            \end{equation*}
            The measure $\nu_d$ defined by
            \begin{equation}\label{E: radial_mixing_law}
                d\nu_d(\beta)=Z(\beta)^d\,d\mu(\beta)
            \end{equation}
            is a probability measure, and
            $Q_d=\int_{(0,\infty)}P_\beta^{\otimes d}\,d\nu_d(\beta)$.
            Conditionally iid sampling from $P_A$, with $A\sim\nu_d$,
            therefore gives an exchangeable extension. For $d\geq3$, this
            extension is unique. For $d=2$, it is unique if and only if
            $\nu_2$ is a point mass, equivalently, if and only if
            $F(t)=ce^{-bt}$ for some $c,b>0$.
        \end{theorem}

        For $E=(0,\infty)$, $m(dx)=dx$, and $T(x)=x$, the laws $Q_d$ are $\ell_1$-norm symmetric; their connection with Williamson's $d$-transform and Kimberling's complete-monotonicity criterion for Archimedean generators valid in all dimensions is discussed in \cite{mcneilneslehova09}. Theorem~\ref{T: energy_radial_extendibility} treats general statistics $T$ satisfying the stated hypotheses and establishes the exact uniqueness thresholds among arbitrary infinite exchangeable extensions.
 
        For $q>0$ write
        \begin{equation*}
            \norm{x}_q=\pa{\sum_{j=1}^d\abs{x_j}^q}^{1/q},
        \end{equation*}
        which is the usual quasi-norm when $0<q<1$. For $\lambda>0$ and
        $d\geq1$, let $Q_{q,d,\lambda}$ be the law on $\R^d$ with density
        \begin{equation}\label{E: lq_density_intro}
            f_{q,d,\lambda}(x)
            =c_{q,d,\lambda}\exp\pa{-\lambda\norm{x}_q},
            \quad x\in\R^d,
        \end{equation}
        where $c_{q,d,\lambda}>0$ is the normalizing constant.
 
        \begin{corollary}\label{C: lq_phase_diagram}
            Let $d\geq2$ and $\lambda>0$.
            \begin{enumerate}
                \item[(a)] If $0<q<1$, then $Q_{q,d,\lambda}$ has no
                    infinite exchangeable extension.
                \item[(b)] If $q=1$, then $Q_{q,d,\lambda}$ is the law of
                    $d$ iid Laplace random variables with rate $\lambda$,
                    and its infinite exchangeable extension is unique.
                \item[(c)] If $q>1$, then $Q_{q,d,\lambda}$ has an infinite
                    exchangeable extension. It is unique if $d\geq3$ and
                    not unique if $d=2$.
            \end{enumerate}
        \end{corollary}
 
    \subsection{Organization of the paper}\label{SS: organization}
 
        Section~\ref{S: three_rigidity_proof} proves
        Theorems~\ref{T: general_rigidity}
        and~\ref{T: completeness_rigidity}, and
        Section~\ref{S: exponential_proofs} proves
        Corollary~\ref{C: exponential_rigidity}, characterizes scalar
        completeness, and proves sharpness of the scalar threshold.
        Section~\ref{S: spherical_proofs} proves
        Theorem~\ref{T: three-point-rigidity},
        Proposition~\ref{P: sharpness} and
        Corollary~\ref{C: isotropic_app}.
        Section~\ref{S: radial_applications} contains the applications to
        exchangeable laws whose densities depend only on
        $\sum_j T(x_j)$. Appendix~\ref{A: finite_mixture_operators}
        relates the moment-operator method to finite-mixture bounds.

\section{Rigidity of general mixtures}\label{S: three_rigidity_proof}

    This section proves Theorems~\ref{T: general_rigidity}
    and~\ref{T: completeness_rigidity}. Throughout, $\eta$ is a random
    probability measure on $E$ defined on $(\Omega,\Sigma,\PP)$. We prove both theorems in the stronger form
    given in the remark following Theorem~\ref{T: general_rigidity}: there is a $Y$-valued random variable $A$ on $(\Omega,\Sigma,\PP)$ with law $\nu$ such that $\eta=P_A$ almost surely. Since $\fG_\nu$ is the image of $\nu$ under $y\mapsto P_y$, this gives $\cL(\eta)=\cL(P_A)=\fG_\nu$. All $L^2$ spaces in this section are real.

    For bounded Borel measurable $h:E\to\R$ we write $K_h$ for the function
    \begin{align*}
        K_h :\,\, & Y \to \R \\
        &y \mapsto P_y(h).
    \end{align*}

        For every $f\in L^2(\nu)$, the signed mixture $Q_\nu[f]$ satisfies
        \begin{equation*}
            \|Q_\nu[f]\|_{\mathrm{TV}}
            \leq\int_Y|f|\,d\nu
            \leq\|f\|_{L^2(\nu)},
        \end{equation*}
        where $\|\cdot\|_{\mathrm{TV}}$ denotes the total variation norm.
        The second moment identity will define an isometry on the
        functions $K_h$. To extend this isometry to $L^2(\nu)$, we use
        the following characterization of completeness.

        \begin{lemma}\label{L: Kh_dense}
            The following conditions are equivalent:
            \begin{enumerate}
                \item $(\nu,P)$ satisfies the completeness
                    property~\eqref{E: completeness}.
                \item The linear span of
                    \begin{equation*}
                        \{K_h:h\text{ is a bounded real-valued
                        Borel measurable map from }E\text{ to }\R\}
                    \end{equation*}
                    is dense in $L^2(\nu)$.
                \item If $(U,V)$ has joint law $\nu(du)P_u(dv)$, then
                    \begin{equation*}
                        \EE[f(U)\mid V]=0
                        \quad\Longrightarrow\quad
                        f(U)=0\quad\text{almost surely},
                        \quad f\in L^2(\nu).
                    \end{equation*}
            \end{enumerate}
        \end{lemma}
        \begin{proof}
            Suppose that $f\in L^2(\nu)$ is orthogonal to every
            $K_h$. Since $\nu$ is a probability measure,
            $f\in L^1(\nu)$. For every bounded Borel measurable map $h$,
            Fubini's theorem gives
            \begin{equation*}
                \int_E h\,dQ_\nu[f]
                =\int_Y f(y)P_y(h)\,d\nu(y)=0,
            \end{equation*}
            so $Q_\nu[f]=0$. Conversely, if $Q_\nu[f]=0$, the same
            identity shows that $f$ is orthogonal to every $K_h$.
            Thus the orthogonal complement of their span is the kernel
            of $Q_\nu$, which proves the equivalence of the first two
            conditions.

            To compare with the third condition, put $q:=Q_\nu[\mathbf1]$,
            the law of $V$. Then $Q_\nu[f]\ll q$: if $q(C)=0$, positivity
            gives $P_y(C)=0$ for $\nu$-almost every $y$, and hence
            $Q_\nu[f](C)=0$. Let $r_f:=dQ_\nu[f]/dq$. For every bounded
            Borel measurable $h:E\to\R$,
            \begin{equation*}
                \EE[f(U)h(V)]
                =\int_E h\,dQ_\nu[f]
                =\EE[r_f(V)h(V)].
            \end{equation*}
            Consequently,
            \begin{equation*}
                \frac{dQ_\nu[f]}{dQ_\nu[\mathbf1]}(V)
                =\EE[f(U)\mid V]
                \quad\text{almost surely}.
            \end{equation*}
            Therefore $Q_\nu[f]=0$ if and only if
            $\EE[f(U)\mid V]=0$ almost surely, proving the remaining
            equivalence.
        \end{proof}

    Since $h\mapsto K_h$ is linear, these functions form a linear subspace
    of $L^2(\nu)$. By that lemma, this subspace is dense
    under~\eqref{E: completeness}. For a bounded random variable $u$ and a
    bounded Borel function $v:Y\to\R$, we write $M^\PP_u$ and $M^\nu_v$
    for multiplication by $u$ on $L^2(\PP)$ and by $v$ on $L^2(\nu)$.

    The construction of $A$ is modeled on the case in which $A$ is already
    given. Let $A$ be a $Y$-valued random variable with law $\nu$. Because
    $A\sim\nu$, evaluation $g\mapsto g(A)$ is a linear isometry from
    $L^2(\nu)$ into $L^2(\PP)$. It fixes $\1$ and sends the indicator of
    each Borel set $B\subseteq Y$ to the indicator $\1_{\{A\in B\}}$. It
    also sends $K_h$ to $P_A(h)$, which equals $\eta(h)$ when $\eta=P_A$.
    Our proofs run this observation backwards. Starting from a competing
    $\eta$, completeness and the second moment identity yield a linear
    isometry $W:L^2(\nu)\to L^2(\PP)$ with $W\1=\1$ and $WK_h=\eta(h)$ for
    every bounded Borel measurable $h$. It remains to show that $W$ is
    evaluation at a random parameter with law $\nu$. Evaluation maps are
    exactly the linear isometries fixing $\1$ that send indicators to
    indicators, and it suffices to check the latter on a generating
    $\pi$-system. Everything therefore reduces to showing that $W\1_B$ is
    an indicator for every $B$ in some $\pi$-system generating $\cB(Y)$.

    In Section~\ref{SS: isometry}, Lemma~\ref{L: moment_isometry}
    translates the moment identities into properties of $W$. Under
    completeness, two observations give the isometry $W$ and three give a
    compression identity for $W^*M^\PP_{\eta(h)}W$. Four give exact
    intertwining of $W$ with multiplication by $K_h$, which is the
    identity satisfied by evaluation maps. If $W$ were onto, compression
    would already imply intertwining. This is the situation of
    Remark~\ref{R: finite_surjectivity}, where the target has $m$ linearly
    independent components and the competitor has at most $m$ atoms.
    Without such a bound, the range of $W$ can be a proper closed subspace
    of $L^2(\PP)$. Remark~\ref{R: completeness_sharpness} shows that under
    completeness alone three observations then do not suffice.

    Section~\ref{SS: point_realization} proves the reduction to
    indicators. There, Lemma~\ref{L: point_realization} constructs $A$
    from the indicators $W\1_B$ and shows that $\eta=P_A$ almost surely.
    The two rigidity proofs verify the indicator property using the
    identities available at each order. In
    Section~\ref{SS: proof_general_rigidity},
    Lemma~\ref{L: moment_isometry} shows the following for every event $D$
    on which $W\1_B$ is bounded: the measure
    $\EE[\1_D(W\1_B)^2\eta(\cdot)]$ is a signed mixture dominated by
    $Q_\nu[\1_B]$. Weak localization turns this domination into
    $W\1_B\in\{0,1\}$ almost surely for $B$ in the $\pi$-system
    of~\eqref{E: weak_localization}, proving
    Theorem~\ref{T: general_rigidity}. In
    Section~\ref{SS: proof_completeness_rigidity}, exact intertwining shows
    directly that $W\1_B$ is idempotent for every Borel set $B$. Thus
    completeness suffices at order four, proving
    Theorem~\ref{T: completeness_rigidity}.
    Subsection~\ref{SS: uniform_kernel} verifies localization for a
    simple uniform kernel.

    \subsection{What the moment identities say}\label{SS: isometry}

        By the definition of $\cM_k$, for a directing law $\Lambda$ and
        bounded real-valued Borel maps $h_1,\ldots,h_k$ we have
        \begin{equation*}
            \cM_k\pa{\Lambda}\pa{h_1\otimes\cdots\otimes h_k}
            =\int_{\cP(E)}\prod_{j=1}^k\mu(h_j)\,d\Lambda(\mu),
        \end{equation*}
        where $(h_1\otimes\cdots\otimes h_k)(x_1,\ldots,x_k)
        :=\prod_{j=1}^k h_j(x_j)$. If $\Lambda=\cL(\eta)$, the
        right-hand side is $\EE[\prod_{j=1}^k\eta(h_j)]$. Equality of
        $k$th moment measures is therefore equivalent to equality of
        these mixed moments for every such choice of $h_1,\ldots,h_k$.

        Lower moment measures are obtained by marginalization.
        For $1\leq j\leq k$,
        let $\pi_{k,j}:E^k\to E^j$ be projection onto the first $j$
        coordinates. Fubini's theorem gives
        \begin{equation*}
            (\pi_{k,j})_*\cM_k(\Lambda)=\cM_j(\Lambda).
        \end{equation*}
        Thus equality of $k$th moment measures implies equality of all
        lower moment measures.

        For $k\geq1$ consider the $k$th moment identity
        \begin{equation}\label{E: k_moment_identity}
            \EE[\eta^{\otimes k}]=\int_Y P_y^{\otimes k}\,d\nu(y).
        \end{equation}
        For $k=3$ this is Equation~\eqref{E: three_rigidity}. The
        identity of order $k$ implies those of lower order by
        marginalization.

                \begin{lemma}\label{L: moment_isometry}
            Suppose that $(\nu,P)$ satisfies~\eqref{E: completeness}, and
            let $h$ range over bounded Borel measurable maps $E\to\R$.
            \begin{enumerate}
                \item[(a)] If~\eqref{E: k_moment_identity} holds for
                $k=2$, there is a unique isometric embedding
                $W:L^2(\nu)\to L^2(\PP)$ such that
                \begin{equation}\label{E: WKh}
                    WK_h=\eta(h)
                \end{equation}
                for every $h$. Moreover, $W\1=\1$, and the adjoint
                $W^*:L^2(\PP)\to L^2(\nu)$ satisfies $W^*W=I$ and is
                characterized by
                \begin{equation}\label{E: adjoint_characterization}
                    \ip{W^*u}{K_h}_{L^2(\nu)}=\EE[u\,\eta(h)],
                    \quad u\in L^2(\PP),
                \end{equation}
                for every $h$.
                \item[(b)] If~\eqref{E: k_moment_identity} holds for
                $k=3$, then
                \begin{equation}\label{E: compressed_multiplication}
                    W^*M^\PP_{\eta(h)}W=M^\nu_{K_h}.
                \end{equation}
                \item[(c)] If~\eqref{E: k_moment_identity} holds for
                $k=4$, then
                \begin{equation}\label{E: exact_intertwining}
                    M^\PP_{\eta(h)}W=WM^\nu_{K_h}.
                \end{equation}
            \end{enumerate}
        \end{lemma}

        Part~(c) expresses exact intertwining with multiplication by the
        functions $K_h$, as holds for evaluation at a random
        parameter. Parts~(a) and~(b) are what remains visible from two or
        three observations.

        \begin{proof}
            (a) Integrating $h\otimes h'$ against both sides of the second
            moment identity gives
            \begin{equation*}
                \EE[\eta(h)\eta(h')]
                =\int_Y K_hK_{h'}\,d\nu
                =\ip{K_h}{K_{h'}}_{L^2(\nu)}.
            \end{equation*}
            Applying this to $h-h'$ and using linearity,
            \begin{equation*}
                \EE\bra{\pa{\eta(h)-\eta(h')}^2}
                =\norm{K_h-K_{h'}}^2_{L^2(\nu)}.
            \end{equation*}
            Thus $\eta(h)=\eta(h')$ almost surely whenever $K_h=K_{h'}$ in
            $L^2(\nu)$, so $K_h\mapsto\eta(h)$ is a well-defined linear
            isometry on the subspace of functions $K_h$. This subspace is
            dense by Lemma~\ref{L: Kh_dense}, so the map extends uniquely
            by continuity to a linear isometry
            $W:L^2(\nu)\to L^2(\PP)$. Since $K_1=\1$ and $\eta(1)=1$, we
            have $W\1=\1$.

            Being an isometry, $W$ is bounded with $\norm{W}=1$. For each
            $u\in L^2(\PP)$, the map $g\mapsto\ip{u}{Wg}_{L^2(\PP)}$ is
            therefore a bounded linear functional on $L^2(\nu)$, and the
            Riesz representation theorem gives a unique $W^*u\in L^2(\nu)$
            with
            \begin{equation*}
                \ip{W^*u}{g}_{L^2(\nu)}=\ip{u}{Wg}_{L^2(\PP)},
                \quad g\in L^2(\nu).
            \end{equation*}
            This defines the adjoint $W^*:L^2(\PP)\to L^2(\nu)$, a bounded
            linear operator with $\norm{W^*}=\norm{W}=1$. Since $W$
            preserves inner products, $W^*W=I$. By~\eqref{E: WKh} and the
            density of the functions $K_h$, $W^*u$ is the unique element of
            $L^2(\nu)$ satisfying~\eqref{E: adjoint_characterization} for
            every $h$.

            (b) The third moment identity implies the second, so (a)
            applies. Since $\eta(h)$ and
            $K_h$ are real-valued and bounded by $\norm{h}_\infty$, the multiplication operators $M^\PP_{\eta(h)}$ and
            $M^\nu_{K_h}$ are bounded and self-adjoint. Hence
            both sides of~\eqref{E: compressed_multiplication} are bounded
            operators on $L^2(\nu)$. Integrating $h'\otimes h\otimes h''$
            against both sides of the third moment identity, and using the
            definition of $W^*$ and~\eqref{E: WKh}, we obtain
            \begin{align*}
                \ip{K_{h'}}{W^*M^\PP_{\eta(h)}WK_{h''}}_{L^2(\nu)}
                &=\ip{WK_{h'}}{M^\PP_{\eta(h)}WK_{h''}}_{L^2(\PP)}
                =\EE[\eta(h')\eta(h)\eta(h'')]\\
                &=\int_YK_{h'}K_hK_{h''}\,d\nu
                =\ip{K_{h'}}{M^\nu_{K_h}K_{h''}}_{L^2(\nu)}.
            \end{align*}
            Thus the bounded bilinear forms
            $(f,g)\mapsto\ip{f}{W^*M^\PP_{\eta(h)}Wg}$ and
            $(f,g)\mapsto\ip{f}{M^\nu_{K_h}g}$ on $L^2(\nu)$ agree whenever
            both arguments lie in the dense subspace of functions $K_h$.
            By continuity they agree everywhere, which
            proves~\eqref{E: compressed_multiplication}.

            (c) The fourth moment identity implies the second and third, so
            (a) and (b) apply. Abbreviate $M:=M^\PP_{\eta(h)}$ and
            $N:=M^\nu_{K_h}$. Integrating $h'\otimes h\otimes h\otimes h''$
            against both sides of the fourth moment identity gives, as
            in~(b),
            \begin{align*}
                \ip{K_{h'}}{W^*M^2WK_{h''}}_{L^2(\nu)}
                &=\ip{WK_{h'}}{M^2WK_{h''}}_{L^2(\PP)}
                =\EE[\eta(h')\eta(h)^2\eta(h'')]\\
                &=\int_YK_{h'}K_h^2K_{h''}\,d\nu
                =\ip{K_{h'}}{N^2K_{h''}}_{L^2(\nu)},
            \end{align*}
            and the density argument of~(b) yields
            \begin{equation*}
                W^*M^2W=N^2.
            \end{equation*}
            For $g\in L^2(\nu)$, self-adjointness of $M$ and $N$, the
            definition of $W^*$, the preceding identities, and the isometry
            property of $W$ give the following calculation, with inner
            products in $L^2(\PP)$ on the first line and in $L^2(\nu)$
            thereafter:
            \begin{align*}
                \norm{MWg-WNg}^2_{L^2(\PP)}
                &=\ip{Wg}{M^2Wg}-2\ip{Wg}{MW(Ng)}+\norm{W(Ng)}^2\\
                &=\ip{g}{W^*M^2Wg}-2\ip{g}{W^*MW(Ng)}+\norm{Ng}^2\\
                &=\ip{g}{N^2g}-2\ip{g}{N^2g}+\norm{Ng}^2\\
                &=\norm{Ng}^2-2\norm{Ng}^2+\norm{Ng}^2=0.
            \end{align*}
            Since $g$ was arbitrary, $MW=WN$, which
            is~\eqref{E: exact_intertwining}.
        \end{proof}

        \begin{remark}\label{R: finite_surjectivity}
            Suppose the target has $m$ linearly independent components
            and the competing directing law has at most $m$ atoms.
            Realize the competitor canonically on its finite support.
            The isometry in Lemma~\ref{L: moment_isometry}(a) then maps
            an $m$-dimensional space into a space of dimension at most
            $m$, so it is onto. Hence $WW^*=I$, and multiplying
            \eqref{E: compressed_multiplication} on the left by $W$
            gives exact intertwining from three observations. See
            Appendix~\ref{A: finite_mixture_operators} for the
            identification argument and its block-observation extension.
        \end{remark}

    \subsection{Evaluation at a random parameter}
    \label{SS: point_realization}

        In the model, $W\1_B=\1_{\{A\in B\}}$ is an indicator. The next
        lemma shows that this property characterizes evaluation maps,
        and that it suffices to check a generating $\pi$-system.

        \begin{lemma}\label{L: point_realization}
            Let $W:L^2(\nu)\to L^2(\PP)$ be a linear isometry with
            $W\1=\1$, and suppose that $W\1_B$ is an indicator for every
            $B$ in a $\pi$-system generating $\cB(Y)$. Then there is a
            $Y$-valued random variable $A$ with law $\nu$ such that
            \begin{equation}\label{E: evaluation_isometry}
                Wg=g(A)\quad\text{in }L^2(\PP)
            \end{equation}
            for every $g\in L^2(\nu)$.
            If, in addition, $W$ is the isometry of
            Lemma~\ref{L: moment_isometry}(a), then $\eta=P_A$ almost
            surely.
        \end{lemma}

        \begin{proof}
            We first show that $W\1_B$ is an indicator function for any $B \in \cB(Y)$.
            Since $W$ is a linear isometry, it preserves the inner product
            \begin{equation*}
                \EE[Wf\,Wg]=\int_Y fg\,d\nu
            \end{equation*}
            for all $f, g \in L^2(\nu)$.
            Let $\cH$ be the class of Borel sets $B\subseteq Y$ for
            which $W\1_B$ has an indicator version. This class contains
            $Y$ and is closed under complements because $W\1=\1$ and
            $W\1_{Y\setminus B}=\1-W\1_B$. If
            $B_1,B_2,\ldots\in\cH$ are pairwise disjoint, choose
            indicator versions $S_j$ of $W\1_{B_j}$. For $i\neq j$,
            \begin{equation*}
                \EE[S_iS_j]=\nu(B_i\cap B_j)=0.
            \end{equation*}
            Thus, outside a single null set, at most one $S_j$ equals
            one. The partial sums $\sum_{j\leq n}S_j$ therefore
            converge boundedly to an indicator, while
            $\sum_{j\leq n}\1_{B_j}\to\1_{\bigcup_jB_j}$ in
            $L^2(\nu)$. Continuity of $W$ and bounded convergence give
            \begin{equation*}
                W\1_{\bigcup_jB_j}=\sum_{j\geq1}S_j
                \quad\text{in }L^2(\PP).
            \end{equation*}
            Hence $\cH$ is a Dynkin system. It contains the generating
            $\pi$-system in the hypothesis, so the $\pi$-$\lambda$
            theorem gives $\cH=\cB(Y)$.
        
            We now construct the random parameter $A$. Since $Y$ is Polish,
            \cite[Theorem~1.8(ii)]{KallenbergFoundations} provides a Borel
            set $F\subseteq[0,1]$ and a bijection $\iota:Y\to F$ such
            that both $\iota$ and $\iota^{-1}$ are Borel measurable.
            For $q\in\mathbb Q$, put
            \begin{equation*}
                B_q:=\{y\in Y:\iota(y)>q\},
            \end{equation*}
            and choose an indicator version $S_q$ of $W\1_{B_q}$.
            Since $W\1=\1$, the inner-product identity gives
            $\EE[S_q]=\nu(B_q)$. For rational $q<r$,
            \begin{equation*}
                \EE[S_r(1-S_q)]
                =\nu(B_r)-\nu(B_r\cap B_q)=0,
            \end{equation*}
            so $S_r\leq S_q$ almost surely. Moreover, $S_q=1$ for $q<0$
            and $S_q=0$ for $q\geq1$, almost surely. These relations
            hold simultaneously outside a common measurable null set.
            On that null set, replace $S_q$ by $\1_{\{q<0\}}$,
            so the relations hold everywhere. Define
            \begin{equation*}
                R:=\sup\{q\in\mathbb Q:S_q=1\}\in[0,1].
            \end{equation*}
            For each rational $q$, nesting gives
            \begin{equation*}
                \{R>q\}
                =\bigcup_{n\geq1}\{S_{q+1/n}=1\}
                \subseteq\{S_q=1\}.
            \end{equation*}
            In particular, $R$ is measurable. Since
            $B_{q+1/n}\uparrow B_q$, continuity from below gives
            \begin{equation*}
                \PP(R>q)
                =\lim_{n\to\infty}\EE[S_{q+1/n}]
                =\lim_{n\to\infty}\nu(B_{q+1/n})
                =\nu(B_q)
                =\EE[S_q].
            \end{equation*}
            Consequently, $S_q=\1_{\{R>q\}}$ almost surely, and $R$
            has law $\iota_*\nu$. In particular, $R\in F$ almost
            surely. Fix $y_0\in Y$ and define
            \begin{equation*}
                A:=
                \begin{cases}
                    \iota^{-1}(R),& R\in F,\\
                    y_0,& R\notin F.
                \end{cases}
            \end{equation*}
            Because $F$ is Borel and $\iota^{-1}$ is measurable,
            $A$ is measurable, and its law is $\nu$.
        
            We now identify $W$ with evaluation at $A$. For every
            rational $q$,
            \begin{equation*}
                W\1_{B_q}=S_q
                =\1_{\{R>q\}}
                =\1_{\{A\in B_q\}}
                \quad\text{almost surely}.
            \end{equation*}
            Since $\iota$ is a Borel isomorphism, the sets $B_q$ form
            a $\pi$-system containing $Y$ and generating $\cB(Y)$.
            Their indicators span a dense subspace of $L^2(\nu)$:
            the Borel sets whose indicators lie in the closed linear
            span form a Dynkin system containing this $\pi$-system,
            so the span contains all Borel indicators.
            Evaluation $g\mapsto g(A)$ is an isometry because $A$
            has law $\nu$. Its agreement with $W$ on a dense
            subspace therefore proves \eqref{E: evaluation_isometry}.
        
            If $W$ is the isometry of
            Lemma~\ref{L: moment_isometry}(a), then
            \eqref{E: WKh} and \eqref{E: evaluation_isometry} give
            \begin{equation*}
                \eta(h)=WK_h=K_h(A)=P_A(h)
            \end{equation*}
            almost surely for each bounded Borel function $h$.
            Apply this to the indicators of a countable algebra
            generating $\cB(E)$ and intersect the resulting
            probability-one events. Then $\eta$ and $P_A$ agree
            on that algebra, and hence on $\cB(E)$, almost surely.
        \end{proof}

    \subsection{Three observations under weak localization}
    \label{SS: proof_general_rigidity}

        The moment isometry of Lemma~\ref{L: moment_isometry} assumes
        completeness. For the three-observation proof, we first verify
        that weak localization supplies this hypothesis.

        \begin{lemma}\label{L: localization_completeness}
            Strong localization implies weak localization, and weak
            localization implies completeness:
            \begin{equation*}
                \eqref{E: strong_localization}
                \quad\Longrightarrow\quad
                \eqref{E: weak_localization}
                \quad\Longrightarrow\quad
                \eqref{E: completeness}.
            \end{equation*}
        \end{lemma}
        \begin{proof}
            The first implication follows by integrating over
            $Y\setminus B$. For the second, let $\cA$ be the countable
            generating $\pi$-system in~\eqref{E: weak_localization}.
            If $Q_\nu[f]=0$, then $\int_Y f\,d\nu=0$, and
            \eqref{E: weak_localization} gives
            $\int_{Y\setminus B}f\,d\nu=0$, hence $\int_B f\,d\nu=0$,
            for every $B\in\cA$. Since $\cA\cup\{Y\}$ is a generating
            $\pi$-system, uniqueness of finite signed measures gives
            $f\,d\nu=0$. Thus $f=0$ $\nu$-almost everywhere.
        \end{proof}

        We are now prepared to prove Theorem~\ref{T: general_rigidity}.

        \begin{proof}[Proof of Theorem~\ref{T: general_rigidity}]
            Let $\eta$ be a random probability measure on $E$, defined on
            $(\Omega,\Sigma,\PP)$, satisfying~\eqref{E: three_rigidity}. We need to show that $\eta = P_A$ almost surely for $A \sim \nu$. Weak
            localization implies completeness by
            Lemma~\ref{L: localization_completeness}, so
            Lemma~\ref{L: moment_isometry}(a,b) applies and provides the
            isometry $W$. By Lemma~\ref{L: point_realization}, it suffices
            to show that $S:=W\1_B$ is an indicator for every $B$ in the
            $\pi$-system $\cA$ of~\eqref{E: weak_localization}. Fix such a
            $B$, and choose a finite measurable representative of $S=W\1_B$.

            If $\eta$ were already of the form $P_A$ for a $Y$-valued
            random variable $A$ with law $\nu$, then $W$ would be
            evaluation at $A$, by the uniqueness in
            Lemma~\ref{L: moment_isometry}(a), since evaluation at $A$ is
            an isometry sending $K_h$ to $K_h(A)=\eta(h)$. In that case
            $S=\1_{\{A\in B\}}$. For general $\eta$, we use that $S$ is an
            indicator exactly when $S^2(1-S)=0$ almost surely. It
            therefore suffices to prove that
            \begin{equation}\label{E: cubic_identity}
                \EE[\1_DS^2(1-S)]=0
            \end{equation}
            for every event $D\in\Sigma$ on which $S$ is bounded. Indeed,
            put $H:=S^2(1-S)$ and fix $N\geq1$. For
            $D=\{\abs{S}\leq N,\ H>0\}$,~\eqref{E: cubic_identity} gives
            $\EE[\1_DH]=0$ with $H>0$ on $D$, so $\PP(D)=0$. The same
            argument applies to $\{\abs{S}\leq N,\ H<0\}$. Thus $H=0$
            almost surely on $\{\abs{S}\leq N\}$. Since these events
            increase to $\Omega$, letting $N\to\infty$ gives $H=0$ almost
            surely.

            To prove~\eqref{E: cubic_identity}, fix $D\in\Sigma$ on which
            $S$ is bounded, and consider the localized measure $\mu_D$ on $E$ defined by
            \begin{equation*}
                \mu_D(C):=\EE[\1_DS^2\eta(C)]
            \end{equation*}
            for all $C \in\cB(E)$, i.e. $\eta$ averaged over $D$ with weight $S^2$. Since
            $0\leq\eta(C)\leq1$ and $S\in L^2(\PP)$, monotone convergence
            shows that $\mu_D$ is a finite positive measure, and
            $\mu_D(h)=\EE[\1_DS^2\eta(h)]$ for every bounded Borel $h$. The
            same holds for the unlocalized measure $\mu_\Omega$. Because
            $S$ is bounded on $D$, the random variable $\1_DS^2$ is bounded
            and so lies in $L^2(\PP)$. Put
            \begin{equation*}
                f_D:=W^*(\1_DS^2)\in L^2(\nu).
            \end{equation*}
            We show that, in the setwise order on measures,
            \begin{equation}\label{E: localized_sandwich}
                0\leq Q_\nu[f_D]=\mu_D\leq\mu_\Omega=Q_\nu[\1_B].
            \end{equation}
            The first equality says that $\mu_D$ is a mixture of the
            components with signed weight $f_D$. The last says that the
            unlocalized measure is the mixture of the components with
            parameters in $B$. The inequalities are immediate, and
            together these facts place $f_D$ within the scope of weak
            localization.

            For the first equality,~\eqref{E: adjoint_characterization}
            with $u=\1_DS^2$ gives, for every bounded Borel $h$,
            \begin{equation*}
                \mu_D(h)
                =\EE[\1_DS^2\eta(h)]
                =\ip{f_D}{K_h}_{L^2(\nu)}
                =\int_Yf_D(y)P_y(h)\,d\nu(y)
                =Q_\nu[f_D](h).
            \end{equation*}
            The inequalities $0\leq\mu_D\leq\mu_\Omega$ hold because
            $0\leq\1_DS^2\leq S^2$ and $\eta\geq0$. For the last equality,
            the compression
            identity~\eqref{E: compressed_multiplication} gives, for every
            bounded Borel $h$,
            \begin{align*}
                \mu_\Omega(h)
                &=\ip{W\1_B}{M^\PP_{\eta(h)}W\1_B}_{L^2(\PP)}
                =\ip{\1_B}{W^*M^\PP_{\eta(h)}W\1_B}_{L^2(\nu)}\\
                &=\ip{\1_B}{M^\nu_{K_h}\1_B}_{L^2(\nu)}
                =\int_BP_y(h)\,d\nu(y)
                =Q_\nu[\1_B](h).
            \end{align*}

            By~\eqref{E: localized_sandwich}, weak localization applies to
            $f_D$. By weak localization, the definition of $W^*$, and
            $W(\1-\1_B)=\1-S$, we obtain
            \begin{align*}
                0
                &=\int_{Y\setminus B}f_D\,d\nu\\
                &=\ip{W^*(\1_DS^2)}{\1-\1_B}_{L^2(\nu)}\\
                &=\ip{\1_DS^2}{W(\1-\1_B)}_{L^2(\PP)}\\
                &=\EE[\1_DS^2(1-S)].
            \end{align*}
            This proves~\eqref{E: cubic_identity}.

            Hence $W\1_B$ is an indicator for every $B\in\cA$.
            Lemma~\ref{L: point_realization} now gives a $Y$-valued random
            variable $A$ with law $\nu$ such that $\eta=P_A$ almost surely,
            so $\cL(\eta)=\fG_\nu$.
        \end{proof}

    \subsection{Four observations under completeness}
    \label{SS: proof_completeness_rigidity}

        We are now prepared to prove Theorem~\ref{T: completeness_rigidity}.

        \begin{proof}[Proof of Theorem~\ref{T: completeness_rigidity}]
            Let $\eta$ be a random probability measure on $E$, defined on
            $(\Omega,\Sigma,\PP)$, with
            $\EE[\eta^{\otimes4}]=\cM_4(\fG_\nu)$. We need to show that $\eta = P_A$ almost surely for $A \sim \nu$. Since $(\nu,P)$
            satisfies~\eqref{E: completeness},
            Lemma~\ref{L: moment_isometry}(a,c) supplies an isometry $W$
            satisfying the exact intertwining
            identity~\eqref{E: exact_intertwining}. By
            Lemma~\ref{L: point_realization}, it suffices to show that
            $W\1_B$ is an indicator for every Borel set $B$.
        
            Fix such a $B$ and put $S:=W\1_B$. By
            Lemma~\ref{L: Kh_dense}, we may choose bounded Borel maps
            $h_n:E\to\R$ such that $K_{h_n}\to\1_B$ in $L^2(\nu)$.
            Applying~\eqref{E: exact_intertwining} to $\1_B$ gives
            \begin{equation*}
                W(K_{h_n}\1_B)=\eta(h_n)S.
            \end{equation*}
            Since $W$ is an isometry and $\PP$ is a probability measure,
            \begin{align*}
                \norm{W(K_{h_n}\1_B)-S}_{L^1(\PP)}
                &\leq\norm{W(K_{h_n}\1_B)-W\1_B}_{L^2(\PP)}\\
                &=\norm{(K_{h_n}-\1_B)\1_B}_{L^2(\nu)}\\
                &\leq\norm{K_{h_n}-\1_B}_{L^2(\nu)}
                \longrightarrow0.
            \end{align*}
            On the other hand, $\eta(h_n)=WK_{h_n}$, so the
            Cauchy--Schwarz inequality gives
            \begin{align*}
                \norm{\eta(h_n)S-S^2}_{L^1(\PP)}
                &\leq\norm{\eta(h_n)-S}_{L^2(\PP)}
                     \norm{S}_{L^2(\PP)}\\
                &=\norm{K_{h_n}-\1_B}_{L^2(\nu)}
                     \norm{S}_{L^2(\PP)}
                \longrightarrow0.
            \end{align*}
            The intertwining identity identifies the two sequences,
            so uniqueness of the $L^1(\PP)$ limit gives $S=S^2$
            almost surely. Thus $W\1_B$ is an indicator for every
            Borel set $B$. Lemma~\ref{L: point_realization} yields
            a $Y$-valued random variable $A$ with law $\nu$ such that
            $\eta=P_A$ almost surely, and hence
            $\cL(\eta)=\fG_\nu$.
        \end{proof}

        \begin{remark}\label{R: three_versus_four}
            Weak localization applies to the localized weights
            $\1_DS^2$, where $S=W\1_B$ and $D$ is an event on which
            $S$ is bounded, without requiring them to belong to the range
            of $W$; see \eqref{E: localized_sandwich}.
        \end{remark}

        \begin{corollary}
            \label{C: translation_gaussian_rigidity}
            Let $g\in\cP(\R^p)$ have a characteristic function whose zero
            set has empty interior, and let $P_y=g(\,\cdot-y)$.
            Then $\fG_\nu$ is $4$-rigid for every $\nu\in\cP(\R^p)$.
            In particular, if the first four coordinates of an infinite
            exchangeable sequence are jointly Gaussian, then the sequence
            is a Gaussian process.
        \end{corollary}
        \begin{proof}
            For Borel $B\subseteq\R^p$, the map
            $y\mapsto P_y(B)=\int_{\R^p}\1_B(y+z)\,g(dz)$ is Borel by
            Tonelli's theorem, so $P$ is a measurable probability kernel.
            Fix $\nu\in\cP(\R^p)$ and a real $f\in L^2(\nu)\subseteq L^1(\nu)$,
            and put $\sigma:=f\nu$. For a finite signed measure $\tau$ on
            $\R^p$, write $\widehat\tau(t):=\int e^{i\ip{t}{x}}\,\tau(dx)$.
            By Fubini's theorem, $Q_\nu[f]=\sigma*g$, and hence
            \begin{equation*}
                \widehat{Q_\nu[f]}(t)=\widehat\sigma(t)\,\widehat g(t),
                \quad t\in\R^p.
            \end{equation*}
            If $Q_\nu[f]=0$, then the continuous function $\widehat\sigma$
            vanishes on $\{\widehat g\neq0\}$, which is dense by hypothesis,
            and therefore everywhere. Uniqueness of Fourier transforms of
            finite signed measures gives $\sigma=0$, that is, $f=0$
            $\nu$-almost everywhere. Thus \eqref{E: completeness} holds, and
            Theorem~\ref{T: completeness_rigidity} shows that $\fG_\nu$ is
            $4$-rigid.

            For the Gaussian assertion, let $X$ be an infinite exchangeable sequence with
            $(X_1,\ldots,X_4)$ jointly Gaussian, and put $\mu:=\EE X_1$,
            $b:=\operatorname{Cov}(X_1,X_2)$, and
            $a:=\operatorname{Var}(X_1)-b$. By exchangeability,
            $2a=\operatorname{Var}(X_1-X_2)\geq0$ and
            \begin{equation*}
                0\leq\operatorname{Var}\Bigl(\frac1n\sum_{j=1}^nX_j\Bigr)
                =\frac an+b
                \quad\text{for all }n\geq1,
            \end{equation*}
            so $b\geq0$. With the convention $\mathcal N(m,0):=\delta_m$,
            apply the first part with $g=\mathcal N(0,a)$, whose
            characteristic function $e^{-at^2/2}$ never vanishes, and
            $\nu=\mathcal N(\mu,b)$; let $\Gamma$ be the resulting
            $4$-rigid directing law. If $A\sim\mathcal N(\mu,b)$ is
            independent of iid standard Gaussian variables
            $G_1,G_2,\ldots$, then $Y_j:=A+\sqrt a\,G_j$ are conditionally
            iid with law $\mathcal N(A,a)$, so
            $\cL(Y_1,\ldots,Y_n)=\cM_n(\Gamma)$ for every $n$. Each such
            block is Gaussian with mean $\mu$ in every coordinate and
            covariance matrix $aI+bJ$, where $J$ is the all-ones matrix.
            The block $(X_1,\ldots,X_4)$ has the same mean and covariance,
            so the directing law $\Lambda$ of $X$ satisfies
            $\cM_4(\Lambda)=\cL(X_1,\ldots,X_4)=\cM_4(\Gamma)$. Hence
            $\Lambda=\Gamma$, and every finite block of $X$ has the
            Gaussian law $\cM_n(\Gamma)=\cL(Y_1,\ldots,Y_n)$. Thus $X$ is a
            Gaussian process.
        \end{proof}

    \subsection{A localized uniform kernel}\label{SS: uniform_kernel}

        The following example verifies localization by recovering a signed
        measure from its tails. It gives a non-exponential application
        of Theorem~\ref{T: general_rigidity}.

        \begin{example}\label{Ex: uniform_localization}
            Let $E=Y=(0,\infty)$ and let $P_y$ be the uniform
            probability measure on $(0,y)$. For every
            $\nu\in\cP(Y)$, the pair $(\nu,P)$ satisfies strong
            localization, and hence $\fG_\nu$ is $3$-rigid.
        
            To see this, take the countable generating $\pi$-system
            \begin{equation*}
                \cA=\{Y\}\cup
                \{(0,b]:b\in\mathbb Q,\ b>0\}.
            \end{equation*}
            For real $f\in L^2(\nu)$, the signed mixture $Q_\nu[f]$
            has Lebesgue density
            \begin{equation*}
                q_f(x):=\int_{(x,\infty)}\frac{f(y)}{y}\,d\nu(y),
                \quad x>0.
            \end{equation*}
            This integral is absolutely convergent for every $x>0$,
            and $q_f$ is right-continuous. If $B=(0,b]$ and
            \begin{equation*}
                0\leq Q_\nu[f]\leq Q_\nu[\mathbf1_B],
            \end{equation*}
            then $q_f(x)=0$ for Lebesgue-almost every $x>b$,
            since $Q_\nu[\mathbf1_B]$ is supported on $(0,b]$.
            Right-continuity gives $q_f(x)=0$ for every $x>b$.
            Therefore
            \begin{equation*}
                \int_{(a,c]}\frac{f(y)}{y}\,d\nu(y)
                =q_f(a)-q_f(c)=0,
                \quad b<a<c<\infty.
            \end{equation*}
            Uniqueness of finite signed measures on $(b,\infty)$
            now gives $f=0$ $\nu$-almost everywhere on $(b,\infty)$,
            proving \eqref{E: strong_localization}.
        \end{example}

\section{Exponential mixtures}\label{S: exponential_proofs}

    We answer two analytic questions about exponential kernels. The
    first is whether domination of signed mixtures recovers their
    parameter support; under \eqref{E: exponential_MS}, it does and
    gives $3$-rigidity. The second is whether a signed mixture can
    vanish without its mixing measure vanishing. In one parameter,
    this injectivity question has an exact weighted M\"untz criterion.

    Subsection~\ref{SS: laplace_proof} proves scalar support recovery
    and its product extension. Subsection~\ref{SS: proof_exponential_rigidity}
    identifies signed-mixture densities with Laplace integrals,
    controls the inverse partition function, and proves localization.
    Subsection~\ref{SS: exponential_completeness} proves the weighted
    completeness characterization and explains why the statistic
    measure's weights can matter. Subsection~\ref{SS: exponential_two_nonrigidity}
    constructs pair-preserving perturbations that show the scalar
    threshold is exact. Subsection~\ref{SS: exponential_examples}
    gives the standard families, the support-growth examples, and
    the product construction.

    \subsection{Laplace transforms along M\"untz--Sz\'asz sequences}
    \label{SS: laplace_proof}

        In this subsection, we prove the Laplace-transform results used
        to establish localization in
        Corollary~\ref{C: exponential_rigidity}. The main support question
        is when exponential decay at rate $b$ along a sequence of
        exponents forces a finite signed measure on $[0,\infty)$ to be
        supported in $[b,\infty)$. We first prove the general
        M\"untz--Sz\'asz criterion in Proposition~\ref{L: laplace_support}.
        
        Lemma~\ref{L: laplace_uniqueness} proves that Laplace integrals on products of these
        sequences uniquely determine a finite signed measure. Combining this uniqueness
        result with the scalar support estimates gives the
        coordinatewise support restrictions needed for localization
        in several parameters.

        \begin{proof}[Proof of Proposition~\ref{L: laplace_support}]
            We prove the cycle (i)$\Rightarrow$(iii)$\Rightarrow$(ii)$\Rightarrow$(i).
            The first implication uses the Arendt--Batty--Hieber--Neubrander
            support theorem for Laplace transforms sampled along a M\"untz
            sequence; the last uses the full M\"untz theorem to construct a
            measure contradicting (ii).

            Assume (i), fix $\sigma\ne0$, and set
            $\delta:=\inf_{n\geq0}(\lambda_{n+1}-\lambda_n)>0$.
            Replacing $\lambda_n$ by $\frac{\lambda_n}{\delta}$ and $\sigma$
            by its image under $y\mapsto\delta y$ multiplies both sides of
            (iii) by $\delta$ and leaves (i) unchanged. For the proof of
            this implication, we may therefore assume
            $\lambda_{n+1}-\lambda_n\geq1$ for all $n\geq0$, so $\lambda_n\geq n$.

            Since $(\lambda_n)_{n\geq1}$ has gaps at least one and divergent
            reciprocal sum, parts (i) and (ii) of
            \cite[Theorem~2.8.3, p.~103]{ArendtEtAl2011}
            together give
            \begin{equation*}
                \inf\supp\pa{|H(t)|\,dt}
                =-\limsup_{n\to\infty}\frac{1}{\lambda_n}
                    \log\abs{\int_0^\infty e^{-\lambda_n t}H(t)\,dt}
            \end{equation*}
            for every bounded measurable function $H$ on $[0,\infty)$
            that is not zero almost everywhere.
            Put $F(t):=\sigma([0,t])$, a bounded right-continuous function.
            For every $b>0$, $F=0$ almost everywhere on $[0,b)$ if and only
            if $\sigma|_{[0,b)}=0$. Indeed, right-continuity turns
            almost-everywhere vanishing into pointwise vanishing on $[0,b)$,
            and
            $\sigma(\{0\})=F(0)$ and $\sigma((s,t])=F(t)-F(s)$ determine
            the restriction of $\sigma$ to $[0,b)$. Consequently,
            \begin{equation*}
                \inf\supp\pa{|F(t)|\,dt}=\inf\supp|\sigma|.
            \end{equation*}
            Each side is the supremum of the $b\geq0$ for which the
            corresponding measure vanishes on $[0,b)$.
            Since $\sigma\ne0$, letting $b\to\infty$ shows that $F$ is not
            zero almost everywhere, as the support theorem requires.

            For $\lambda>0$, the bound
            \begin{equation*}
                \int_{[0,\infty)}\int_y^\infty
                    e^{-\lambda t}\,dt\,d|\sigma|(y)
                \leq\frac{\|\sigma\|_{\mathrm{TV}}}{\lambda}<\infty
            \end{equation*}
            justifies Fubini's theorem, which gives
            \begin{equation*}
                \int_0^\infty e^{-\lambda t}F(t)\,dt
                =\int_{[0,\infty)}\int_y^\infty
                    e^{-\lambda t}\,dt\,d\sigma(y)
                =\frac{\widehat\sigma(\lambda)}{\lambda}.
            \end{equation*}
            We obtain
            \begin{align*}
                \inf\supp|\sigma|
                    &=-\limsup_{n\to\infty}\frac{1}{\lambda_n}
                        \log\abs{\frac{\widehat\sigma(\lambda_n)}{\lambda_n}}\\
                    &=-\limsup_{n\to\infty}\frac{1}{\lambda_n}
                        \log|\widehat\sigma(\lambda_n)|.
            \end{align*}
            The first equality is the support identity applied to $F$.
            The second follows from
            \begin{equation*}
                \frac{1}{\lambda_n}
                    \log\abs{\frac{\widehat\sigma(\lambda_n)}{\lambda_n}}
                =\frac{1}{\lambda_n}\log|\widehat\sigma(\lambda_n)|
                    -\frac{\log\lambda_n}{\lambda_n}.
            \end{equation*}
            Both sides of this identity are $-\infty$ when
            $\widehat\sigma(\lambda_n)=0$. Since $\lambda_n\geq n\to\infty$
            and $\frac{\log x}{x}\to0$ as $x\to\infty$, the correction
            tends to zero and does not affect the limsup.
            This proves (i)$\Rightarrow$(iii).

            Suppose (iii) holds, and fix a finite signed measure $\sigma$
            and $b\geq0$. The condition
            $\sigma|_{[0,b)}=0$ gives
            \begin{equation*}
                |\widehat\sigma(\lambda_n)|
                \leq\|\sigma\|_{\mathrm{TV}}e^{-b\lambda_n}
                \quad\text{ for all }\quad n\geq1,
            \end{equation*}
            proving the implication $\Rightarrow$ in \eqref{E: laplace_support}.
            If $\sigma=0$, the equivalence is immediate. If $\sigma\ne0$
            and the right-hand side of \eqref{E: laplace_support} holds,
            then $|\widehat\sigma(\lambda_n)|\leq Ce^{-b\lambda_n}$ for some
            $C>0$ and all $n\geq1$. Then
            \begin{equation*}
                \frac{1}{\lambda_n}\log|\widehat\sigma(\lambda_n)|
                \leq\frac{\log C}{\lambda_n}-b.
            \end{equation*}
            The gap condition implies $\lambda_n\to\infty$, so the limsup
            of the left-hand side is at most $-b$. By (iii),
            $\inf\supp|\sigma|\geq b$, and hence $\sigma|_{[0,b)}=0$.
            This proves (iii)$\Rightarrow$(ii).

            We prove (ii)$\Rightarrow$(i) by contraposition.
            Suppose $\sum_{n\geq1}\frac{1}{\lambda_n}<\infty$.
            Since $\frac{\lambda_n}{1+\lambda_n^2}\leq\frac{1}{\lambda_n}$,
            the full M\"untz theorem \cite[Theorem~2.1]{BorweinErdelyi1996}
            shows that $\operatorname{span}\{1,u^{\lambda_n}:n\geq1\}$
            is not dense in $C([0,1])$. The Hahn--Banach and Riesz
            representation theorems therefore give a nonzero finite real
            signed measure $\tau$ on $[0,1]$ such that
            \begin{equation*}
                \tau([0,1])=0
                \quad\text{and}\quad
                \int_{[0,1]}u^{\lambda_n}\,d\tau(u)=0
                \quad\text{ for all }\quad n\geq1.
            \end{equation*}
            The restriction $\tau|_{(0,1]}$ is nonzero, since otherwise
            the total-mass identity would force $\tau=0$. Let $\sigma$ be
            its image under $u\mapsto-\log u$. This map is a bijection from
            $(0,1]$ onto $[0,\infty)$ with inverse $y\mapsto e^{-y}$, so
            $\sigma\ne0$. For all $n\geq1$,
            \begin{equation*}
                \widehat\sigma(\lambda_n)
                =\int_{(0,1]}u^{\lambda_n}\,d\tau(u)
                =\int_{[0,1]}u^{\lambda_n}\,d\tau(u)=0,
            \end{equation*}
            where the second equality uses $\lambda_n>0$, so that
            $u^{\lambda_n}$ vanishes at $u=0$.
            The lower endpoint of $\supp|\sigma|$ is finite, so choose
            $b>\inf\supp|\sigma|$.
            The right-hand side of \eqref{E: laplace_support} holds trivially,
            but $\sigma|_{[0,b)}\ne0$. Thus (ii) fails, proving
            (ii)$\Rightarrow$(i).
        \end{proof}

        In several parameters, scalar support recovery first shows that
        certain weighted coordinate projections vanish below a
        threshold. A vanishing projection can hide cancellation between
        the other coordinates. The next lemma rules out that
        cancellation when all Laplace weights on the remaining product
        grid vanish; Lemma~\ref{L: multivariate_laplace_support} uses
        it to recover the full signed measure on each excluded slab.

        \begin{lemma}\label{L: laplace_uniqueness}
            For $1\leq i\leq d$, let
            $0=\lambda_{i,0}<\lambda_{i,1}<\cdots\to\infty$ satisfy
            $\sum_{n\geq1}\lambda_{i,n}^{-1}=\infty$.
            Let $\sigma$ be a finite signed Borel measure on
            $[0,\infty)^d$. If
            \begin{equation*}
                \int_{[0,\infty)^d}
                    \exp\pa{-\sum_{i=1}^d\lambda_{i,n_i}y_i}
                    \,d\sigma(y)=0
            \end{equation*}
            for all nonnegative integers $n_1,\ldots,n_d$, then
            $\sigma=0$.
        \end{lemma}

        \begin{proof}
            Let $\tau$ be the image of $\sigma$ under
            $y\mapsto(e^{-y_1},\ldots,e^{-y_d})$, regarded as a finite
            signed measure on $[0,1]^d$ with no mass outside $(0,1]^d$.
            The hypothesis says that every generalized product moment
            of $\tau$ is zero. By the full M\"untz theorem
            \cite[Theorem~2.1]{BorweinErdelyi1996}, the
            span of $\{u^{\lambda_{i,n}}:n\geq0\}$ is uniformly dense
            in $C([0,1])$ for each $i$: because
            $\lambda_{i,n}\to\infty$, the divergence of
            $\sum_n\frac{1}{\lambda_{i,n}}$ is equivalent to the
            divergence of
            $\sum_n\frac{\lambda_{i,n}}{1+\lambda_{i,n}^2}$.
            Approximating each factor of an ordinary monomial and
            using polynomial density on $[0,1]^d$ shows that the span
            of the generalized products is uniformly dense there. Integration against $\tau$ therefore
            vanishes on every continuous function. Hence $\tau=0$,
            and therefore $\sigma=0$.
        \end{proof}

        We now combine scalar support recovery with uniqueness in the
        remaining coordinates to obtain a support restriction for a
        finite signed measure on $[0,\infty)^d$.

        \begin{lemma}\label{L: multivariate_laplace_support}
            Let $\sigma$ be a finite signed Borel measure on $[0,\infty)^d$
            and let $b\in[0,\infty)^d$. For each $1\leq i\leq d$, let
            $0=\mu_{i,0}<\mu_{i,1}<\cdots$ be uniformly separated and satisfy
            \begin{equation*}
                \sum_{n=1}^{\infty}\frac{1}{\mu_{i,n}}=\infty.
            \end{equation*}
            Write $\mu_{\boldsymbol n}:=(\mu_{1,n_1},\ldots,\mu_{d,n_d})$.
            Suppose that, for some $0\leq C<\infty$,
            \begin{equation*}
                \abs{\int_{[0,\infty)^d}
                    e^{-\ip{y}{\mu_{\boldsymbol n}}}\,d\sigma(y)}
                \leq C e^{-\ip{b}{\mu_{\boldsymbol n}}}
                \quad
                \text{ for all }
                \quad
                \boldsymbol n\in\mathbb N_0^d.
            \end{equation*}
            Then
            \begin{equation*}
                \sigma\big|_{[0,\infty)^d
                    \setminus\prod_{i=1}^d[b_i,\infty)}=0.
            \end{equation*}
        \end{lemma}
        
        \begin{proof}
            When $d=1$, the conclusion follows from
            Proposition~\ref{L: laplace_support}. Suppose that $d\geq2$.
            Fix a coordinate $i$. We show that $\sigma$ vanishes on
            \begin{equation*}
                H_i:=\{y\in[0,\infty)^d:y_i<b_i\}
            \end{equation*}
            by applying the scalar support lemma to weighted projections
            onto that coordinate, and then using uniqueness in the remaining
            coordinates.
        
            Fix indices $n_j\geq0$ for $j\ne i$, and write
            $\boldsymbol n_{-i}:=(n_j)_{j\ne i}$.
            For each Borel set $A\subseteq[0,\infty)$, define
            \begin{equation*}
                \sigma_{i,\boldsymbol n_{-i}}(A)
                :=\int_{[0,\infty)^d}\1_A(y_i)
                    \exp\pa{-\sum_{j\ne i}\mu_{j,n_j}y_j}
                    \,d\sigma(y).
            \end{equation*}
            This is a finite signed measure because the exponential weight
            lies in $[0,1]$. Its Laplace transform at $\mu_{i,n}$ is the
            integral in the hypothesis with $n_i=n$. Hence
            \begin{equation*}
                \abs{\widehat\sigma_{i,\boldsymbol n_{-i}}(\mu_{i,n})}
                \leq C\exp\pa{-b_i\mu_{i,n}
                              -\sum_{j\ne i}b_j\mu_{j,n_j}}
                \leq C e^{-b_i\mu_{i,n}}
                \quad
                \text{ for all }
                \quad n\geq0,
            \end{equation*}
            where the second inequality uses $b_j\geq0$ and
            $\mu_{j,n_j}\geq0$. Proposition~\ref{L: laplace_support} gives
            \begin{equation}\label{E: weighted_projection_support}
                \sigma_{i,\boldsymbol n_{-i}}|_{[0,b_i)}=0
            \end{equation}
            for every choice of the remaining indices.
        
            To recover $\sigma$ on $H_i$, fix a Borel set
            $A\subseteq[0,b_i)$ and write $y_{-i}:=(y_j)_{j\ne i}$.
            Define a finite signed measure $\tau_A$ on $[0,\infty)^{d-1}$ by
            \begin{equation*}
                \tau_A(D):=\int_{[0,\infty)^d}
                    \1_A(y_i)\1_D(y_{-i})\,d\sigma(y)
                \quad
                \text{ for all }
                \quad D\in\cB([0,\infty)^{d-1}).
            \end{equation*}
            For every choice of $n_j\geq0$, $j\ne i$, the definitions give
            \begin{align*}
                \int_{[0,\infty)^{d-1}}
                    \exp\pa{-\sum_{j\ne i}\mu_{j,n_j}u_j}\,d\tau_A(u)
                &=\int_{[0,\infty)^d}\1_A(y_i)
                    \exp\pa{-\sum_{j\ne i}\mu_{j,n_j}y_j}\,d\sigma(y)\\
                &=\sigma_{i,\boldsymbol n_{-i}}(A)=0,
            \end{align*}
            where $u=(u_j)_{j\ne i}$ and the last equality follows from
            Equation~\eqref{E: weighted_projection_support}.
            Lemma~\ref{L: laplace_uniqueness}, applied to the remaining
            sequences, therefore gives $\tau_A=0$.
        
            Thus $\sigma$ assigns zero mass to every rectangle specified by
            $y_i\in A$ and $y_{-i}\in D$, with $A\subseteq[0,b_i)$ and $D$
            Borel. These rectangles form a $\pi$-system generating
            $\cB(H_i)$ and include $H_i$ itself. Since $\sigma$ is finite,
            the $\pi$--$\lambda$ theorem gives $\sigma|_{H_i}=0$.
            Applying this to every coordinate yields
            \begin{equation*}
                |\sigma|\pa{[0,\infty)^d
                    \setminus\prod_{i=1}^d[b_i,\infty)}
                \leq\sum_{i=1}^d|\sigma|(H_i)=0,
            \end{equation*}
            which proves the claim.
        \end{proof}

    \subsection{Localization of exponential mixtures}
    \label{SS: proof_exponential_rigidity}

        We prove Corollary~\ref{C: exponential_rigidity} by turning
        measure domination into decay bounds for a tilted signed
        mixing measure. The preceding Laplace support lemmas then
        exclude parameters outside the prescribed upper rectangle.

        Fix $\nu\in\cP(Y)$ and the point $t_0\in\supp\rho$ in
        \eqref{E: exponential_MS}. Put
        \begin{equation*}
            z(y):=\frac{1}{Z(y)},
            \quad U:=t_0+(0,\infty)^d.
        \end{equation*}
        We call $U$ the open upper orthant based at $t_0$, and
        understand vector inequalities coordinatewise. Mass near $t_0$
        gives a lower bound for $Z(y)$. Consequently the tilted inverse
        partition function $z(y)e^{-\ip{y}{t}}$ is locally uniformly
        bounded for $t\in U$. This makes the signed-mixture transforms
        continuous there, even for mixing laws with unbounded parameter
        support. We use the support grid from \eqref{E: exponential_MS}
        after shifting it into $U$.

        \begin{lemma}\label{L: normalizer_bound}
            For every $c\in U$, the map
            \begin{equation*}
                y\longmapsto z(y)e^{-\langle y,c\rangle}
            \end{equation*}
            is bounded on $Y$.
        \end{lemma}

        \begin{proof}
            Fix $c\in U$ and choose $r$ with $t_0<r<c$.
            Choose
            \begin{equation*}
                0<\varepsilon<\min_{1\leq i\leq d}
                    \{r_i-(t_0)_i\},
                \quad
                V:=B(t_0,\varepsilon)\cap[0,\infty)^d,
            \end{equation*}
            where $B(t_0,\varepsilon)$ is the open Euclidean ball.
            Then $t\leq r$ for every $t\in V$, and
            $t_0\in\supp\rho$ gives $\rho(V)>0$.
            Choose $y_*\in Y$. Equation~\eqref{E: finite_partition_function}
            gives
            \begin{equation*}
                0<\rho(V)\leq e^{\langle y_*,r\rangle}Z(y_*)<\infty.
            \end{equation*}
            For every $y\in Y$,
            \begin{equation*}
                Z(y)\geq\int_V e^{-\langle y,t\rangle}\,d\rho(t)
                    \geq e^{-\langle y,r\rangle}\rho(V).
            \end{equation*}
            Consequently,
            \begin{equation*}
                z(y)e^{-\langle y,c\rangle}
                \leq\frac{e^{-\langle y,c-r\rangle}}{\rho(V)}
                \leq\frac{1}{\rho(V)},
            \end{equation*}
            which proves the claim.
        \end{proof}

        The transform below represents the density of $Q_\nu[f]$
        relative to $m$, evaluated at the statistic $T(x)$. Its
        continuity lets us replace the almost-everywhere inequalities
        between densities by pointwise inequalities on
        $\supp\rho\cap U$.

        \begin{lemma}\label{L: transform_continuity}
            For every $f\in L^2(\nu)$, the map $\mathcal L_\nu f$
            defined in \eqref{E: exponential_transform} is well-defined
            and continuous on $U$.
        \end{lemma}

        \begin{proof}
            Since $\nu$ is a probability measure, $f\in L^1(\nu)$.
            Fix $t\in U$ and choose $c$ with $t_0<c<t$.
            Lemma~\ref{L: normalizer_bound} gives a constant
            $C_c<\infty$ such that
            \begin{equation*}
                z(y)e^{-\langle y,c\rangle}\leq C_c,
                \quad y\in Y.
            \end{equation*}
            This proves that the integral in
            Equation~\eqref{E: exponential_transform} is absolutely
            convergent. If $t_n\to t$, then $t_n\geq c$ for all
            sufficiently large $n$, and hence
            \begin{equation*}
                \abs{f(y)z(y)e^{-\langle y,t_n\rangle}}
                \leq C_c|f(y)|.
            \end{equation*}
            Dominated convergence then yields continuity at $t$.
        \end{proof}

        The integral in~\eqref{E: exponential_transform} is also
        absolutely convergent for $\rho$-almost every $t$ outside $U$.
        Indeed, Tonelli's theorem and the definition of $Z$ give
        \begin{equation}\label{E: exponential_fubini}
            \int_{[0,\infty)^d}\int_Y
                |f(y)|z(y)e^{-\langle y,t\rangle}\,d\nu(y)\,d\rho(t)
            =\int_Y|f(y)|\,d\nu(y)<\infty.
        \end{equation}
        Fubini's theorem therefore shows that the signed mixture
        $Q_\nu[f]$ has $m$-density
        $(\mathcal L_\nu f)(T(\cdot))$.

        \begin{proof}[Proof of Corollary~\ref{C: exponential_rigidity}]
            Fix $\nu\in\cP(Y)$. We prove strong localization
            \eqref{E: strong_localization} for the countable generating
            $\pi$-system
            \begin{equation*}
                \cA:=\left\{
                    B_b:=Y\cap\prod_{i=1}^d[b_i,\infty):
                    b\in(\mathbb Q\cap[0,\infty))^d
                \right\}.
            \end{equation*}
            Fix $B=B_b\in\cA$ and a real $f\in L^2(\nu)$ satisfying
            \begin{equation*}
                0\leq Q_\nu[f]\leq Q_\nu[\1_B].
            \end{equation*}
            We must show that $f=0$ $\nu$-almost everywhere on $Y\setminus B$.
            We obtain bounds for a tilted signed measure and apply
            Lemma~\ref{L: multivariate_laplace_support} to exclude that set.
        
            By~\eqref{E: exponential_fubini}, the two signed mixtures have
            $m$-densities $(\mathcal L_\nu f)(T(\cdot))$ and
            $(\mathcal L_\nu\1_B)(T(\cdot))$. Since $\rho=T_*m$,
            their domination gives
            $0\leq\mathcal L_\nu f\leq\mathcal L_\nu\1_B$
            $\rho$-almost everywhere. The functions $\mathcal L_\nu f$ and
            $\mathcal L_\nu\1_B$ are continuous on $U$ by
            Lemma~\ref{L: transform_continuity}. Failure of either inequality
            at a point of $\supp\rho\cap U$ would therefore imply failure
            on an open neighborhood of positive $\rho$-measure. Consequently,
            \begin{equation*}
                0\leq(\mathcal L_\nu f)(t)
                    \leq(\mathcal L_\nu\1_B)(t)
                \quad
                \text{ for all }
                \quad t\in\supp\rho\cap U.
            \end{equation*}
        
            To use these inequalities at the points supplied
            by~\eqref{E: exponential_MS}, set
            \begin{align*}
                s&:=t_0+(\lambda_{1,1},\ldots,\lambda_{d,1}),\\
                \mu_{i,n}&:=\lambda_{i,n+1}-\lambda_{i,1}
                    \quad\text{ for all }\quad n\geq0,
            \end{align*}
            for each $1\leq i\leq d$.
            Each shifted sequence starts at zero and is uniformly separated.
            Its reciprocal sum diverges because
            $0<\mu_{i,n}<\lambda_{i,n+1}$ for every $n\geq1$.
            Writing $\mu_{\boldsymbol n}:=(\mu_{1,n_1},\ldots,\mu_{d,n_d})$,
            condition~\eqref{E: exponential_MS} gives
            $s+\mu_{\boldsymbol n}\in\supp\rho\cap U$ for every
            $\boldsymbol n\in\mathbb N_0^d$.
        
            Lemma~\ref{L: normalizer_bound}, applied with $c=s$, bounds
            $z(y)e^{-\ip{y}{s}}$ on $Y$. Since $f\in L^1(\nu)$, the measure
            \begin{equation*}
                d\sigma_s(y):=f(y)z(y)e^{-\ip{y}{s}}\,d\nu(y)
            \end{equation*}
            on $Y$, extended by zero outside $Y$, is finite. Moreover,
            \begin{equation*}
                C_s:=\int_B z(y)e^{-\ip{y}{s}}\,d\nu(y)<\infty.
            \end{equation*}
            At $s+\mu_{\boldsymbol n}$, the transform inequalities give
            nonnegativity and the upper bound below. Since $y\geq b$ on $B$
            and $\mu_{\boldsymbol n}\geq0$, we obtain
            \begin{align}
                \abs{\int_{[0,\infty)^d}
                    e^{-\ip{y}{\mu_{\boldsymbol n}}}\,d\sigma_s(y)}
                &=(\mathcal L_\nu f)(s+\mu_{\boldsymbol n})\notag\\
                &\leq\int_B z(y)e^{-\ip{y}{s}}
                    e^{-\ip{y}{\mu_{\boldsymbol n}}}\,d\nu(y)\notag\\
                &\leq C_s e^{-\ip{b}{\mu_{\boldsymbol n}}}
                    \label{E: exponential_transform_domination}
            \end{align}
            for every $\boldsymbol n\in\mathbb N_0^d$.
        
            Lemma~\ref{L: multivariate_laplace_support} now shows that
            $\sigma_s$ vanishes outside $\prod_{i=1}^d[b_i,\infty)$.
            Hence
            \begin{equation*}
                0=|\sigma_s|(Y\setminus B)
                =\int_{Y\setminus B}|f(y)|z(y)e^{-\ip{y}{s}}\,d\nu(y).
            \end{equation*}
            The weight $z(y)e^{-\ip{y}{s}}$ is strictly positive on $Y$,
            so $f=0$ $\nu$-almost everywhere on $Y\setminus B$.
            This proves strong localization. Strong localization implies
            weak localization, and Theorem~\ref{T: general_rigidity}
            therefore gives $3$-rigidity of $\fG_\nu$.
        \end{proof}
        
    \subsection{Completeness in one parameter}
    \label{SS: exponential_completeness}

        We prove Proposition~\ref{P: exponential_completeness}, which
        characterizes completeness simultaneously for every mixing
        law on the full parameter space $Y=(0,\infty)$. By
        Theorem~\ref{T: completeness_rigidity}, this also gives a
        sufficient condition for every mixture over the family to
        be $4$-rigid.

        If the minimum of the statistic support is an atom, the
        weight in \eqref{E: weighted_muntz_completeness} is bounded
        above and below by positive constants. \cite[Theorem~2.1]{BorweinErdelyi1996} then gives an explicit criterion.

        \begin{corollary}\label{C: exponential_completeness_atom}
            Under the assumptions of
            Proposition~\ref{P: exponential_completeness}, suppose
            additionally that $\rho(\{a\})>0$. Then completeness
            \eqref{E: completeness} holds for every
            $\nu\in\cP((0,\infty))$ if and only if
            \begin{equation*}
                \sup_{\substack{F\subseteq\mathcal S_\rho\\
                                F\text{ finite}}}
                    \sum_{\lambda\in F}
                        \frac{\lambda}{1+\lambda^2}
                =\infty.
            \end{equation*}
        \end{corollary}

        The criterion allows arbitrary statistic support, including
        support with finite accumulation points. When
        $\mathcal S_\rho=\{\lambda_n:n\geq1\}$ with
        $0<\lambda_1<\lambda_2<\cdots\to\infty$, it reduces to
        \begin{equation*}
            \sum_{n = 1}^\infty\frac{1}{\lambda_n}=\infty,
        \end{equation*}
        without any uniform separation assumption.

        In general, the weight $w_\rho$ is essential:
        Example~\ref{Ex: exponential_completeness_weight} gives two
        exponential kernels with the same statistic support, one
        complete for every mixing law and the other incomplete for
        some mixing law. We prove both criteria and give the example
        below.

        \begin{proof}[Proof of
            Proposition~\ref{P: exponential_completeness}]
            We first verify the properties of the weight, then reduce
            completeness to injectivity on finite signed mixing measures.
            The substitution $u=e^{1-y}$ and a weighted translation
            relate injectivity to the weighted density condition.
            In the converse direction, the translation also supplies
            the possible missing zero at $t=a$.
        
            The integrand defining $w_\rho(u)$ is positive for $u>0$
            and bounded by $e^{-(t-a)}$, whose $\rho$-integral is
            $e^aZ(1)<\infty$. At zero it converges pointwise to
            $\mathbf1_{\{a\}}(t)$. Dominated convergence therefore
            gives continuity of $w_\rho$ on $[0,1]$, and positivity
            gives $w_\rho(u)>0$ on $(0,1]$.
        
            Set
            \begin{equation*}
                p_t(y):=\frac{e^{-ty}}{Z(y)}.
            \end{equation*}
            For $a<s<t$, we have
            $0<\rho([a,s])\leq e^sZ(1)<\infty$ and
            $Z(y)\geq e^{-sy}\rho([a,s])$. Hence
            \begin{equation*}
                0<p_t(y)\leq
                    \frac{e^{-(t-s)y}}{\rho([a,s])}
                \quad\text{ for all }\quad y>0.
            \end{equation*}
            Thus $p_t$ is bounded and tends to zero as $y\to\infty$.
            When $\rho(\{a\})>0$, we also have
            $p_a(y)\leq\frac{1}{\rho(\{a\})}$.
            The expression for $w_\rho$ gives
            \begin{equation}\label{E: exponential_completeness_substitution}
                p_{a+\lambda}(1-\log u)
                =e^{-\lambda}\frac{u^\lambda}{w_\rho(u)}
                \quad\text{ for all }\quad
                \lambda>0,\ u\in(0,1].
            \end{equation}
            Since $1-\log u\to\infty$ as $u\downarrow0$, this proves
            the asserted continuous extension of the weighted
            monomials at zero.
        
            Completeness for every mixing law is equivalent to
            injectivity of $\xi\mapsto\int P_y\,d\xi(y)$ on finite
            signed Borel measures on $(0,\infty)$. Indeed,
            injectivity applied to $d\xi=f\,d\nu$, with
            $f\in L^2(\nu)\subseteq L^1(\nu)$, gives completeness.
            Conversely, a nonzero $\xi$ in the kernel gives
            \begin{equation*}
                \nu:=\frac{|\xi|}{\norm{\xi}_{\mathrm{TV}}},
                \quad
                f:=\frac{d\xi}{d|\xi|},
            \end{equation*}
            for which $|f|=1$ $\nu$-almost everywhere and
            $Q_\nu[f]=0$, so completeness fails.
        
            For such a signed measure $\xi$, define
            \begin{equation*}
                F_\xi(t):=\int_{(0,\infty)}p_t(y)\,d\xi(y),
                \quad t>a,
            \end{equation*}
            and also at $a$ when $\rho(\{a\})>0$.
            The preceding bounds make these integrals finite.
            Since $p_t\leq p_b$ whenever $t\geq b>a$, dominated
            convergence gives continuity of $F_\xi$ on $(a,\infty)$.
            Tonelli's theorem gives
            \begin{equation*}
                \int_{E\times(0,\infty)}p_{T(x)}(y)\,
                    d(m\otimes|\xi|)(x,y)
                =\norm{\xi}_{\mathrm{TV}},
            \end{equation*}
            because each $P_y$ is a probability measure.
            Fubini's theorem therefore gives $F_\xi\circ T$ as the
            density of $\int P_y\,d\xi(y)$, with its value at $a$
            needed only when $\rho(\{a\})>0$. The signed mixture
            vanishes exactly when $F_\xi=0$ $\rho$-almost everywhere.
            Since $F_\xi$ is continuous on $(a,\infty)$, this is
            equivalent to
            \begin{equation}\label{E: exponential_completeness_kernel}
                F_\xi(t)=0
                \quad\text{ for all }\quad
                t\in\supp\rho\cap(a,\infty),
                \quad\text{ and }\quad
                F_\xi(a)=0\text{ if }\rho(\{a\})>0.
            \end{equation}
        
            We use $u=e^{1-y}$ so that
            $w_\rho(1)=e^aZ(1)$ is finite even when $\rho$ has
            infinite mass. To move arbitrary mixing measures into
            $(1,\infty)$, define
            \begin{equation*}
                \int h\,d(\tau\xi)
                :=\int h(y+1)e^a\frac{Z(y+1)}{Z(y)}\,d\xi(y)
            \end{equation*}
            for bounded Borel measurable $h$.
            Since $\rho$ is concentrated on $[a,\infty)$, the
            multiplier lies in $(0,1]$. Thus $\tau\xi$ is finite
            and concentrated on $(1,\infty)$; positivity of the
            multiplier and injectivity of translation also make
            $\tau$ injective. Direct substitution gives
            \begin{equation}\label{E: exponential_completeness_translation}
                F_{\tau\xi}(t)=e^{-(t-a)}F_\xi(t)
                \quad\text{ for all }\quad t>a,
            \end{equation}
            and also at $a$ when $\rho(\{a\})>0$.
        
            Put $C_0:=\{g\in C[0,1]:g(0)=0\}$, identified by
            restriction with $C_0((0,1])$.
            Suppose first that \eqref{E: weighted_muntz_completeness}
            holds and $\xi$ satisfies
            \eqref{E: exponential_completeness_kernel}.
            Let $\sigma$ be the pushforward of $\tau\xi$ under
            $y\mapsto e^{1-y}$, viewed as a measure on $(0,1]$.
            For every $\lambda\in\mathcal S_\rho$,
            \begin{equation*}
                \int_{(0,1]}\frac{u^\lambda}{w_\rho(u)}\,d\sigma(u)
                =e^\lambda F_{\tau\xi}(a+\lambda)
                =F_\xi(a+\lambda)=0,
            \end{equation*}
            since $a+\lambda\in\supp\rho\cap(a,\infty)$.
            Density implies that $\sigma$ annihilates $C_0$,
            so Riesz uniqueness gives $\sigma=0$.
            Hence $\tau\xi=0$, and injectivity of $\tau$ gives
            $\xi=0$.
        
            Conversely, if \eqref{E: weighted_muntz_completeness}
            fails, Hahn--Banach and Riesz representation on
            $C_0((0,1])$ give a nonzero finite signed measure
            $\sigma$ on $(0,1]$ annihilating every weighted monomial.
            Its pushforward $\xi$ under $u\mapsto1-\log u$ is
            nonzero, finite, and concentrated on $[1,\infty)$.
            By \eqref{E: exponential_completeness_substitution},
            $F_\xi(a+\lambda)=0$ for all
            $\lambda\in\mathcal S_\rho$.
            Set $\zeta:=\xi-\tau\xi$. Then
            \begin{equation*}
                F_\zeta(t)=\bigl(1-e^{-(t-a)}\bigr)F_\xi(t)
            \end{equation*}
            for $t>a$, and also at $a$ when $\rho(\{a\})>0$.
            Thus $\zeta$ satisfies
            \eqref{E: exponential_completeness_kernel}, including
            the possible condition at $a$.
            Moreover, $y_0:=\min\supp|\xi|\geq1$ exists, and
            $|\xi|([y_0,y_0+1))>0$. Since $\tau\xi$ is concentrated
            on $[y_0+1,\infty)$, $\zeta$ agrees with $\xi$ on this
            interval and is therefore nonzero. Injectivity fails.
        
            This proves the equivalence. The rigidity conclusion
            follows from Theorem~\ref{T: completeness_rigidity}.
        \end{proof}

        \begin{proof}[Proof of
            Corollary~\ref{C: exponential_completeness_atom}]
            By Proposition~\ref{P: exponential_completeness}, it
            suffices to show that
            \eqref{E: weighted_muntz_completeness} is equivalent to
            the divergence condition in
            Corollary~\ref{C: exponential_completeness_atom}. We
            remove the weight using the atom at $a$, adjoin the
            constants to reach the setting of the full M\"untz
            theorem, and then reduce arbitrary exponent sets to
            countable ones. Let $C_0$ be as in the proof of
            Proposition~\ref{P: exponential_completeness}. For
            $\mathcal S\subseteq(0,\infty)$, write
            \begin{equation*}
                \mathcal V(\mathcal S)
                :=\operatorname{span}
                    \bigl(\{\mathbf1\}\cup\{u^\lambda:\lambda\in\mathcal S\}\bigr),
                \quad
                \Delta(\mathcal S)
                :=\sup_{\substack{F\subseteq\mathcal S\\F\text{ finite}}}
                    \sum_{\lambda\in F}\frac{\lambda}{1+\lambda^2}.
            \end{equation*}
            The full M\"untz theorem concerns density of
            $\mathcal V(\mathcal S)$, and the condition in
            Corollary~\ref{C: exponential_completeness_atom} reads
            $\Delta(\mathcal S_\rho)=\infty$.
         
            We first remove the weight. For $0<u\leq1$, the integrand
            defining $w_\rho(u)$ is nonnegative, equals $1$ at $t=a$,
            and is at most $e^{-(t-a)}$. Together with
            $w_\rho(0)=\rho(\{a\})$, this gives
            \begin{equation*}
                0<\rho(\{a\})\leq w_\rho(u)\leq e^aZ(1)<\infty
                \quad\text{ for all }\quad u\in[0,1].
            \end{equation*}
            By Proposition~\ref{P: exponential_completeness},
            $w_\rho$ is continuous, so $\frac{1}{w_\rho}$ is
            continuous and bounded as well. Multiplication by
            $w_\rho$ is therefore a linear homeomorphism of $C_0$
            onto itself, with inverse given by multiplication by
            $\frac{1}{w_\rho}$. It sends each weighted monomial
            $\frac{u^\lambda}{w_\rho(u)}$ to $u^\lambda$. Hence
            \eqref{E: weighted_muntz_completeness} holds if and only
            if $\operatorname{span}\{u^\lambda:\lambda\in\mathcal S_\rho\}$
            is dense in $C_0$.
         
            We next adjoin the constants. Let $V\subseteq C_0$ be a
            subspace. If $V$ is dense in $C_0$, the decomposition
            $g=g(0)+(g-g(0))$ shows that $\R\mathbf1+V$ is dense in
            $C[0,1]$. Conversely, the projection $g\mapsto g-g(0)$ is
            a continuous linear map of $C[0,1]$ onto $C_0$ that sends
            $\R\mathbf1+V$ onto $V$. It therefore carries density of
            $\R\mathbf1+V$ in $C[0,1]$ to density of $V$ in $C_0$.
            Applied to
            $V=\operatorname{span}\{u^\lambda:\lambda\in\mathcal S_\rho\}$,
            for which $\R\mathbf1+V=\mathcal V(\mathcal S_\rho)$,
            this shows that \eqref{E: weighted_muntz_completeness}
            holds if and only if $\mathcal V(\mathcal S_\rho)$ is
            dense in $C[0,1]$.
         
            We now apply the full M\"untz theorem. For countable
            $\mathcal S$, the summands are positive, so
            $\Delta(\mathcal S)$ is the sum of the series in any
            enumeration of $\mathcal S$. The full M\"untz theorem
            \cite[Theorem~2.1]{BorweinErdelyi1996} characterizes
            density for an arbitrary sequence of distinct positive
            exponents, with no separation or monotonicity condition.
            This generality is needed here, since $\mathcal S_\rho$
            may have finite accumulation points. For countably
            infinite $\mathcal S$, the theorem states that
            $\mathcal V(\mathcal S)$ is dense in $C[0,1]$ if and only
            if $\Delta(\mathcal S)=\infty$. The same equivalence holds
            for finite $\mathcal S$, including $\mathcal S=\emptyset$,
            since then $\Delta(\mathcal S)<\infty$ and
            $\mathcal V(\mathcal S)$ is finite-dimensional, hence
            closed and proper. Thus, for every countable
            $\mathcal S\subseteq(0,\infty)$, density of
            $\mathcal V(\mathcal S)$ in $C[0,1]$ is equivalent to
            $\Delta(\mathcal S)=\infty$.
         
            It remains to remove the countability assumption. If
            $\Delta(\mathcal S_\rho)=\infty$, choose finite sets
            $F_n\subseteq\mathcal S_\rho$ such that
            \begin{equation*}
                \sum_{\lambda\in F_n}\frac{\lambda}{1+\lambda^2}\geq n
                \quad\text{ for all }\quad n\geq1.
            \end{equation*}
            Their union $\mathcal S'$ is countable and satisfies
            $\Delta(\mathcal S')=\infty$. Hence
            $\mathcal V(\mathcal S')$ is dense, and so is the larger
            space $\mathcal V(\mathcal S_\rho)$. If
            $\Delta(\mathcal S_\rho)<\infty$, then for each $k\geq1$
            at most $k\Delta(\mathcal S_\rho)$ exponents satisfy
            $\frac{\lambda}{1+\lambda^2}\geq\frac{1}{k}$. Every
            $\lambda>0$ satisfies this inequality for some $k$, so
            $\mathcal S_\rho$ is countable. The countable case then
            shows that $\mathcal V(\mathcal S_\rho)$ is not dense.
            This proves the equivalence.
        \end{proof}

        \begin{example}\label{Ex: exponential_completeness_weight}
            Let $E=[0,1]$, $T(x)=x$, and $Y=(0,\infty)$.
            Consider the two base measures
            \begin{equation*}
                \rho_-:=\sum_{n\geq1}e^{-2\pi n}\delta_{1/n},
                \quad
                \rho_+:=\delta_0+\rho_-.
            \end{equation*}
            Both have support
            \begin{equation*}
                \supp\rho_-=\supp\rho_+
                    =\{0\}\cup\{1/n:n\geq1\},
            \end{equation*}
            and both define exponential kernels with finite,
            positive partition functions on $(0,\infty)$.
            Denote these kernels and partition functions by
            $P_y^\pm$ and $Z_\pm(y)$.

            Since $\rho_+(\{0\})=1$ and
            \begin{equation*}
                \sum_{n\geq1}
                    \frac{1/n}{1+1/n^2}
                =\infty,
            \end{equation*}
            Corollary~\ref{C: exponential_completeness_atom}
            shows that $(\nu,P^+)$ is complete for every
            $\nu\in\cP((0,\infty))$.

            For $P^-$, define
            \begin{equation*}
                h(y):=\sum_{k\geq0}
                    \frac{(-1)^k\pi^{2k+1}y^{2k}}
                         {(2k+1)!(2k)!}.
            \end{equation*}
            The ratio test gives locally uniform convergence. The
            coefficients are chosen so that termwise Laplace
            integration produces $\sin(\frac{\pi}{t})$, which vanishes
            at every $t=\frac1n$. For every
            $t>0$, integration of its absolute terms gives
            \begin{equation*}
                \int_0^\infty e^{-ty}|h(y)|\,dy
                \leq
                \sum_{k\geq0}
                    \frac{\pi^{2k+1}}
                         {(2k+1)!t^{2k+1}}
                =\sinh\pa{\frac{\pi}{t}}.
            \end{equation*}
            Termwise integration is therefore justified and yields
            \begin{equation*}
                \int_0^\infty e^{-ty}h(y)\,dy
                =\sin\pa{\frac{\pi}{t}}.
            \end{equation*}
            In particular,
            \begin{equation*}
                0<A:=\int_0^\infty Z_-(y)|h(y)|\,dy
                \leq
                \sum_{n\geq1}e^{-2\pi n}\sinh(\pi n)
                <\infty.
            \end{equation*}
            Positivity of $A$ follows from $h(0)=\pi$ and continuity.
            Define
            \begin{equation*}
                d\nu(y):=\frac{Z_-(y)|h(y)|}{A}\,dy,
                \quad
                f(y):=\operatorname{sgn}h(y).
            \end{equation*}
            Then $\nu$ is a probability measure and
            $\norm{f}_{L^2(\nu)}=1$. Nevertheless, the signed
            mixture associated with $P^-$ satisfies
            \begin{equation*}
                Q_\nu^-[f](\{1/n\})
                =
                \frac{e^{-2\pi n}}{A}
                    \int_0^\infty e^{-y/n}h(y)\,dy
                =
                \frac{e^{-2\pi n}}{A}\sin(\pi n)
                =0
            \end{equation*}
            for every $n\geq1$. Since $P_y^-$ is concentrated
            on $\{1/n:n\geq1\}$, this proves $Q_\nu^-[f]=0$.
            Thus $(\nu,P^-)$ is not complete, despite having
            the same statistic support as the universally
            complete kernel $P^+$.
        \end{example}

    \subsection{The exact threshold in one parameter}
    \label{SS: exponential_two_nonrigidity}

        In this subsection we show that the rigidity conclusion at
        order three in Corollary~\ref{C: exponential_rigidity} is sharp
        in one parameter whenever $T$ is not almost everywhere constant
        and the mixing law is not a point mass. We do so by constructing another
        directing law with the same second moment measure as
        $\mathfrak G_\nu$ but a different third moment measure.

        We begin with the point-mass case, for which the directing law
        is a single atom. The following observation gives the required
        two-coordinate uniqueness.

        \begin{lemma}\label{L: product_law_rigidity}
            For every $\mu\in\cP(E)$, the directing law $\delta_\mu$
            is $2$-rigid.
        \end{lemma}
        \begin{proof}
            If $\EE[\eta^{\otimes2}]=\mu^{\otimes2}$, then for every
            bounded Borel map $h:E\to\R$,
            \begin{equation*}
                \EE[\eta(h)]=\mu(h),
                \quad
                \EE[\eta(h)^2]=\mu(h)^2.
            \end{equation*}
            Thus $\operatorname{Var}(\eta(h))=0$, so
            $\eta(h)=\mu(h)$ almost surely. Applying this to the
            indicators of a countable generating algebra gives
            $\eta=\mu$ almost surely.
        \end{proof}

        The proof of Proposition~\ref{P: exponential_two_nonrigidity}
        uses only the scalar exponential form and the
        nonnegativity and nonconstancy of $T$. The M\"untz--Sz\'asz
        condition is needed for the complementary rigidity result at
        order three, but not for this counterexample construction. If $\nu$ is a point mass, $\mathfrak G_\nu$ is
        $2$-rigid by Lemma~\ref{L: product_law_rigidity}.
        Example~\ref{Ex: multiparameter_two_rigidity} shows that the
        restriction $d=1$ cannot be removed.

        The construction will use finite signed perturbations of $\fG_\nu$.
        Let $\Xi$ be a finite signed Borel measure on $\cP(E)$ and let $k\geq0$.
        For every bounded Borel measurable
        $s:E^k\to\R$, the map $\mu\mapsto\mu^{\otimes k}(s)$ is bounded
        and Borel measurable. This follows first for indicators of Borel
        rectangles, and then for all bounded Borel maps by a
        monotone-class argument. Consequently $\cM_k(\Xi)$ is a
        well-defined finite signed measure and
        \begin{equation*}
            \cM_k\pa{\Xi}(s)
            =\int_{\cP(E)}\mu^{\otimes k}(s)\,d\Xi(\mu)
        \end{equation*}
        for every bounded Borel map $s:E^k\to\R$.

        The map $\Xi\mapsto\cM_k(\Xi)$ is linear on finite signed
        measures. We will use this linearity when forming and averaging
        the perturbations below.

        The construction begins with two distinct component laws
        $P$ and $Q$ and the affine line
        \begin{equation*}
            \mu_\vartheta=P+\vartheta(Q-P).
        \end{equation*}
        By multilinearity, $\mu_\vartheta^{\otimes j}$ is a polynomial
        in $\vartheta$ of degree at most $j$. We choose four distinct
        parameters $\vartheta_0,\ldots,\vartheta_3$ and nonzero weights
        $a_0,\ldots,a_3$ satisfying
        \begin{equation*}
            \sum_{i=0}^3 a_i\vartheta_i^j=0
            \quad\text{ for all } \quad j\in\{0,1,2\}.
        \end{equation*}
        The signed perturbation
        $\Xi:=\sum_{i=0}^3a_i\delta_{\mu_{\vartheta_i}}$
        therefore has zero total mass and zero first and second
        moment measures.
    
        Preserving the third moment measure would impose a fourth
        condition,
        \begin{equation*}
            \sum_{i=0}^3a_i\vartheta_i^3=0.
        \end{equation*}
        Indeed, choose a bounded measurable function $h$ with
        $P(h)=0$ and $Q(h)=1$. Then
        $\mu_\vartheta(h)=\vartheta$, so testing the third moment
        measure against $h^{\otimes3}$ gives this condition.
        Together, the four conditions would make the weighted sum
        vanish for every polynomial of degree at most three.
        For each $i$, the cubic polynomial
        \begin{equation*}
            p_i(\vartheta)
            :=\prod_{k\ne i}
                \frac{\vartheta-\vartheta_k}
                     {\vartheta_i-\vartheta_k}
        \end{equation*}
        equals $1$ at $\vartheta_i$ and $0$ at the other three
        parameters. Hence
        $\sum_{k=0}^3a_kp_i(\vartheta_k)=a_i=0$.
        Thus preserving the third moment measure would force the
        perturbation to vanish. Preserving the fourth would also
        rule out this perturbation, since the third moment measure
        is a marginal of the fourth.
    
        To turn these cancellations into a perturbation of
        $\mathfrak G_\nu$, we must also preserve positivity.
        A one-sided bound $Q\leq MP$ allows one of the four measures
        to lie slightly beyond $P$ on the affine line while remaining
        in $\cP(E)$. The weights remove mass only at $P$ and $Q$.
        We then average these perturbations over suitable pairs of
        component laws and verify that their negative part is
        dominated by $\mathfrak G_\nu$, while their third moment
        remains nonzero. A sufficiently small perturbation therefore
        gives the required probability law.

        \begin{lemma}\label{L: four_chord}
            Let $P,Q\in\cP(E)$ and suppose that $Q\leq MP$ for some
            $M>1$. Choose $0<\delta<\frac{1}{M-1}$, and define
            \begin{equation*}
                (\vartheta_0,\vartheta_1,\vartheta_2,\vartheta_3)
                    =\pa{-\delta,0,\frac{1}{2},1},
                \quad\text{ and }\quad
                a_i=-\frac{1}
                    {\prod_{k\ne i}(\vartheta_i-\vartheta_k)}.
            \end{equation*}
            Write $\mu_\vartheta:=P+\vartheta(Q-P)$.
            Then all four measures $\mu_{\vartheta_i}$ belong to $\cP(E)$.
            The finite signed measure
            \begin{equation*}
                \Xi_{P,Q}:=\sum_{i=0}^3a_i\delta_{\mu_{\vartheta_i}}
            \end{equation*}
            on $\cP(E)$ satisfies
            \begin{equation}\label{E: four_chord_cancellation}
                \begin{aligned}
                    \cM_j(\Xi_{P,Q})&=0
                        \quad\text{ for all }\quad j\in\{0,1,2\},\\
                    \cM_3(\Xi_{P,Q})&=-(Q-P)^{\otimes3}.
                \end{aligned}
            \end{equation}
            Moreover, $a_0,a_2>0$ and $a_1,a_3<0$, and the Jordan
            negative part $\Xi_{P,Q}^-$ satisfies
            \begin{equation*}
                \Xi_{P,Q}^-\leq -a_1\delta_P-a_3\delta_Q.
            \end{equation*}
        \end{lemma}
        
        \begin{proof}
            The only measure outside the chord joining $P$ and $Q$ is
            $\mu_{-\delta}=(1+\delta)P-\delta Q$. The bound $Q\leq MP$
            gives
            \begin{equation*}
                \mu_{-\delta}
                    \geq\bra{1-\delta(M-1)}P\geq0,
            \end{equation*}
            and its total mass is one. The remaining three measures are
            convex combinations of $P$ and $Q$. Thus all four measures
            belong to $\cP(E)$.
        
            To prove the moment identities, we first verify that the
            weights cancel the constant, linear, and quadratic powers
            of $\vartheta$, while giving $-1$ for the cubic power.
            For each $i\in\{0,1,2,3\}$, construct the cubic polynomial
            \begin{equation*}
                p_i(\vartheta)
                    :=\prod_{k\ne i}
                        \frac{\vartheta-\vartheta_k}
                             {\vartheta_i-\vartheta_k}.
            \end{equation*}
            This polynomial equals $1$ at $\vartheta_i$ and $0$ at the
            other three parameters. Consequently, for each
            $j\in\{0,1,2,3\}$, the polynomials $\vartheta^j$ and
            $\sum_{i=0}^3\vartheta_i^jp_i(\vartheta)$ agree at all four
            parameters. Their difference has degree at most three and
            four distinct roots, so
            \begin{equation*}
                \vartheta^j
                    =\sum_{i=0}^3\vartheta_i^jp_i(\vartheta)
                \quad\text{ for all }\quad j\in\{0,1,2,3\}.
            \end{equation*}
            The coefficient of $\vartheta^3$ in $p_i$ is $-a_i$.
            Comparing cubic coefficients therefore gives
            \begin{align*}
                \sum_{i=0}^3a_i\vartheta_i^j&=0
                    \quad\text{ for all }\quad j\in\{0,1,2\}, 
                    \quad \text{ and } \quad
                \sum_{i=0}^3a_i\vartheta_i^3=-1.
            \end{align*}
        
            The identity for $j=0$ gives $\cM_0(\Xi_{P,Q})=0$.
            For the remaining moment measures, expand
            $\mu_\vartheta^{\otimes j}$ by multilinearity.
            Summing with the weights $a_i$ cancels every term of degree
            zero, one, or two in $\vartheta$. Hence
            $\cM_1(\Xi_{P,Q})=\cM_2(\Xi_{P,Q})=0$.
            When $j=3$, the only surviving term is
            \begin{equation*}
                \cM_3(\Xi_{P,Q})
                    =\pa{\sum_{i=0}^3a_i\vartheta_i^3}
                        (Q-P)^{\otimes3}
                    =-(Q-P)^{\otimes3}.
            \end{equation*}
            This proves~\eqref{E: four_chord_cancellation}.
        
            Finally, the ordering
            $\vartheta_0<\vartheta_1<\vartheta_2<\vartheta_3$
            gives $a_0,a_2>0$ and $a_1,a_3<0$. Thus
            \begin{equation*}
                \Xi_{P,Q}
                    =\bra{a_0\delta_{\mu_{\vartheta_0}}
                           +a_2\delta_{\mu_{\vartheta_2}}}
                     -\bra{-a_1\delta_P-a_3\delta_Q}
            \end{equation*}
            is a difference of positive measures. Its Jordan negative
            part is therefore bounded by $-a_1\delta_P-a_3\delta_Q$,
            as claimed.
        \end{proof}

        We now average the four-chord perturbations over two separated
        parameter sets of positive $\nu$-mass. Ordering the scalar
        parameters gives a common likelihood-ratio bound, and a
        strictly monotone component expectation ensures that the
        changes in the third moment do not cancel under averaging.

        \begin{proof}[Proof of Proposition~\ref{P: exponential_two_nonrigidity}]
            We must construct a probability measure on $\cP(E)$ with the
            same second moment measure as $\mathfrak G_\nu$ but a different
            third moment measure. We first average the perturbations from
            Lemma~\ref{L: four_chord} to obtain a finite signed measure
            $\Xi$ with zero total mass and zero first two moment measures.
            Averaging will also make its negative part dominated by
            $\mathfrak G_\nu$, so that a small positive multiple of $\Xi$
            can be added to $\mathfrak G_\nu$. Finally, we use a suitable
            test function to show that the third moment changes do not
            cancel under averaging.
        
            Since $\nu$ is not a point mass and is Radon on $Y$, choose
            compact sets $I,J\subseteq Y$ such that
            \begin{equation*}
                \max I<\min J,
                \quad \nu(I)>0,
                \quad\text{ and }\quad \nu(J)>0.
            \end{equation*}
            The partition function $Z$ is decreasing because $T\geq0$.
            For $a\in I$ and $b\in J$, we therefore have
            \begin{equation*}
                \frac{dP_b}{dP_a}(x)
                    =\frac{Z(a)}{Z(b)}e^{-(b-a)T(x)}
                    \leq\frac{Z(\min I)}{Z(\max J)}=:M<\infty.
            \end{equation*}
            Moreover, $M>1$: nonconstancy and nonnegativity of $T$ imply
            that $T>0$ on a set of positive $m$-measure, so
            $Z(\min I)>Z(\max J)$.
        
            Choose $0<\delta<\frac{1}{M-1}$, and use the parameters
            $\vartheta_i$ and weights $a_i$ from
            Lemma~\ref{L: four_chord} with this fixed $\delta$.
            This single choice ensures that
            \begin{equation*}
                P_a+\vartheta_i(P_b-P_a)\in\cP(E)
                \quad\text{ for all }\quad
                (a,b)\in I\times J
                \text{ and }i\in\{0,1,2,3\}.
            \end{equation*}
            Thus the lemma supplies $\Xi_{P_a,P_b}$ for every pair
            $(a,b)\in I\times J$, with the same four weights.
        
            We now average these perturbations. The purpose of averaging
            is to distribute the mass removed at $P_a$ and $P_b$
            according to $\nu$ on $I$ and $J$, making it controlled by
            $\mathfrak G_\nu$. Define the probability measure
            \begin{equation*}
                \zeta:=\frac{\nu|_I}{\nu(I)}
                    \otimes\frac{\nu|_J}{\nu(J)}
            \end{equation*}
            on $I\times J$. For each $i\in\{0,1,2,3\}$, define
            $\Phi_i:I\times J\to\cP(E)$ by
            \begin{equation*}
                \Phi_i(a,b):=P_a+\vartheta_i(P_b-P_a).
            \end{equation*}
            Since $\beta\mapsto P_\beta$ is Borel, each $\Phi_i$ is Borel.
            Consequently,
            \begin{equation*}
                \Xi:=\sum_{i=0}^3a_i(\Phi_i)_*\zeta
            \end{equation*}
            is a well-defined finite signed measure on $\cP(E)$.
            Equivalently,
            \begin{equation*}
                \Xi(B)
                    =\int_{I\times J}\Xi_{P_a,P_b}(B)\,d\zeta(a,b)
                \quad\text{ for all }\quad B\in\cB(\cP(E)).
            \end{equation*}
        
            Equation~\eqref{E: four_chord_cancellation} and Fubini's
            theorem give
            \begin{equation*}
                \cM_j(\Xi)=0
                \quad\text{ for all }\quad j\in\{0,1,2\}.
            \end{equation*}
            Since $a_0,a_2>0$ and $a_1,a_3<0$, the defining sum for
            $\Xi$ expresses it as a difference of positive measures.
            Using $\Phi_1(a,b)=P_a$ and $\Phi_3(a,b)=P_b$, we obtain
            \begin{align*}
                \Xi^-
                    &\leq -a_1(\Phi_1)_*\zeta-a_3(\Phi_3)_*\zeta\\
                    &=\frac{-a_1}{\nu(I)}
                        \int_I\delta_{P_a}\,d\nu(a)
                      +\frac{-a_3}{\nu(J)}
                        \int_J\delta_{P_b}\,d\nu(b)\\
                    &\leq C\mathfrak G_\nu,
            \end{align*}
            where
            \begin{equation*}
                C:=\frac{-a_1}{\nu(I)}+\frac{-a_3}{\nu(J)}
                    \in(0,\infty).
            \end{equation*}
            The final inequality follows because each restricted integral
            is dominated by
            $\mathfrak G_\nu=\int_Y\delta_{P_\beta}\,d\nu(\beta)$.
        
            Choose $0<\varepsilon<\frac{1}{C}$ and set
            \begin{equation*}
                \Lambda_\varepsilon:=\mathfrak G_\nu+\varepsilon\Xi.
            \end{equation*}
            Then
            \begin{equation*}
                \Lambda_\varepsilon
                    \geq\mathfrak G_\nu-\varepsilon\Xi^-
                    \geq(1-\varepsilon C)\mathfrak G_\nu
                    \geq0.
            \end{equation*}
            Since $\cM_0(\Xi)=0$, the total mass of
            $\Lambda_\varepsilon$ is one. Thus it is a probability
            measure, and
            \begin{equation*}
                \cM_j(\Lambda_\varepsilon)=\cM_j(\mathfrak G_\nu)
                \quad\text{ for all }\quad j\in\{1,2\}.
            \end{equation*}
        
            It remains to show that the third moment measure changes.
            Take the bounded function $h=e^{-T}$. For $a<b$, the
            likelihood-ratio formula gives
            \begin{align*}
                P_b(h)
                    &=\frac{P_a\pa{h e^{-(b-a)T}}}
                            {P_a\pa{e^{-(b-a)T}}},\\
                P_b(h)-P_a(h)
                    &=\frac{\operatorname{Cov}_{P_a}
                            \pa{e^{-T},e^{-(b-a)T}}}
                            {P_a\pa{e^{-(b-a)T}}}.
            \end{align*}
            The denominator is positive. To show that the numerator is
            strictly positive, let $A,A'$ be independent copies of $T$
            under $P_a$. Then
            \begin{align*}
                &2\operatorname{Cov}_{P_a}
                    \pa{e^{-T},e^{-(b-a)T}}\\
                &\quad=\EE\bra{
                    \pa{e^{-A}-e^{-A'}}
                    \pa{e^{-(b-a)A}-e^{-(b-a)A'}}
                }.
            \end{align*}
            Both exponential functions are strictly decreasing, so the
            integrand is nonnegative and is strictly positive when
            $A\ne A'$. Since $P_a$ is equivalent to $m$, the
            nonconstancy of $T$ implies that this event has positive
            probability. Hence $P_b(h)-P_a(h)>0$.
        
            Applying~\eqref{E: four_chord_cancellation} to the averaged
            perturbation now gives
            \begin{align*}
                \cM_3(\Xi)(h^{\otimes3})
                    &=\int_{I\times J}
                        \cM_3(\Xi_{P_a,P_b})(h^{\otimes3})
                        \,d\zeta(a,b)\\
                    &=-\int_{I\times J}
                        \bra{P_b(h)-P_a(h)}^3\,d\zeta(a,b)
                    <0.
            \end{align*}
            The strict inequality holds because $a<b$ throughout
            $I\times J$ and $\zeta$ is a probability measure.
            Thus $\Lambda_\varepsilon$ has the same second moment
            measure as $\mathfrak G_\nu$ but a different third moment
            measure. Therefore $\mathfrak G_\nu$ is not $2$-rigid.
        \end{proof}

        Under the hypotheses of
        Corollary~\ref{C: exponential_rigidity}, Proposition~\ref{P: exponential_two_nonrigidity} gives the exact
        threshold for mixing laws that are not point masses in one
        parameter. In several parameters,
        such a mixing law can already be $2$-rigid, as the following
        example shows.
        
        \begin{example}\label{Ex: multiparameter_two_rigidity}
            Let $P_y$ be the law of two independent gamma variables with
            a fixed common shape $\alpha>0$ and rates
            $y=(y_1,y_2)\in(0,\infty)^2$. Put $P_1=P_{(1,2)}$,
            $P_2=P_{(2,1)}$, and
            \begin{equation*}
                \nu:=\frac{1}{2}\pa{\delta_{(1,2)}+\delta_{(2,1)}}.
            \end{equation*}
            We show that $\mathfrak G_\nu$ is $2$-rigid. The two laws have
            the same normalizing constant, so
            \begin{equation*}
                \frac{dP_2}{dP_1}(x)=e^{x_2-x_1},
                \quad\text{ and }\quad
                \frac{dP_1}{dP_2}(x)=e^{x_1-x_2}.
            \end{equation*}
            Both ratios are essentially unbounded.
        
            Let $\eta$ be a random probability measure satisfying
            $\EE[\eta^{\otimes2}]=\cM_2(\mathfrak G_\nu)$.
            Since $P_1,P_2$ are linearly independent, the completeness
            hypothesis of Lemma~\ref{L: moment_isometry}(a) holds.
            Its isometry $W$, with $S:=W\1_{\{(1,2)\}}$, gives
            \begin{equation*}
                \eta=SP_1+(1-S)P_2,
                \quad\text{ and }\quad
                \EE S=\EE S^2=\frac{1}{2}.
            \end{equation*}
            Here the almost sure representation follows from $W\1=\1$
            and the two values $P_1(h),P_2(h)$ of $K_h$ on the support
            of $\nu$.
        
            Nonnegativity of $\eta$ forces $0\leq S\leq1$. Indeed, for
            a realization with $S=s>1$,
            \begin{equation*}
                \frac{d\eta}{dP_1}(x)=s-(s-1)e^{x_2-x_1},
            \end{equation*}
            which is negative on a set of positive $P_1$-measure when
            $x_2-x_1$ is sufficiently large. For $s<0$,
            \begin{equation*}
                \frac{d\eta}{dP_2}(x)=(1-s)-\abs{s}e^{x_1-x_2},
            \end{equation*}
            which is negative on a set of positive $P_2$-measure when
            $x_1-x_2$ is sufficiently large. Both cases contradict $\eta$
            being a probability measure.
        
            Therefore $S(1-S)\geq0$, and
            \begin{equation*}
                \EE\bra{S(1-S)}=\EE S-\EE S^2=0.
            \end{equation*}
            Hence $S\in\{0,1\}$ almost surely, with each value having
            probability $\frac{1}{2}$. The law of $\eta$ is consequently
            $\frac{1}{2}\pa{\delta_{P_1}+\delta_{P_2}}=\mathfrak G_\nu$,
            proving $2$-rigidity.
        \end{example}
        
        This example explains where the scalar perturbation argument can
        fail in several parameters. Lemma~\ref{L: four_chord} requires a
        probability measure beyond one endpoint of the chord, whose
        nonnegativity is ensured by a one-sided likelihood-ratio bound.
        Ordering scalar parameters supplies this bound. Here the two
        parameter vectors are not coordinatewise comparable, and both
        ratios are unbounded, so the chord cannot be extended beyond
        either endpoint while remaining in $\cP(E)$.

    \subsection{Examples}\label{SS: exponential_examples}

        We first verify the standard scalar families. We then compare
        Conway--Maxwell--Poisson and discrete Gaussian statistic
        supports, explain the obstruction for Gaussian location
        families, and verify closure of the criterion under finite
        products.

        To apply the corollary, we write the family in the form
        \begin{equation*}
            \frac{dP_y}{dm}(x)
                =\frac{e^{-yT(x)}}{Z(y)},
            \quad\text{ where }\quad
            Z(y):=\int_E e^{-yT(x)}\,dm(x),
        \end{equation*}
        with $T\geq0$ and $y>0$. For Weibull laws, the choice $T(x)=x^a$
        makes the parameter multiply the statistic in the exponent.
        For Pareto laws, the same role is played by
        $T(x)=\log\pa{\frac{x}{\ell}}$. For Poisson laws with
        $0<\theta<M$, absorbing $M^n$ into the reference measure gives
        the positive parameter $y=\log\pa{\frac{M}{\theta}}$.
        These choices explain the representations in the table.
    
        Every partition function in the table is finite and positive.
        In the first five rows, $\rho$ has a strictly positive density
        on $(0,\infty)$, so $\supp\rho=[0,\infty)$. In the last two,
        every nonnegative integer has positive mass, so
        $\supp\rho=\mathbb N_0$. Thus~\eqref{E: exponential_MS} holds
        with $t_0=0$ and $\lambda_n=n$ in every row. The parameter maps
        are Borel isomorphisms onto their stated ranges, so the
        conclusion applies to arbitrary mixing laws in the usual
        parameters. The exponential and geometric families are the
        cases $a=1$ in the gamma row and $r=1$ in the negative binomial
        row. Reflecting $x$ to $1-x$ also treats beta laws with the
        first shape fixed.
    
        The following examples examine the role and limits of the
        support criterion before we verify the product construction. In particular, it gives a family whose statistic
        support contains no infinite arithmetic progression but satisfies
        the general M\"untz--Sz\'asz condition.

        \begin{example}\label{Ex: cmp_rigidity}
            On $E=\mathbb N_0$, take counting measure and
            $T(n)=\log(n!)$. The resulting family is
            \begin{equation*}
                P_y(\{n\})
                    =\frac{(n!)^{-y}}{Z(y)},
                \quad\text{ where }\quad
                Z(y):=\sum_{n=0}^{\infty}(n!)^{-y}.
            \end{equation*}
            This is the Conway--Maxwell--Poisson family with its rate
            parameter fixed at $1$ and its dispersion parameter equal
            to $y$; see~\cite{ShmueliEtAl2005}. For every $y>0$, the
            ratio of successive summands is $(n+1)^{-y}\to0$, so
            $Z(y)$ is finite and positive. The statistic measure is
            \begin{equation*}
                \rho=2\delta_0+\sum_{n=2}^{\infty}
                    \delta_{\log(n!)}.
            \end{equation*}
            Set
            \begin{equation*}
                t_0:=0,
                \quad\text{ and }\quad
                \lambda_n:=\log\pa{(n+1)!}
                \quad\text{ for all } \quad n\geq0.
            \end{equation*}
            Then $\lambda_0=0$ and
            $\lambda_{n+1}-\lambda_n=\log(n+2)\geq\log2$.
            Moreover,
            \begin{equation*}
                \frac{1}{\lambda_n}
                    \geq\frac{1}{(n+1)\log(n+1)}
                \quad\text{ for all } \quad n\geq1,
            \end{equation*}
            so the reciprocal sum diverges.
            Thus~\eqref{E: exponential_MS} holds.
    
            Suppose that $\supp\rho$ contained an infinite arithmetic
            progression with initial point $a\geq0$ and step $\delta>0$:
            \begin{equation*}
                a+k\delta\in\supp\rho
                \quad\text{ for all } \quad k\geq0.
            \end{equation*}
            The gap in $\supp\rho$ following each $a+k\delta$ would be
            at most $\delta$. For $n\geq2$, however, the gap following
            $\log(n!)$ is $\log(n+1)\to\infty$, a contradiction.
            Corollary~\ref{C: exponential_rigidity},
            using the general sequence criterion, gives $3$-rigidity
            for every mixing law on $(0,\infty)$.
            Proposition~\ref{P: exponential_two_nonrigidity} shows that
            this order is exact whenever the mixing law is not a point
            mass.
        \end{example}
    
        The CMP support grows as $n\log n$, whereas the following
        discrete Gaussian support grows as $n^2$; both have gaps
        tending to infinity and finite partition functions, but their
        reciprocal sums diverge and converge, respectively.
    
        \begin{example}\label{Ex: discrete_gaussian_scope}
            On $E=\mathbb Z$, take counting measure and $T(n)=n^2$.
            Then
            \begin{equation*}
                P_y(\{n\})=\frac{e^{-yn^2}}{Z(y)},
                \quad\text{ where }\quad
                Z(y)=1+2\sum_{n=1}^{\infty}e^{-yn^2}<\infty
                \quad\text{ for all } \quad y>0.
            \end{equation*}
            These are centered discrete Gaussian laws. Their statistic
            measure is
            \begin{equation*}
                \rho=\delta_0+2\sum_{n=1}^{\infty}\delta_{n^2}.
            \end{equation*}
            Suppose that a strictly increasing sequence starting at
            $\lambda_0=0$ satisfied
            \begin{equation*}
                t_0+\lambda_n\in\supp\rho
                \quad\text{ for all } \quad n\geq0.
            \end{equation*}
            Write $t_0=k_0^2$ and $t_0+\lambda_n=k_n^2$, where the
            $k_n$ are strictly increasing nonnegative integers.
            Then $k_n\geq k_0+n$, so
            \begin{equation*}
                \lambda_n=(k_n-k_0)(k_n+k_0)\geq n^2
                \quad\text{ for all } \quad n\geq1.
            \end{equation*}
            Hence $\sum_{n\geq1}\lambda_n^{-1}<\infty$, and no such
            sequence satisfies~\eqref{E: exponential_MS}.
            Since $\rho(\{0\})=1$ and
            \begin{equation*}
                \sum_{n=1}^{\infty}\frac{n^2}{1+n^4}<\infty,
            \end{equation*}
            Corollary~\ref{C: exponential_completeness_atom} also shows
            that completeness fails for some mixing law on
            $(0,\infty)$. Thus this family is outside the scope of the
            rigidity corollary. We do not know whether
            $\fG_\nu$ is $3$-rigid for every mixing law
            $\nu\in\cP((0,\infty))$ over this family.
        \end{example}
    
        Nonnegativity of the statistic is a separate restriction.
        For Gaussian location laws $P_\mu=N(\mu,\sigma^2)$ with
        $\sigma^2>0$ fixed, the natural statistic is $x$, which takes
        both signs. More concretely, for $\mu\ne\mu'$,
        \begin{equation*}
            \frac{dP_{\mu'}}{dP_\mu}(x)
                =\exp\pa{
                    \frac{\mu'-\mu}{\sigma^2}x
                    +\frac{\mu^2-(\mu')^2}{2\sigma^2}}.
        \end{equation*}
        This ratio and its reciprocal are essentially unbounded.
        In a scalar representation with $T\geq0$, the ordering $y<y'$
        would give
        $\frac{dP_{y'}}{dP_y}\leq\frac{Z(y)}{Z(y')}$, as in the proof of
        Proposition~\ref{P: exponential_two_nonrigidity}; neither
        ordering of two distinct Gaussian means can satisfy this bound.
        Thus Corollary~\ref{C: exponential_rigidity} covers Gaussian
        laws with fixed mean and varying variance, as in the first row
        of Table~\ref{Tab: exponential_families}, but does not cover
        Gaussian laws with fixed variance and varying mean. The latter
        is the location family in the motivating Gaussian result of
        the introduction. The obstruction concerns this criterion,
        rather than rigidity of Gaussian location mixtures.
    
        Finally, consider finite products of kernels from
        Table~\ref{Tab: exponential_families} or
        Example~\ref{Ex: cmp_rigidity}. Let their reference measures
        and statistics be $m_i,T_i$. On $E=\prod_{i=1}^dE_i$, set
        \begin{equation*}
            m:=\bigotimes_{i=1}^d m_i,
            \quad\text{ and }\quad
            T(x):=(T_1(x_1),\ldots,T_d(x_d)).
        \end{equation*}
        Tonelli's theorem gives
        \begin{equation*}
            Z(y)=\prod_{i=1}^dZ_i(y_i),
            \quad
            \rho=\bigotimes_{i=1}^d\rho_i,
            \quad\text{ and }\quad
            P_y=\bigotimes_{i=1}^dP_{i,y_i}.
        \end{equation*}
        The product of the support sequences for the factors verifies
        \eqref{E: exponential_MS}. Therefore $\fG_\nu$ is
        $3$-rigid for every joint mixing law
        $\nu\in\cP((0,\infty)^d)$, including laws with arbitrary
        dependence among the parameters. In particular, three vector
        observations determine the entire exchangeable law for mixtures
        of fixed-mean Gaussian distributions with positive diagonal
        covariance matrices, even when the latent variances are dependent.

\section{Spherical symmetry}\label{S: spherical_proofs}
 
    This section proves Theorem~\ref{T: three-point-rigidity} and
    Proposition~\ref{P: sharpness}. Section~\ref{SS: spherical_pair} gives
    a direct proof of the pair-existence criterion from the pair law,
    and Remark~\ref{R: tail_conditioning} explains why the usual
    tail-conditioning argument does not apply in dimension two.

    Subsection~\ref{SS: proof_three_point} specializes
    Corollary~\ref{C: exponential_rigidity} to the Gaussian scale family,
    including an atom at zero variance. Section~\ref{SS: proof_sharpness}
    combines Proposition~\ref{P: exponential_two_nonrigidity} with an
    argument at the zero-variance endpoint.
    Section~\ref{SS: isotropic_stable} applies these thresholds to
    isotropic symmetric stable laws.

    \subsection{The spherical pair}\label{SS: spherical_pair}

        We identify the Gaussian scale-mixture candidate directly from
        the characteristic function of the pair.

        The applications in this and the following section share the same
        analytic step. Spherical or additive structure produces a continuous
        function $F$ for which $(s,t)\mapsto F(s+t)$ is positive semidefinite.
        Such a function has a bilateral Laplace representation; boundedness
        restricts the representing measure to $[0,\infty)$, while
        normalization makes it a probability measure. In the spherical
        setting, the resulting representation is precisely a Gaussian
        scale mixture. The de Finetti representation will give the required
        positivity through an expectation of a square. We first record the
        required representation theorem.

        \begin{lemma}\label{L: exponential_convex_representation}
            Let $F:[0,\infty)\to\R$ be continuous. Suppose that
            \begin{equation*}
                \bra{F\pa{s_j+s_k}}_{j,k=1}^n
            \end{equation*}
            is positive semidefinite for every $n\geq1$ and every
            $s_1,\ldots,s_n\geq0$. Then there is a unique finite positive
            measure $\mu$ on $\R$ such that
            \begin{equation}\label{E: exponential_convex_representation}
                F\pa{t}=\int_\R e^{-\beta t}\,d\mu\pa{\beta},
                \quad t\geq0.
            \end{equation}
            If $F$ is bounded, then $\mu$ is supported in $[0,\infty)$
            and $F$ is completely monotone on $(0,\infty)$:
            \begin{equation*}
                (-1)^nF^{(n)}(t)\geq0,
                \quad n\geq0,\quad t>0.
            \end{equation*}
            If, in addition, $F(0)=1$, then $\mu$ is a probability measure.
        \end{lemma}

        \begin{proof}
            The exponentially convex representation theorem
            \cite[Chapter~VI, Sections~20--21]{Widder} gives a unique
            positive measure $\mu$ on $\R$ such that
            \begin{equation*}
                F(t)=\int_\R e^{-\beta t}\,d\mu(\beta),
                \quad t>0.
            \end{equation*}
            Fatou's lemma as $t\downarrow0$ first gives
            $\mu(\R)\leq F(0)<\infty$. On $[0,\infty)$, monotone
            convergence applies as $t\downarrow0$; on $(-\infty,0)$,
            dominated convergence applies after fixing $t_0>0$, since
            $e^{-\beta t}\leq e^{-\beta t_0}$ for $0<t\leq t_0$.
            Continuity of $F$ at zero therefore gives $\mu(\R)=F(0)$,
            and the representation extends to $t=0$.

            Suppose that $F$ is bounded. If
            $\mu\pa{(-\infty,-\varepsilon]}>0$ for some $\varepsilon>0$,
            then
            \begin{equation*}
                F(t)\geq e^{\varepsilon t}
                    \mu\pa{(-\infty,-\varepsilon]},
            \end{equation*}
            contradicting boundedness as $t\to\infty$. Thus $\mu$ is
            supported in $[0,\infty)$. Differentiation under the integral
            for $t>0$ gives complete monotonicity, and
            $\mu(\R)=F(0)$ proves the last assertion.
        \end{proof}

        Let $\eta$ be the directing measure of an infinite exchangeable
        sequence $X=(X_i)_{i=1}^\infty$, and write
        \begin{equation*}
            \varphi_\eta(t)=\int_\R e^{itx}\,d\eta(x).
        \end{equation*}
        Conditional iid sampling gives
        \begin{equation}\label{E: definetti_characteristic}
            \EE\bra{\exp\pa{i\sum_{j=1}^kt_jX_j}}
            =\EE\bra{\prod_{j=1}^k\varphi_\eta(t_j)}.
        \end{equation}

        \begin{lemma}\label{L: spherical_pair}
            If $(X_1,X_2)$ is spherically symmetric, then $\eta$ is
            symmetric almost surely and there is a unique probability
            measure $\Pi$ on $[0,\infty)$ such that
            \begin{equation}\label{E: spherical_pair_mixture}
                \cL\pa{X_1,X_2}
                =\int_{[0,\infty)}\gamma_v^{\otimes2}\,d\Pi(v).
            \end{equation}
            Thus the pair has a unique rotatable extension. The given
            exchangeable extension need not coincide with this extension.
        \end{lemma}

        \begin{proof}
            Invariance under $-I\in O(2)$ makes the characteristic function
            of the pair real-valued. Spherical symmetry and
            Equation~\eqref{E: definetti_characteristic} therefore give a
            continuous function $F:[0,\infty)\to\R$ such that
            \begin{equation}\label{E: pair_radial}
                \EE\bra{\varphi_\eta(s)\varphi_\eta(t)}=F(s^2+t^2).
            \end{equation}
            In particular, $F(0)=1$ and $\abs{F}\leq1$. Since
            $\varphi_\eta(-t)=\overline{\varphi_\eta(t)}$,
            Equation~\eqref{E: pair_radial} gives
            \begin{align*}
                \EE\bra{\abs{\varphi_\eta(t)-\varphi_\eta(-t)}^2}
                &=2\EE\bra{\varphi_\eta(t)\varphi_\eta(-t)}
                  -\EE\bra{\varphi_\eta(t)^2}
                  -\EE\bra{\varphi_\eta(-t)^2}\\
                &=0.
            \end{align*}
            Intersecting over rational $t$ and then using continuity shows
            that $\varphi_\eta$ is real and even almost surely. Hence
            $\eta$ is symmetric almost surely.

            Put $r_\eta(u)=\varphi_\eta(\sqrt u)$ for $u\geq0$.
            For $u_1,\ldots,u_m\geq0$ and $c_1,\ldots,c_m\in\R$,
            \begin{equation*}
                \sum_{j,k=1}^m c_jc_kF(u_j+u_k)
                =\EE\bra{\pa{\sum_{j=1}^m c_jr_\eta(u_j)}^2}
                \geq0.
            \end{equation*}
            Lemma~\ref{L: exponential_convex_representation} gives a unique
            probability measure $\lambda$ on $[0,\infty)$ satisfying
            \begin{equation*}
                F(u)=\int_{[0,\infty)}e^{-a u}\,d\lambda(a).
            \end{equation*}
            Let $\Pi$ be the image of $\lambda$ under $a\mapsto2a$.
            Then
            \begin{equation*}
                F(s^2+t^2)
                =\int_{[0,\infty)}
                    \exp\pa{-\frac v2(s^2+t^2)}\,d\Pi(v),
            \end{equation*}
            which is the characteristic function of the law in
            Equation~\eqref{E: spherical_pair_mixture}. Uniqueness follows
            from Lemma~\ref{L: exponential_convex_representation}.
            Sampling conditionally iid from $\gamma_V$, where $V\sim\Pi$,
            gives a rotatable extension; Freedman's theorem
            \cite{freedman62} (see also \cite[Theorem~1.31]{kallenberg05}) shows that every rotatable extension has this
            form. Thus the rotatable extension is unique.
        \end{proof}

        \begin{corollary}\label{C: schoenberg_two}
            A spherically symmetric probability measure on $\R^2$ is
            the law of $\pa{X_1,X_2}$ for some infinite exchangeable
            sequence $X$ if and only if its characteristic function is
            \begin{equation*}
                u\longmapsto\int_{[0,\infty)}e^{-\frac{v}{2}\norm{u}^2}\,d\Pi(v)
            \end{equation*}
            for some probability measure $\Pi$ on $[0,\infty)$.
            In this case $\Pi$ is unique.
        \end{corollary}
        \begin{proof}
            Necessity and uniqueness of $\Pi$ follow from
            Lemma~\ref{L: spherical_pair}. For sufficiency, let
            $V\sim\Pi$ be independent of an iid standard Gaussian
            sequence $G$. Then $\pa{\sqrt VG_i}_{i=1}^\infty$ is an
            exchangeable sequence with the prescribed pair law.
        \end{proof}

        The same criterion holds on $\R^d$ for every $d\geq2$:
        spherical symmetry makes the full characteristic function
        $F(\norm{u}^2)$ depend on the same function $F$ as its pair
        marginal. Two is sharp. In dimension one, any symmetric law
        extends to a constant exchangeable sequence, whereas, for
        example, $\frac12(\delta_{-1}+\delta_1)$ is not a Gaussian
        scale mixture.

        \begin{remark}\label{R: tail_conditioning}
            The tail-conditioning proof of the extendibility criterion in
            \cite[Theorem~3.3]{mai20extendibility} uses invariance of the
            joint law of the initial block and the remaining sequence under
            rotations of the block alone. For a spherical pair this joint
            invariance can fail. If it always held, the argument would show
            that every exchangeable extension of a spherical pair is
            rotatable, which Proposition~\ref{P: sharpness} rules out. For
            blocks of length at least three, the joint invariance does hold,
            because Theorem~\ref{T: three-point-rigidity} makes the sequence
            rotatable. Among dimensions at least two, the gap therefore concerns exactly the
            two-dimensional case. The criterion itself remains true in
            dimension two, and Lemma~\ref{L: spherical_pair} proves it
            directly from the pair law.
        \end{remark}

    \subsection{From a spherical triple to the full sequence}
        \label{SS: proof_three_point}

        Specialize the exponential family to
        \begin{equation*}
            E=\R,\quad m(dx)=dx,\quad T(x)=x^2.
        \end{equation*}
        In this case
        \begin{equation*}
            Z(\beta)=\sqrt{\frac{\pi}{\beta}},
            \quad
            P_\beta=\mathcal N\pa{0,\frac{1}{2\beta}},
            \quad \beta>0.
        \end{equation*}

        \begin{corollary}\label{C: gaussian_three_rigidity}
            Let $\Pi\in\cP([0,\infty))$, and let $\eta$ be a random
            probability measure on $\R$. If
            \begin{equation}\label{E: gaussian_third_moment_rigidity}
                \EE\bra{\eta^{\otimes3}}
                =\int_{[0,\infty)}\gamma_v^{\otimes3}\,d\Pi(v),
            \end{equation}
            then there is a random variable $V\sim\Pi$, on the probability
            space carrying $\eta$, such that $\eta=\gamma_V$ almost surely.
        \end{corollary}

        \begin{proof}
            Put $p_0:=\Pi(\{0\})$ and $U:=\eta(\{0\})$. The first
            two marginals of \eqref{E: gaussian_third_moment_rigidity}
            give
            \begin{equation*}
                \EE U=\EE U^2=p_0.
            \end{equation*}
            Since $0\leq U\leq1$, we have
            $\EE[U(1-U)]=0$, so $U\in\{0,1\}$ almost surely and
            $\PP(U=1)=p_0$. On $\{U=1\}$, $\eta=\delta_0$.
            If $p_0=1$, take $V=0$.

            Suppose that $p_0<1$ and put
            $\Pi_+:=\frac{\Pi|_{(0,\infty)}}{1-p_0}$.
            Subtracting the contribution on $\{U=1\}$ yields
            \begin{equation*}
                \EE[\eta^{\otimes3}\mid U=0]
                =\int_{(0,\infty)}\gamma_v^{\otimes3}\,d\Pi_+(v).
            \end{equation*}
            Let $\nu$ be the image of $\Pi_+$ under
            $v\mapsto\frac{1}{2v}$. The exponential kernel with
            $T(x)=x^2$ has statistic measure $t^{-1/2}\,dt$ and
            satisfies \eqref{E: exponential_MS}. Corollary~\ref{C: exponential_rigidity}
            and the realization in Theorem~\ref{T: general_rigidity}
            therefore give, on the conditional probability space
            $\{U=0\}$, a random parameter $A\sim\nu$ with
            $\eta=P_A=\gamma_{1/(2A)}$ almost surely.
            Set $V=\frac{1}{2A}$ on $\{U=0\}$ and $V=0$ on
            $\{U=1\}$, assigning zero on any remaining null set.
            Then $V$ is measurable, has law $\Pi$, and
            $\eta=\gamma_V$ almost surely.
        \end{proof}

        \begin{proof}[Proof of Theorem~\ref{T: three-point-rigidity}]
            Let $\eta$ be the directing measure of $X$. Spherical symmetry
            of the triple and Equation~\eqref{E: definetti_characteristic}
            give a continuous function $F:[0,\infty)\to\R$ such that
            \begin{equation}\label{E: spherical_triple_generator}
                \EE\bra{\prod_{j=1}^3\varphi_\eta(t_j)}
                =F(t_1^2+t_2^2+t_3^2).
            \end{equation}
            Taking $t_3=0$ and applying Lemma~\ref{L: spherical_pair}
            gives a unique probability measure $\Pi$ on $[0,\infty)$
            satisfying
            \begin{equation*}
                F(u)=\int_{[0,\infty)}e^{-vu/2}\,d\Pi(v),
                \quad u\geq0.
            \end{equation*}
            Substituting this representation into
            Equation~\eqref{E: spherical_triple_generator} shows that
            \begin{equation*}
                \cL\pa{X_1,X_2,X_3}
                =\int_{[0,\infty)}\gamma_v^{\otimes3}\,d\Pi(v).
            \end{equation*}
            Corollary~\ref{C: gaussian_three_rigidity} now gives a random
            variable $V\sim\Pi$ such that $\eta=\gamma_V$ almost surely.
            Conditional iid sampling gives
            \begin{equation*}
                X\stackrel d=\pa{\sqrt V G_i}_{i=1}^\infty,
            \end{equation*}
            where $G_1,G_2,\ldots$ are iid standard Gaussian random
            variables independent of $V$. Thus $X$ is rotatable.
        \end{proof}

    \subsection{When two coordinates suffice}\label{SS: proof_sharpness}

        We next determine when the spherical pair already identifies its
        exchangeable extension. The atom at zero is always determined by
        the pair, so all possible nonuniqueness lies in the
        positive-variance component. There,
        Proposition~\ref{P: exponential_two_nonrigidity} applies to the
        Gaussian variance family.

        \begin{proof}[Proof of Proposition~\ref{P: sharpness}]
            Let $\eta$ direct an arbitrary exchangeable extension of the
            pair in Proposition~\ref{P: sharpness}. Put
            \begin{equation*}
                p_0=\Pi(\{0\}),\quad U=\eta(\{0\}).
            \end{equation*}
            The first two moment measures give $\EE U=\EE U^2=p_0$, and
            hence $\EE[U(1-U)]=0$. Thus $U\in\{0,1\}$ almost surely,
            $\PP(U=1)=p_0$, and $\eta=\delta_0$ on $\{U=1\}$.
            If $p_0=1$, this proves uniqueness.

            Suppose that $p_0<1$, and let
            \begin{equation*}
                \Pi_+=\frac{\Pi|_{(0,\infty)}}{1-p_0}.
            \end{equation*}
            Subtracting the endpoint contribution gives
            \begin{equation*}
                \EE\bra{\eta^{\otimes2}\mid U=0}
                =\int_{(0,\infty)}\gamma_v^{\otimes2}\,d\Pi_+(v).
            \end{equation*}
            If $\Pi_+=\delta_{v_0}$, Lemma~\ref{L: product_law_rigidity},
            applied conditionally on $U=0$, gives $\eta=\gamma_{v_0}$ on
            this event. Hence the rotatable extension is unique.

            Suppose instead that $\Pi_+$ is not a point mass, and let $\nu$
            be its image under $v\mapsto(2v)^{-1}$, so that
            \begin{equation*}
                \fG_\nu
                =\int_{(0,\infty)}\delta_{\gamma_v}\,d\Pi_+(v).
            \end{equation*}
            The Gaussian variance family of
            Section~\ref{SS: proof_three_point} satisfies the hypotheses of
            Proposition~\ref{P: exponential_two_nonrigidity}, since
            $T(x)=x^2$ is not almost everywhere constant. That proposition
            gives a probability measure $\Lambda_+\neq\fG_\nu$ on
            $\cP(\R)$ with $\cM_2(\Lambda_+)=\cM_2(\fG_\nu)$. Put
            \begin{equation*}
                \Lambda=p_0\delta_{\delta_0}+(1-p_0)\Lambda_+.
            \end{equation*}
            Conditionally iid sampling from a directing measure with law
            $\Lambda$ gives an exchangeable sequence with the prescribed
            pair law. Its directing law differs from
            $p_0\delta_{\delta_0}+(1-p_0)\fG_\nu$, which is the
            directing law of the unique rotatable extension given by
            Lemma~\ref{L: spherical_pair}. Thus the new extension is not
            rotatable.
        \end{proof}

    \subsection{Isotropic stable laws}\label{SS: isotropic_stable}

        The isotropic stable laws give a canonical nondegenerate instance
        of the preceding threshold. Their Gaussian mixing variance is
        positive stable for $0<\alpha<2$ and deterministic at the Gaussian
        endpoint $\alpha=2$.

        \begin{corollary}\label{C: isotropic_app}
            Let $0<\alpha\leq2$ and $k\geq1$, and let $\mu_{\alpha,k}$ be
            the isotropic symmetric $\alpha$-stable law on $\R^k$ with
            characteristic function $u\mapsto e^{-\norm{u}^\alpha}$. If
            $0<\alpha<2$, then $\mu_{\alpha,k}$ has a unique infinite
            exchangeable extension if and only if $k\geq3$. If $\alpha=2$,
            it has a unique infinite exchangeable extension if and only if
            $k\geq2$.
        \end{corollary}

        \begin{proof}
            Put $p=\alpha/2$. For $0<\alpha<2$, let $A_\alpha$ be the
            positive strictly $p$-stable random variable characterized by
            \begin{equation}\label{E: positive_stable_laplace}
                \EE\bra{e^{-sA_\alpha}}=e^{-s^p},\quad s\geq0;
            \end{equation}
            see \cite[Chapter~XIII, Section~6]{Feller}. At $\alpha=2$,
            put $A_2=1$. If $G_1,G_2,\ldots$ are iid standard Gaussian
            random variables independent of $A_\alpha$, then
            \begin{align*}
                \EE\bra{\exp\pa{i\sum_{j=1}^d
                    t_j\sqrt{2A_\alpha}G_j}}
                &=\EE\bra{\exp\pa{-A_\alpha\sum_{j=1}^dt_j^2}}\\
                &=\exp\pa{-\pa{\sum_{j=1}^dt_j^2}^{\alpha/2}}.
            \end{align*}
            Thus $(\sqrt{2A_\alpha}G_i)_{i=1}^\infty$ is a rotatable
            extension of $\mu_{\alpha,k}$ for every $k$. This is the
            classical Gaussian-mixture representation of an isotropic
            stable vector; see \cite{CambanisHuangSimons} for the general
            elliptical framework.

            If the first triple of an infinite exchangeable sequence has
            law $\mu_{\alpha,3}$, Corollary~\ref{C: gaussian_three_rigidity}
            identifies its directing measure as $\gamma_V$, where $V$ has
            the law of $2A_\alpha$. Its $d$-dimensional marginal is
            therefore $\mu_{\alpha,d}$ for every $d\geq1$. By
            marginalization, uniqueness holds whenever $k\geq3$.

            When $0<\alpha<2$, Equation~\eqref{E: positive_stable_laplace}
            shows that $A_\alpha$ is nondegenerate, while
            \begin{equation*}
                \PP(A_\alpha=0)
                =\lim_{s\to\infty}\EE\bra{e^{-sA_\alpha}}=0.
            \end{equation*}
            Its positive support therefore contains at least two points.
            Proposition~\ref{P: sharpness} gives a nonrotatable extension
            of $\mu_{\alpha,2}$, proving failure of uniqueness at order
            two. At $\alpha=2$, the pair law is $\gamma_2^{\otimes2}$,
            so Lemma~\ref{L: product_law_rigidity} gives uniqueness from
            two coordinates.

            Finally, one coordinate never suffices. If $Z\sim\mu_{\alpha,1}$,
            the constant sequence $X_i=Z$ is exchangeable and has the
            required first marginal. It differs from the rotatable
            extension above: in that extension $A_\alpha>0$ almost surely
            and $\PP(X_1=X_2)=0$, whereas the constant sequence has
            $\PP(X_1=X_2)=1$.
        \end{proof}

\section{\texorpdfstring{Radial laws and $\ell_q$-norm exponential laws}{Radial laws and norm-exponential laws}}
\label{S: radial_applications}
 
    The spherical argument of Section~\ref{S: spherical_proofs} extracts a
    Gaussian mixture from orthogonal symmetry. We now treat laws whose
    densities depend on the observations only through the total energy
    $\sum_{j=1}^dT(x_j)$. The pattern is the same as in
    Section~\ref{S: spherical_proofs}: the prescribed marginal determines
    a unique exponential-mixture candidate, and three coordinates
    determine the sequence.
 
    The new step is to recover positivity of the two-variable kernel
    $F(s+t)$ from a single $d$-dimensional marginal. If $\eta$ directs an
    exchangeable extension, the $d$-dimensional law gives
    $\EE[\eta(f)^2\eta(u)^{d-2}]\geq0$ for bounded test functions $f$ and
    $u\geq0$. Concentrating $u$ near zero energy removes the last $d-2$
    coordinates and leaves the kernel $F(T(x)+T(y))$. The Laplace
    representation of Section~\ref{SS: spherical_pair} then gives an
    exponential mixture, and the results of
    Section~\ref{S: exponential_proofs} decide whether its exchangeable
    extension is unique.
 
    Section~\ref{SS: energy_radial} proves the general criterion.
    Section~\ref{SS: lq_laws} applies it to $\ell_q$-norm exponential
    densities. It recovers Mai's conditionally iid construction for
    $q\geq1$ \cite[Corollary~3.5]{mai20norm}, rules out extendibility for
    $0<q<1$, and identifies the exact uniqueness threshold.
 
    \subsection{Extendibility of energy-radial laws}\label{SS: energy_radial}
 
        We prove Theorem~\ref{T: energy_radial_extendibility}.
        Only the prescribed marginal $Q_d$ is assumed to have the
        stated form; its exchangeable extensions are otherwise unrestricted. For the related classical representation theory of $\ell_1$-norm symmetric distributions and Archimedean copulas, including the results of Williamson and Kimberling, see also \cite{mcneilneslehova09}.

        Once positive semidefiniteness of $F(s+t)$ has been established,
        the passage to a Laplace representation is classical; see the
        representation theory of Berg, Christensen, and Ressel
        \cite[Chapters~3--6]{BergChristensenRessel} and Ressel's analytical
        approach to de Finetti-type theorems \cite{ressel85}.  The issue
        here is how that positivity can be recovered from a single
        finite-dimensional marginal. The next two lemmas carry out the
        localization, and the third uses normalization of $Q_d$ to
        exclude nonpositive rates.
 
        \begin{lemma}\label{L: bottom_energy_localization}
            Let $d\geq2$, suppose that $Z(1)<\infty$ and
            $\operatorname{supp}(T_*m)=[0,\infty)$, and let
            $F:[0,\infty)\to[0,\infty)$ be continuous. If a random
            probability measure $\eta$ on $E$ satisfies
            \begin{equation*}
                \EE\bra{\eta^{\otimes d}}(dx)
                =F\pa{\sum_{j=1}^dT(x_j)}\,dm^{\otimes d}(x),
            \end{equation*}
            then, for every bounded real-valued Borel function $f$ that
            vanishes outside $\{T\leq M\}$ for some $0\leq M<\infty$,
            \begin{equation*}
                \int_{E^2}f(x)f(y)F\pa{T(x)+T(y)}\,dm(x)dm(y)\geq0.
            \end{equation*}
        \end{lemma}
 
        \begin{proof}
            Every energy sublevel set has finite $m$-measure. Indeed, on
            $\{T\leq M\}$ we have $e^{-T}\geq e^{-M}$, and hence
            \begin{equation*}
                m(T\leq M)\leq e^MZ(1)<\infty.
            \end{equation*}
            The integral in the conclusion is therefore absolutely
            convergent: $F$ is bounded on $[0,2M]$, while $f$ is bounded
            and vanishes outside $\{T\leq M\}$. Write $J$ for this
            integral.
 
            If $d=2$, the moment identity gives $J=\EE[\eta(f)^2]\geq0$
            directly. Suppose that $d\geq3$. Full support of $\rho=T_*m$
            and the sublevel estimate give
            $0<m(T\leq n^{-1})<\infty$ for $n\geq1$, so
            \begin{equation*}
                u_n(x):=\frac{\1_{\{T\leq n^{-1}\}}(x)}
                                 {m(T\leq n^{-1})}
            \end{equation*}
            is a bounded nonnegative function, and $u_n\,dm$ is a
            probability measure concentrated on $\{T\leq n^{-1}\}$. The
            moment identity gives
            \begin{equation*}
                0\leq\EE\bra{\eta(f)^2\eta(u_n)^{d-2}}
                =\int_{E^d}f(x_1)f(x_2)\prod_{j=3}^du_n(x_j)
                    F\pa{\sum_{j=1}^dT(x_j)}\,dm^{\otimes d}(x)
                =:J_n.
            \end{equation*}
            On the support of the integrand defining $J_n$,
            \begin{equation*}
                0\leq\sum_{j=3}^dT(x_j)\leq\frac{d-2}{n}.
            \end{equation*}
            Since $u_n\,dm$ is a probability measure,
            \begin{equation*}
                \abs{J_n-J}
                \leq\norm{f}_\infty^2m(T\leq M)^2
                    \sup_{\substack{0\leq r\leq2M\\
                                    0\leq h\leq(d-2)/n}}
                    \abs{F(r+h)-F(r)}.
            \end{equation*}
            Uniform continuity of $F$ on $[0,2M+d-2]$ makes the right-hand
            side tend to zero. Thus $J_n\to J$, and $J\geq0$.
        \end{proof}
 
        \begin{lemma}\label{L: energy_space_localization}
            Let $F:[0,\infty)\to\R$ be continuous, and suppose that
            $\rho=T_*m$ is locally finite and has full support on
            $[0,\infty)$. If
            \begin{equation*}
                \int_{E^2}f(x)f(y)F\pa{T(x)+T(y)}\,dm(x)dm(y)\geq0
            \end{equation*}
            for every bounded real-valued Borel function $f$ that vanishes
            outside an energy sublevel set, then the kernel
            $(s,t)\mapsto F(s+t)$ is positive semidefinite on
            $[0,\infty)$.
        \end{lemma}
 
        \begin{proof}
            Fix $s_1,\ldots,s_k\geq0$ and $c_1,\ldots,c_k\in\R$. For
            $\varepsilon>0$, the sets
            \begin{equation*}
                I_{j,\varepsilon}
                :=(s_j-\varepsilon,s_j+\varepsilon)\cap[0,\infty)
            \end{equation*}
            are bounded, nonempty and relatively open, so full support and
            local finiteness give $0<\rho(I_{j,\varepsilon})<\infty$. The
            function
            \begin{equation*}
                f_\varepsilon(x)
                :=\sum_{j=1}^k
                    \frac{c_j\1_{\{T\in I_{j,\varepsilon}\}}(x)}
                         {\rho(I_{j,\varepsilon})}
            \end{equation*}
            is bounded, Borel, and vanishes outside an energy sublevel set.
            The assumed inequality and $\rho=T_*m$ give
            \begin{equation*}
                0\leq\sum_{j,\ell=1}^k c_jc_\ell
                    \frac{\displaystyle
                        \int_{I_{j,\varepsilon}}\int_{I_{\ell,\varepsilon}}
                        F(r+t)\,d\rho(r)d\rho(t)}
                         {\rho(I_{j,\varepsilon})\rho(I_{\ell,\varepsilon})}.
            \end{equation*}
            Each fraction is an average of $F(r+t)$ over points with
            $\abs{r+t-s_j-s_\ell}<2\varepsilon$, so by continuity it tends
            to $F(s_j+s_\ell)$ as $\varepsilon\downarrow0$. Passing to the
            limit in the finite sum gives
            $\sum_{j,\ell}c_jc_\ell F(s_j+s_\ell)\geq0$.
        \end{proof}
 
        \begin{lemma}\label{L: canonical_laplace_support}
            Suppose that Equation~\eqref{E: critical_parameter_intro}
            holds. Let $d\geq1$, and let $\mu$ be a finite positive Borel
            measure on $\R$ such that
            \begin{equation*}
                F(t)=\int_\R e^{-\beta t}\,d\mu(\beta),
                \quad t\geq0.
            \end{equation*}
            If $F(\sum_{j=1}^dT(x_j))$ is the $m^{\otimes d}$-density of
            a probability measure $Q_d$, then $\mu((-\infty,0])=0$,
            Equation~\eqref{E: radial_mixing_law} defines a probability
            measure $\nu_d$, and
            $Q_d=\int_{(0,\infty)}P_\beta^{\otimes d}\,d\nu_d(\beta)$.
        \end{lemma}
 
        \begin{proof}
            All integrands below are nonnegative, so Tonelli's theorem
            applies without any prior integrability assumption. Since
            $Q_d$ has total mass one,
            \begin{equation*}
                1
                =\int_{E^d}F\pa{\sum_{j=1}^dT(x_j)}\,dm^{\otimes d}(x)
                =\int_\R\int_{E^d}
                    e^{-\beta\sum_{j=1}^dT(x_j)}
                    \,dm^{\otimes d}(x)\,d\mu(\beta)
                =\int_\R Z(\beta)^d\,d\mu(\beta).
            \end{equation*}
            By hypothesis, $Z(\beta)=\infty$ for every $\beta\leq0$.
            Hence $\mu((-\infty,0])=0$, and $\nu_d$ is a probability
            measure on $(0,\infty)$. The mixture
            $\int P_\beta^{\otimes d}\,d\nu_d(\beta)$ has
            $m^{\otimes d}$-density
            \begin{equation*}
                \int_{(0,\infty)}
                    \frac{e^{-\beta\sum_{j=1}^dT(x_j)}}{Z(\beta)^d}
                    \,d\nu_d(\beta)
                =F\pa{\sum_{j=1}^dT(x_j)},
            \end{equation*}
            which proves the asserted identity.
        \end{proof}
 
        \begin{proof}[Proof of Theorem~\ref{T: energy_radial_extendibility}]
            Suppose first that $Q_d$ has an infinite exchangeable extension,
            and let $\eta$ be its directing measure, so that
            $\EE[\eta^{\otimes d}]=Q_d$.
            Lemmas~\ref{L: bottom_energy_localization}
            and~\ref{L: energy_space_localization} show that
            $(s,t)\mapsto F(s+t)$ is positive semidefinite.
            Lemma~\ref{L: exponential_convex_representation} gives a unique
            finite positive measure $\mu$ on $\R$ with
            $F(t)=\int_\R e^{-\beta t}\,d\mu(\beta)$ for $t\geq0$, and
            Lemma~\ref{L: canonical_laplace_support} places $\mu$ on
            $(0,\infty)$ and gives the mixture representation. A Laplace
            transform of a positive measure on $(0,\infty)$ is completely
            monotone, so $F$ is completely monotone.
 
            Conversely, suppose that $F$ is completely monotone. The
            Bernstein--Widder theorem \cite[Chapter~IV, Section~12]{Widder}
            gives a unique positive measure $\mu$ on $[0,\infty)$ such
            that $F(t)=\int e^{-\beta t}\,d\mu(\beta)$ for $t>0$.
            Monotone convergence and continuity of $F$ at zero give
            $\mu([0,\infty))=F(0)<\infty$, so the representation also holds
            at $t=0$. Lemma~\ref{L: canonical_laplace_support} shows that
            $\mu(\{0\})=0$, that $\nu_d$ is a probability measure, and
            that $Q_d=\int P_\beta^{\otimes d}\,d\nu_d(\beta)$.
            Conditionally iid sampling from $P_A$, where $A\sim\nu_d$,
            therefore gives an infinite exchangeable extension.
 
            It remains to decide uniqueness. The family $(P_\beta)_{\beta>0}$
            is in the setting of Corollary~\ref{C: exponential_rigidity}
            with $d=1$ and $Y=(0,\infty)$. Since $\supp\rho=[0,\infty)$
            contains $\N_0$, condition~\eqref{E: exponential_MS} holds with
            $t_0=0$ and $\lambda_n=n$. If $d\geq3$, the three-dimensional
            marginal of $Q_d$ is $\cM_3(\fG_{\nu_d})$, and
            Corollary~\ref{C: exponential_rigidity} identifies the directing
            law of every exchangeable extension as $\fG_{\nu_d}$.
 
            If $d=2$ and $\nu_2$ is a point mass, uniqueness follows from
            Lemma~\ref{L: product_law_rigidity}. Otherwise,
            Proposition~\ref{P: exponential_two_nonrigidity} gives a
            different directing law with the same pair law; full support
            of $\rho$ ensures that $T$ is not $m$-almost everywhere
            constant. Finally, $Z(\beta)^2$ is finite and strictly positive
            for $\beta>0$, so $\nu_2$ is a point mass if and only if
            $\mu=c\delta_b$ for some $c,b>0$. By uniqueness of the Laplace
            representation, this is equivalent to $F(t)=ce^{-bt}$.
        \end{proof}
 
    \subsection{\texorpdfstring{$\ell_q$-norm exponential laws}{Norm-exponential laws}}
    \label{SS: lq_laws}
 
        We next apply the criterion to a family with a phase transition at
        $q=1$. We prove Corollary~\ref{C: lq_phase_diagram} for the laws
        defined in \eqref{E: lq_density_intro}. These laws belong to the
        broader literature on norm symmetry and extendibility surveyed
        in \cite[Remark~3.12]{mai20extendibility}.

        Both thresholds are sharp. For $d=1$ the norm reduces to
        $\abs{x}$, so $Q_{q,1,\lambda}$ is the Laplace law for every $q$.
        A single coordinate never determines its extension: iid sampling
        and the constant sequence $X_i=X_1$ are distinct infinite
        exchangeable extensions, because the Laplace law is diffuse.
 
        \begin{proof}[Proof of Corollary~\ref{C: lq_phase_diagram}]
            Take $E=\R$, $m(dx)=dx$, and $T(x)=\abs{x}^q$. Then $T_*m$ has
            a strictly positive density on $(0,\infty)$, so its support is
            $[0,\infty)$, and
            \begin{equation}\label{E: lq_partition_function}
                Z_q(\beta)=\int_\R e^{-\beta\abs{x}^q}\,dx
                =2\Gamma\pa{1+\frac1q}\beta^{-1/q},
                \quad \beta>0.
            \end{equation}
            The integral is infinite for $\beta\leq0$, so
            Equation~\eqref{E: critical_parameter_intro} holds. Since
            $\norm{x}_q=(\sum_jT(x_j))^{1/q}$, the energy profile of
            Equation~\eqref{E: lq_density_intro} is
            \begin{equation*}
                F(s)=c_{q,d,\lambda}e^{-\lambda s^\alpha},
                \quad \alpha:=\frac1q.
            \end{equation*}
            Theorem~\ref{T: energy_radial_extendibility} reduces each
            assertion to a property of $F$.
 
            (a) If $0<q<1$, then $\alpha>1$, and
            \begin{equation*}
                \frac{d^2}{ds^2}e^{-\lambda s^\alpha}
                =e^{-\lambda s^\alpha}\lambda\alpha s^{\alpha-2}
                    \bra{\lambda\alpha s^\alpha-(\alpha-1)}
            \end{equation*}
            is negative for all sufficiently small $s>0$. Hence $F$ is not
            completely monotone, and no infinite exchangeable extension
            exists.
 
            (b) If $q=1$, then $F$ is a single exponential, and the density
            factorizes into iid Laplace densities with rate $\lambda$.
            The theorem gives uniqueness for every $d\geq2$.
 
            (c) If $q>1$, then $0<\alpha<1$, and there is a nondegenerate
            positive $\alpha$-stable random variable $B_{\alpha,\lambda}$
            with
            \begin{equation}\label{E: lq_stable_laplace}
                \EE\bra{e^{-sB_{\alpha,\lambda}}}
                =e^{-\lambda s^\alpha},\quad s\geq0;
            \end{equation}
            see \cite[Chapter~XIII, Section~6]{Feller}. Thus $F$ is
            completely monotone, with representing measure $\mu$ equal to
            $c_{q,d,\lambda}$ times the law $\xi_{\alpha,\lambda}$ of
            $B_{\alpha,\lambda}$. The integral representation of a negative
            power and Tonelli's theorem give
            \begin{equation*}
                \int_{(0,\infty)}u^{-d/q}\,d\xi_{\alpha,\lambda}(u)
                =\frac{1}{\Gamma(d/q)}\int_0^\infty
                    t^{d/q-1}e^{-\lambda t^{1/q}}\,dt
                =\frac{q\,\Gamma(d)}{\Gamma(d/q)\lambda^d}<\infty.
            \end{equation*}
            Equations~\eqref{E: radial_mixing_law}
            and~\eqref{E: lq_partition_function} therefore identify the
            mixing law as the polynomial tilt
            \begin{equation}\label{E: lq_mixing_law}
                d\nu_d(\beta)
                =\frac{\beta^{-d/q}}
                    {\displaystyle\int_{(0,\infty)}
                        u^{-d/q}\,d\xi_{\alpha,\lambda}(u)}
                    \,d\xi_{\alpha,\lambda}(\beta).
            \end{equation}
            This law is not a point mass, since the stable law is
            nondegenerate and the tilting weight is strictly positive.
            The theorem gives an extension that is unique for $d\geq3$
            and not unique for $d=2$.
        \end{proof}
 
        For $q>1$, the canonical extension is a polynomially tilted
        positive $1/q$-stable mixture of iid centered exponential-power
        laws. Mai \cite[Corollary~3.5]{mai20norm} obtained the
        corresponding conditionally iid representation for $\ell_q$-norm
        exponential densities on the positive orthant when $q\geq1$, with
        the polynomially tilted stable mixing law displayed there. Adding
        iid Rademacher signs, independent of the positive coordinates and
        the mixing variable, gives the centered version in
        Equation~\eqref{E: lq_density_intro}.
        Corollary~\ref{C: lq_phase_diagram} also rules out extendibility
        for $0<q<1$ and determines the exact uniqueness threshold among
        arbitrary infinite exchangeable extensions.
 
        The mixing law in Equation~\eqref{E: lq_mixing_law} depends on the
        dimension $d$ of the prescribed marginal, because the laws
        $Q_{q,d,\lambda}$ for different $d$ are not marginals of one
        another. Once the marginal is fixed, its canonical extension uses
        the same $\nu_d$ at every sample size.

\section*{Acknowledgments}

Large language models were used during preparation of the manuscript to assist with brainstorming, prose revision, proofreading, and identification of candidate literature. Every mathematical argument and bibliographic claim was checked and edited by the author, who takes full responsibility for the contents.

\appendix

\section{Finite-mixture identifiability from moment operators}
\label{A: finite_mixture_operators}

    This appendix recovers established finite-mixture bounds through the
    moment-operator framework of Section~\ref{S: three_rigidity_proof}.
    The observation counts are those of
    \cite[Theorems~4.1, 4.3, and 4.5--4.7]{VandermeulenScott}, whose
    approach is also operator-theoretic. We distinguish identification
    among competitors with a bounded number of components from rigidity
    against arbitrary directing laws. Lemma~\ref{L: finite_competitor_reduction}
    explains why, for a finite target and order at least two, allowing
    arbitrary competitors adds no further obstruction to uniqueness
    among all finite competitors.

    Throughout, $E$ is Polish and
    \begin{equation*}
        \Lambda_0=\sum_{i=1}^m w_i\delta_{P_i},
        \quad w_i>0,\quad \sum_{i=1}^m w_i=1,
    \end{equation*}
    with distinct $P_i\in\cP(E)$. Competing finite laws are likewise
    written with distinct atoms and positive weights. We regard $w$ as
    a probability measure on $\{1,\ldots,m\}$ and put $K_h(i)=P_i(h)$.
    All test functions are bounded, real-valued, and Borel measurable.
    Equality at any moment order implies the lower-order identities by
    marginalization.

    \subsection{Linearly independent components}

        Linear independence of $P_1,\ldots,P_m$ as signed measures is
        exactly completeness~\eqref{E: completeness} for $(w,P)$, since
        $Q_w[f]=\sum_i w_if(i)P_i$ and all weights are positive.

        \begin{proposition}\label{P: finite_independent_bounds}
            Suppose that $P_1,\ldots,P_m$ are linearly independent.
            \begin{enumerate}
                \item[(a)] If $\Lambda$ has at most $m$ atoms and
                $\cM_3(\Lambda)=\cM_3(\Lambda_0)$, then
                $\Lambda=\Lambda_0$.
                \item[(b)] The law $\Lambda_0$ is $4$-rigid.
            \end{enumerate}
        \end{proposition}
        \begin{proof}
            (a) Write $\Lambda=\sum_{j=1}^{\ell}v_j\delta_{Q_j}$,
            where $\ell\leq m$, and realize $\eta(j)=Q_j$ on
            $\{1,\ldots,\ell\}$ with probability measure $v$.
            Lemma~\ref{L: moment_isometry}(a,b) gives an isometry
            $W:L^2(w)\to L^2(v)$ and
            \begin{equation*}
                W^*M^v_{\eta(h)}W=M^w_{K_h}.
            \end{equation*}
            Thus $m\leq\ell\leq m$, so $W$ is onto and $WW^*=I$.
            Multiplication on the left by $W$ gives exact intertwining.
            No independence assumption on the $Q_j$ is needed.

            The map $h\mapsto(P_1(h),\ldots,P_m(h))$ is onto $\R^m$:
            otherwise a nonzero linear functional annihilating its
            range would give a linear dependence among the $P_i$.
            Choose dual tests $h_i$ with $P_j(h_i)=\1_{\{i=j\}}$, and put
            $e_i=\1_{\{i\}}$ and $Z_i=We_i=\eta(h_i)$.
            Intertwining applied to $h_i$ and $e_i$ gives $Z_i^2=Z_i$,
            while $\sum_i Z_i=W\1=1$. Hence $Z_i=\1_{\{A=i\}}$ for a
            random index $A$, and
            \begin{equation*}
                \PP(A=i)=\EE[Z_i]
                =\ip{We_i}{W\1}_{L^2(v)}=w_i.
            \end{equation*}
            For every $h$, $\eta(h)=WK_h=\sum_iP_i(h)Z_i=P_A(h)$.
            Applying this simultaneously to the indicators of a
            countable algebra generating $\cB(E)$ gives
            $\eta=P_A$ almost surely, so $\Lambda=\Lambda_0$.

            (b) This is Theorem~\ref{T: completeness_rigidity}, applied
            to the finite complete kernel $(w,P)$.
        \end{proof}

        \begin{remark}\label{R: joint_irreducibility}
            For finitely supported $\nu$,
            \cite[Theorem~4.7]{VandermeulenScott} show that if the
            components are jointly irreducible, then $\fG_\nu$ is
            determined by its second moment measure among all random
            probability measures taking finitely many values. By Lemma~\ref{L: finite_competitor_reduction}, this is equivalent to unrestricted $2$-rigidity. Joint irreducibility is strictly stronger than
            \eqref{E: strong_localization}, under which Theorem~\ref{T: general_rigidity} determines $\fG_\nu$ by its third moment measure among all random probability measures on $E$. Under this weaker assumption the third moment
            cannot in general be replaced by the second; see the end of this remark.

            Let $\nu=\sum_{j=1}^N w_j\delta_{y_j}$ with $w_j>0$ and distinct
            $y_j$, and write $P_j:=P_{y_j}$. The components are
            \emph{jointly irreducible} if every probability measure of the
            form $\sum_jc_jP_j$ has $c_j\geq0$ for all $j$.
            Joint irreducibility implies linear independence: if
            $\sum_jc_jP_j=0$ with some $c_j\neq0$, then
            $P_1+t\sum_jc_jP_j=P_1$ for every real $t$, and a suitable
            choice of $t$ makes some coefficient negative. It also implies
            \eqref{E: strong_localization} for every Borel $B\subseteq Y$,
            hence for any countable generating $\pi$-system. Suppose
            $0\leq Q_\nu[f]\leq Q_\nu[\1_B]$. If $Q_\nu[f]\neq0$, joint
            irreducibility applied to its normalization gives
            $f(y_j)\geq0$ for all $j$. If $Q_\nu[\1_B]-Q_\nu[f]\neq0$, it
            gives $f(y_j)\leq0$ for $y_j\notin B$. If either measure
            vanishes, linear independence gives the same conclusions.
            Hence $f(y_j)=0$ for $y_j\notin B$.

            The converse fails. For the uniform kernel of
            Example~\ref{Ex: uniform_localization} with $y_1=1$ and $y_2=2$,
            the measure $2P_2-P_1$ is the uniform law on $(1,2)$, so
            $P_1,P_2$ are not jointly irreducible. Nevertheless $(\nu,P)$
            is strongly localized for every $\nu$. Here order two does not
            suffice. For $\nu=\frac12\delta_1+\frac12\delta_2$, the random
            measure equal to $\frac34P_1+\frac14P_2$ with probability
            $\frac45$ and to $\frac32P_2-\frac12P_1$ with probability
            $\frac15$ has the same second moment measure as
            $\fG_\nu$ but a different law.
        \end{remark}

        \begin{corollary}\label{C: finite_joint_irreducibility}
            If $P_1,\ldots,P_m$ are jointly irreducible, then
            $\Lambda_0$ is $2$-rigid.
        \end{corollary}
        \begin{proof}
            Joint irreducibility means that whenever $\sum_i c_iP_i$
            is a probability measure, all $c_i$ are nonnegative. It
            implies linear independence
            (Remark~\ref{R: joint_irreducibility}). Let $\eta$ have law
            $\Lambda$ with $\cM_2(\Lambda)=\cM_2(\Lambda_0)$.
            Use Lemma~\ref{L: moment_isometry}(a) and the dual tests
            above to put $Z_i=We_i=\eta(h_i)$. Since
            $K_h=\sum_iP_i(h)e_i$, the same countable-algebra argument
            gives, almost surely,
            \begin{equation*}
                \eta=\sum_i Z_iP_i,\quad \sum_i Z_i=1,
                \quad \EE[Z_i]=w_i,\quad
                \EE[Z_iZ_j]=w_i\1_{\{i=j\}}.
            \end{equation*}
            Joint irreducibility gives $Z_i\geq0$, hence $0\leq Z_i\leq1$.
            Since $\EE[Z_i(1-Z_i)]=w_i-w_i=0$, the $Z_i$ are indicators.
            Thus $\eta=P_A$ with $\PP(A=i)=w_i$.
        \end{proof}

        The three-observation argument in
        Proposition~\ref{P: finite_independent_bounds}(a) uses a bound
        on the competing order to obtain surjectivity. Fourth-moment
        matching gives exact intertwining without surjectivity, whereas
        weak localization in Theorem~\ref{T: general_rigidity} forces
        preservation of indicators without either surjectivity or a
        fourth-moment identity.

    \subsection{Block observations}

        Linear independence after taking product measures gives the
        following extension. The blocks construct the isometry; the
        multiplication operators still use single observations.

        \begin{proposition}\label{P: finite_block_bounds}
            Suppose that $P_1^{\otimes r},\ldots,P_m^{\otimes r}$ are
            linearly independent for some integer $r\geq1$.
            \begin{enumerate}
                \item[(a)] If $\Lambda$ has at most $m$ atoms and
                $\cM_{2r+1}(\Lambda)=\cM_{2r+1}(\Lambda_0)$, then
                $\Lambda=\Lambda_0$.
                \item[(b)] The law $\Lambda_0$ is $(2r+2)$-rigid.
            \end{enumerate}
        \end{proposition}
        \begin{proof}
            Let $\eta$ have the competing law $\Lambda$. For a test
            $H:E^r\to\R$, put $K_H^{(r)}(i)=P_i^{\otimes r}(H)$.
            These functions span $L^2(w)$ by linear independence.
            Testing the order-$2r$ identity against $H\otimes H'$ gives
            \begin{equation*}
                \EE[\eta^{\otimes r}(H)\eta^{\otimes r}(H')]
                =\sum_i w_iK_H^{(r)}(i)K_{H'}^{(r)}(i).
            \end{equation*}
            Consequently $W_rK_H^{(r)}=\eta^{\otimes r}(H)$ defines
            an isometry $W_r:L^2(w)\to L^2(\PP)$ with $W_r\1=\1$.
            For $h:E\to\R$, set $B_h=M^\PP_{\eta(h)}$ and
            $D_h=M^w_{K_h}$. Testing against $H\otimes H'\otimes h$
            at order $2r+1$, and using the spanning property, gives
            \begin{equation*}
                W_r^*B_hW_r=D_h.
            \end{equation*}

            In (a), realize $\eta$ canonically on the $\ell\leq m$
            atoms of $\Lambda$. The isometry gives $m\leq\ell$, so it
            is onto, and $B_hW_r=W_rD_h$ follows as in
            Proposition~\ref{P: finite_independent_bounds}(a).
            In (b), testing against $H\otimes H'\otimes h\otimes h$
            at order $2r+2$ gives $W_r^*B_h^2W_r=D_h^2$.
            The two operator identities and the isometry imply, for
            every $g\in L^2(w)$,
            \begin{equation*}
                \norm{B_hW_rg-W_rD_hg}^2
                =\norm{D_hg}^2-2\norm{D_hg}^2+\norm{D_hg}^2=0.
            \end{equation*}
            Thus exact intertwining holds in both cases.

            Choose $h_*$ so that $a_i=P_i(h_*)$ are pairwise distinct.
            Such a test is a linear combination of finitely many tests
            distinguishing the pairs $P_i,P_j$, with coefficients
            outside finitely many proper hyperplanes. Set $Z_i=W_re_i$.
            Intertwining and $W_r\1=\1$ give
            \begin{equation*}
                (\eta(h_*)-a_i)Z_i=0,\quad \sum_i Z_i=1.
            \end{equation*}
            These finitely many identities hold on one event of
            probability one. There, at most one $Z_i$ is nonzero, so
            $Z_i=\1_{\{A=i\}}$ for a measurable random index $A$.
            As before, $\EE[Z_i]=\ip{W_re_i}{W_r\1}=w_i$.
            Intertwining also gives $\eta(h)Z_i=P_i(h)Z_i$.
            Intersecting the events for all $i$ and for the indicators
            of a countable generating algebra yields $\eta=P_A$
            almost surely. Thus $\Lambda=\Lambda_0$.
        \end{proof}

        \begin{corollary}\label{C: finite_general_bounds}
            Every target with $m$ distinct components is identified
            by $2m-1$ observations among competitors with at most $m$
            atoms, and is $2m$-rigid.
        \end{corollary}
        \begin{proof}
            Suppose $m\geq2$. For $i\ne j$, distinctness gives a test
            $h_{ij}$ with $P_j(h_{ij})=0$ and $P_i(h_{ij})=1$, by
            subtracting a constant from a distinguishing test and
            rescaling. Then
            \begin{equation*}
                H_i=\bigotimes_{j\ne i}h_{ij},
                \quad P_\ell^{\otimes(m-1)}(H_i)=\1_{\{i=\ell\}}.
            \end{equation*}
            Hence the $(m-1)$st tensor powers are linearly independent,
            and Proposition~\ref{P: finite_block_bounds} applies with
            $r=m-1$. For $m=1$, one observation identifies the sole
            component among one-atom competitors, and two give
            unrestricted rigidity by Lemma~\ref{L: product_law_rigidity}.
        \end{proof}

        These are recoveries of the basic bounds, not uniform
        improvements. Vandermeulen and Saitenmacher
        \cite{VandermeulenSaitenmacher} obtain further bounds under
        partial linear-independence assumptions; the block argument
        above does not in general recover all those bounds.

    \subsection{Reduction to finite competitors}

        In \cite[Definitions~3.2 and 3.3]{VandermeulenScott},
        identifiability restricts the competing order to the target
        order, whereas determinedness allows any finite competing
        order. For finite targets, the latter already implies
        unrestricted rigidity at orders at least two.

        \begin{lemma}\label{L: finite_competitor_reduction}
            Let $\Lambda_0$ be finitely supported and $k\geq2$. Then
            $\Lambda_0$ is $k$-rigid if and only if it is the unique
            finitely supported law with moment measure
            $\cM_k(\Lambda_0)$.
        \end{lemma}
        \begin{proof}
            Only the reverse implication needs proof. Let $\eta$ have
            law $\Lambda\ne\Lambda_0$ with
            $\cM_k(\Lambda)=\cM_k(\Lambda_0)$. We construct a finite
            counterexample. No independence of the target components
            is assumed here.

            Put $V=\operatorname{span}\{P_1,\ldots,P_m\}$, and select
            a basis $R_1,\ldots,R_d$ from these probability measures.
            Choose bounded dual tests $h_j$ with
            $R_i(h_j)=\1_{\{i=j\}}$, and define
            \begin{equation*}
                T\mu=\sum_{j=1}^d\mu(h_j)R_j,\quad
                g_h=h-\sum_{j=1}^dR_j(h)h_j.
            \end{equation*}
            Since $TP_i=P_i$, every $P_i$ annihilates $g_h$.
            The second-moment identity therefore gives
            $\EE[\eta(g_h)^2]=0$. Take $h$ to be the indicators of a
            countable algebra generating $\cB(E)$ and intersect the
            resulting events. On that event $\eta$ and $T\eta$ are
            finite signed measures agreeing on the algebra, hence
            $\eta=T\eta\in V$.

            The set $C=V\cap\cP(E)$ is compact in the finite-dimensional
            total variation topology: it is closed and every element
            has total variation norm one. Identify it with its compact
            coordinate image in $\R^d$, using $x_j=\mu(h_j)$; these
            coordinates are Borel measurable. For laws concentrated on
            $C$, equality of $k$th moment measures is equivalent to
            equality of the finitely many coordinate monomials of
            total degree at most $k$. Testing products of the dual
            functions gives one direction; expanding
            $\mu^{\otimes k}=(\sum_jx_jR_j)^{\otimes k}$ gives the other.

            If $\Lambda$ is concentrated on $\{P_1,\ldots,P_m\}$, it
            is already finite. Otherwise, let $x_i$ be the coordinate
            vector of $P_i$ and put
            \begin{equation*}
                q(x)=\prod_{i=1}^m\norm{x-x_i}^2.
            \end{equation*}
            Its expectation is zero under $\Lambda_0$ and strictly
            positive under $\Lambda$. Let $\Phi$ collect all coordinate
            monomials of degree at most $k$, together with $q$.
            By Carath\'eodory's theorem, the convex hull of the compact
            set $\Phi(C)$ is compact. Thus the mean of $\Phi$ under
            $\Lambda$ belongs to this convex hull and is a finite
            convex combination of values $\Phi(x)$, $x\in C$.
            The corresponding finite law $\Lambda'$ has
            $\cM_k(\Lambda')=\cM_k(\Lambda_0)$ and positive
            $q$-expectation, so $\Lambda'\ne\Lambda_0$.
            The polynomial $q$ is only used to preserve nonuniqueness
            in this construction; it is not an additional observation.
        \end{proof}

\bibliographystyle{plainnat}
\bibliography{main}

\end{document}